\documentclass[12pt]{amsart}
\usepackage{euscript}
\usepackage{graphics}
\usepackage{calc,tikz}
\usepackage[matrix,arrow]{xy}
\usepackage{hyperref}

\numberwithin{equation}{section}
\renewcommand{\subsection}[1]{\refstepcounter{subsection}\hspace{-\parindent}{\bf (\arabic{section}\alph{subsection}) #1}.}

\newcommand{\iso}{\cong}
\newcommand{\smooth}{C^\infty}
\newcommand{\htp}{\simeq}
\newcommand{\half}{{\textstyle\frac{1}{2}}}

\newcommand{\bC}{{\mathbb C}}

\newcommand{\bF}{{\mathbb F}}

\newcommand{\bK}{{\mathbb K}}

\newcommand{\bQ}{{\mathbb Q}}
\newcommand{\bR}{{\mathbb R}}

\newcommand{\bZ}{{\mathbb Z}}

\newcommand{\scrA}{\EuScript A}

\newcommand{\scrE}{\EuScript E}

\newcommand{\scrH}{\EuScript H}

\newcommand{\scrJ}{\EuScript J}

\newcommand{\scrL}{\EuScript L}
\newcommand{\scrM}{\EuScript M}

\newcommand{\scrP}{\EuScript P}
\newcommand{\scrQ}{\EuScript Q}

\newcommand{\scrS}{\EuScript S}

\newcommand{\noarrow}[1][]{
    \mathrel{%
        \tikz [x=0.65cm, y=\heightof{\strut}, line width=.1ex, ->, baseline, #1] 
        \draw (0,0.3) -- (1,0.3) ;
   }%
}%
\newcommand{\yesarrow}[1][]{% #1 = optional draw paramaters (applies to both arrow and dot)
    \mathrel{%
        \tikz [x=0.65cm, y=\heightof{\strut}, line width=.1ex, ->, baseline, #1] 
            \draw (0,0.3) -- (1,0.3) 
            node [pos=0.5,shape=circle, fill=black, draw, inner sep=1pt, #1] {};
    }%
}%

\theoremstyle{plain}
\newtheorem{lemma}{Lemma}
\numberwithin{lemma}{section}
\newtheorem{definition}[lemma]{Definition}
\newtheorem{example}[lemma]{Example}

\newtheorem{addendum}[lemma]{Addendum}
\newtheorem{assumption}[lemma]{Assumption}
\newtheorem{application}[lemma]{Application}
\newtheorem{proposition}[lemma]{Proposition}
\newtheorem{theorem}[lemma]{Theorem}
\newtheorem{corollary}[lemma]{Corollary}
\newtheorem{remark}[lemma]{Remark}
\newtheorem{setup}[lemma]{Setup}
\newtheorem{data}[lemma]{Data}

\author{Paul Seidel}
\title[Pair-of-pants product]{The equivariant pair-of-pants product\\ in fixed point Floer cohomology}
\begin{document}

\maketitle
\begin{abstract}
We use equivariant methods and product structures to derive a relation between the fixed point Floer cohomology of an exact symplectic automorphism and that of its square.
\end{abstract}

\section{Introduction}

This paper concerns the Floer cohomology of symplectic automorphisms, and its behaviour under iterations: more specifically, when passing to the square of a given automorphism (one expects parallel results for odd prime powers, but they are beyond our scope here). The concrete situation is as follows. Let $\phi$ be an exact symplectic automorphism of a Liouville domain $M$ (there are some additional conditions on $\phi$, see Setup \ref{th:setup-1} for details).  The Floer cohomology $\mathit{HF}^*(\phi)$ (defined in \cite{dostoglou-salamon93}, generalizing the Hamiltonian case \cite{floer88}) is a $\bZ/2$-graded $\bK$-vector space. Here and throughout the paper, $\bK = \bF_2$ is the field with two elements. The Floer cohomology of $\phi^2$ carries additional structure, namely an action of $\bZ/2$. Denote the invariant part by $\mathit{HF}^*(\phi^2)^{\bZ/2}$. From the viewpoint of applications, our most significant result is the following Smith-type inequality (the name refers to a topological result reproduced as \eqref{eq:classical-smith} below, see \cite[Chapter III, 4.3]{borel60}):

\begin{corollary} \label{th:floer-smith}
There is an inequality of total dimensions,
\begin{equation} \label{eq:quantum-smith}
 \mathrm{dim}\, \mathit{HF}^*(\phi^2)^{\bZ/2} \geq \mathrm{dim}\, \mathit{HF}^*(\phi).
\end{equation}
\end{corollary}

This is not entirely new: under additional topological restrictions (stated below as Assumption \ref{th:stably-trivial}), it has been previously proved by Hendricks \cite{hendricks14}. As in \cite{hendricks14}, the proof involves an equivariant form of Floer cohomology, written as $\mathit{HF}^*_{\mathit{eq}}(\phi^2)$. This is a finitely generated $\bZ/2$-graded module over $\bK[[h]]$, the ring of formal power series in one variable $h$ (the variable has degree $1$). The information encoded in this equivariant theory can be viewed as a refinement of the previously mentioned $\bZ/2$-action. What we obtain is a description of equivariant Floer cohomology after inverting $h$, which means after tensoring with the ring $\bK((h))$ of Laurent series:

\begin{corollary} \label{th:localized-isomorphism}
There is an isomorphism of ungraded $\bK((h))$-modules,
\begin{equation} \label{eq:u-isomorphism}
\mathit{HF}^*(\phi)((h)) = \mathit{HF}^*(\phi) \otimes \bK((h)) \iso \mathit{HF}^*_{\mathit{eq}}(\phi^2) \otimes_{\bK[[h]]} \bK((h)).
\end{equation}
\end{corollary}

Corollary \ref{th:floer-smith} follows from this by purely algebraic arguments (the same step appears in \cite{seidel-smith10, hendricks14}, as well as in ordinary equivariant cohomology \cite[Chapter IV.4]{borel60}).

Naively, \eqref{eq:quantum-smith} may not be surprising: if one thinks of Floer cohomology as a measure of fixed points, $\phi^2$ clearly has more of them than $\phi$. In the same intuitive spirit (and with the localization theorem for equivariant cohomology in mind, which we will recall as Theorem \ref{th:localization} below), one can think of tensoring with $\bK((h))$ as throwing away the fixed points of $\phi^2$ which are not fixed points of $\phi$, leading to \eqref{eq:u-isomorphism}. Indeed, in a sense, the proofs ultimately reduce to such very basic considerations. Before one can get to that point, however, a map has to be defined which allows one to compare the two sides of \eqref{eq:u-isomorphism}. It is at this point that our approach diverges from that in \cite{hendricks14}. We construct an equivariant refinement of the pair-of-pants product \cite{schwarz95,salamon99}, which is a homomorphism of $\bZ/2$-graded $\bK[[h]]$-modules,
\begin{equation} \label{eq:equi-pants}
H^*(\bZ/2; \mathit{CF}^*(\phi) \otimes \mathit{CF}^*(\phi)) \longrightarrow \mathit{HF}^*_{\mathit{eq}}(\phi^2).
\end{equation}
Here $\mathit{CF}^*(\phi)$ is the chain complex underlying $\mathit{HF}^*(\phi)$. We take its tensor product with itself (as a chain complex), equip it with the involution that exchanges the two factors, and consider the associated group cohomology $H^*(\bZ/2; \mathit{CF}^*(\phi) \otimes \mathit{CF}^*(\phi))$. We will see, as part of the elementary formalism of group cohomology, that this depends only on $\mathit{HF}^*(\phi)$. Our main theorem is:

\begin{theorem} \label{th:main}
The equivariant pair-of-pants product \eqref{eq:equi-pants} becomes an isomorphism after tensoring with $\bK((h))$ on both sides.
\end{theorem}

Corollary \ref{th:localized-isomorphism} is a purely algebraic consequence of this statement. Note that in principle, the map \eqref{eq:equi-pants} contains additional information, which is lost when taking the tensor product with $\bK((h))$. 

\begin{addendum} \label{th:warning}
The construction of $\mathit{HF}^*(\phi)$ assumes nondegeneracy of fixed points, and involves additional choices of almost complex structures. Ultimately, one uses continuation maps \cite{salamon-zehnder92} to show that Floer cohomology is independent of those choices up to canonical isomorphism, and also to extend the definition to the degenerate case. 

Similarly, the construction of $\mathit{HF}^*_{\mathit{eq}}(\phi^2)$ and of \eqref{eq:equi-pants} requires nondegeneracy of the fixed points of $\phi^2$, and involves further auxiliary choices (of almost complex structures and, in the case of the product, Hamiltonian functions which serve as inhomogeneous terms for the $\bar\partial$-equations). Even though this should not affect the outcome, in the same sense as before, we will not prove that statement here. 

Now, the proof of Theorem \ref{th:main} makes some specific requirements: in addition to the nondegeneracy of fixed points of $\phi^2$, there is an additional condition on the action functional (see Setup \ref{th:setup-4}; this can be achieved by a small perturbation). One then needs to choose the auxiliary data (specifically, the inhomogeneous terms) that define the equivariant pair-of-pants product to be sufficiently small. The precise statement should therefore be that, for this particular class of $\phi$, one can define \eqref{eq:equi-pants} in such a way that it becomes an isomorphism after tensoring with $\bK((h))$. The same applies to Corollary \ref{th:localized-isomorphism}. However, Corollary \ref{th:floer-smith} does not require any such additional language (because the statement only concerns ordinary Floer cohomology groups).
\end{addendum}

The structure of the paper is as follows. Section \ref{sec:context}, a kind of extended introduction, provides background and context for our constructions. In particular, it describes the algebraic arguments that tie together the statements made above; explains the motivation from classical equivariant cohomology; and discusses some applications. Section \ref{sec:morse} constructs certain auxiliary Morse-theoretic moduli spaces. Using those plus rather standard Floer-theoretic machinery, we construct equivariant Floer cohomology and \eqref{eq:equi-pants}, in Section \ref{sec:con}. Section \ref{sec:linear} contains further background material, this time from symplectic linear algebra. This is used in Section \ref{sec:local} to prove Theorem \ref{th:main}. Finally, Section \ref{sec:monotone} takes a brief look at some of the new phenomena that one can expect if the exactness assumptions are dropped. 

{\em Acknowledgments.} I am indebted to Kristen Hendricks, Graeme Segal, David Treumann and Jingyu Zhao for helpful explanations. This work was partially supported by NSF through grant DMS-1005288; by the Simons Foundation, through a Simons Investigator grant; and by a Fellowship at the Radcliffe Institute for Advanced Study. I would also like to thank the IBS Center for Geometry and Physics (Pohang), where part of the paper was written, for its hospitality. 
\newpage
\section{Context\label{sec:context}}

Since the constructions in this paper are modelled on ones in equivariant cohomology, we include a review of that theory (specialized to the group $\bZ/2$), emphasizing its algebraic aspects. After that, we outline the structure of the Floer-theoretic analogue, and in particular, explain how one goes from Theorem \ref{th:main} to Corollaries \ref{th:floer-smith} and \ref{th:localized-isomorphism}. We will then discuss some sample applications. Finally, returning to the general picture, we consider how our approach to relating the Floer cohomology of $\phi$ and $\phi^2$ compares to that in \cite{hendricks14}, as well as to the purely algebraic theory in \cite{lipshitz-treumann12}. Surprisingly, the attempt to combine the picture here with that in \cite{hendricks14} naturally seems to involve another theory, namely, the Floer homotopy type proposed in \cite{cohen-jones-segal95}.

\subsection{Algebra background}
Let $V$ be a vector space over $\bK = \bF_2$, with a linear action of the group $\bZ/2$, or in other words, an involution $\iota: V \rightarrow V$. The associated group cochain complex is 
\begin{equation} \label{eq:c-complex}
C^*(\bZ/2;V) = V[[h]], \quad
d_C = h(\mathit{id} + \iota),
\end{equation}
where $h$ is a formal variable of degree $1$. Its cohomology, called group cohomology with coefficients in $V$ and denoted by $H^*(\bZ/2;V)$, is a $\bZ$-graded module over $\bK[[h]]$. There is also a version where one inverts $h$, whose cohomology is called Tate cohomology:
\begin{align}
& \hat{C}^*(\bZ/2;V) = C^*(\bZ/2;V) \otimes_{\bK[[h]]} \bK((h)) = V((h)), \\
& \hat{H}^*(\bZ/2;V) = H^*(\hat{C}^*(\bZ/2;V)) \iso H^*(\bZ/2;V) \otimes_{\bK[[h]]} \bK((h)). 
\label{eq:tate-relation}
\end{align}
Both versions are functorial in $V$ (under $\bZ/2$-equivariant linear maps).

\begin{example} \label{th:tate-0}
Let $V$ be a vector space with $\bZ/2$-action, which is equivariantly isomorphic to a direct sum of copies of the standard representation $\bK[\bZ/2]$. In simpler terms, this means that $V$ has a basis freely acted on by $\bZ/2$. Direct computation shows that then, $\hat{H}^*(\bZ/2;V) = 0$.
\end{example}

\begin{remark} \label{th:history-lesson}
Group cohomology, which applies to representations of arbitrary groups, was defined in \cite{eilenberg-maclane43}. The Tate version, for finite groups, was introduced in \cite{tate}. However, the general relation between the two theories takes on a more complicated form than \eqref{eq:tate-relation}. Example \ref{th:tate-0} is a special case of the vanishing of Tate cohomology with coefficients in a free module (see e.g.\ \cite[p.~136]{brown}). 
\end{remark}

The definitions made above generalize to the situation where $V$ is a ($\bZ$-graded or $\bZ/2$-graded) chain complex of vector spaces acted on by $\bZ/2$, in which case the differential on $C^*(\bZ/2;V)$ becomes $d_C = d_V + h(\mathit{id}+\iota)$. Its cohomology $H^*(\bZ/2;V)$ is again a ($\bZ$-graded or $\bZ/2$-graded) $\bK[[h]]$-module. We summarize some of its basic properties:

\begin{lemma} \label{th:3-lemma}
(i) If $H^*(V) = 0$, then $H^*(\bZ/2;V) = 0$.

(ii) If $H^*(V)$ is of finite (total) dimension, then $H^*(\bZ/2;V)$ is a finitely generated $\bK[[h]]$-module.

(iii) Suppose that $V_1$ and $V_2$ are chain complexes with $\bZ/2$-actions, and that we have a chain map $V_1 \rightarrow V_2$ which is $\bZ/2$-equivariant, and which induces an isomorphism $H^*(V_1) \rightarrow H^*(V_2)$. Then the associated map $H^*(\bZ/2;V_1) \rightarrow H^*(\bZ/2;V_2)$ is also an isomorphism.

(iv) Suppose that we have three chain complexes with $\bZ/2$-actions, and equivariant chain maps between them, which form a short exact sequence
\begin{equation}
0 \rightarrow V_1 \longrightarrow V_2 \longrightarrow V_3 \rightarrow 0.
\end{equation}
Then, the associated maps on group cohomology fit into a long exact sequence
\begin{equation} \label{eq:equivariant-les-0}
\cdots \rightarrow H^*(\bZ/2;V_1) \longrightarrow H^*(\bZ/2;V_2) \longrightarrow H^*(\bZ/2;V_3) \longrightarrow H^{*+1}(\bZ/2;V_1) \rightarrow \cdots
\end{equation}
\end{lemma}

\begin{proof}
(i) Take a cocycle $v \in C^*(\bZ/2;V) = V[[h]]$, and write it as $v = v^0 + O(h)$, where $v^0 \in V$ (the notation $O(h)$ means a multiple of $h$, or in other words, an element of $h V[[h]]$). Then, $d_V v^0 = 0$. By assumption, there is a $w^0 \in V$ such that $d_V w^0 = v^0$. One can therefore write $v - d_Cw^0 = hv^1 + O(h^2)$ for some $v^1 \in V$, and then repeat the previous argument to find a $w^1 \in V$ such that $v - d_C(w^0 + hw^1) = O(h^2)$. This iteratively constructs $w = w^0 + h w^1 + \cdots \in V[[h]]$ which satisfies  $d_Cw = v$. 

(ii) The quotient map $C^*(\bZ/2;V) = V[[h]] \rightarrow V[[h]]/hV[[h]] = V$ induces a map
\begin{equation} \label{eq:set-h-to-zero}
H^*(\bZ/2;V) \longrightarrow H^*(V).
\end{equation}
Take cocycles $u_1,\dots,u_r \in C^*(\bZ/2;V)$ whose images in $V$ yield cohomology classes which span the image of \eqref{eq:set-h-to-zero}. Write them as $u_k = u_k^0 + O(h)$. Given any cocycle $v \in C^*(\bZ/2;V)$, write it as $v = v^0 + O(h)$ as well. By assumption, one can find $\gamma_1^0,\dots,\gamma_r^0 \in \bK$ and a $w^0 \in V$ such that $v^0 = \gamma_1^0 u_1^0 + \cdots + \gamma_r^0 u_r^0 + d_V w^0$. One can therefore write 
\begin{equation} \label{eq:s}
v - \gamma_1^0 u_1 - \cdots - \gamma_r^0 u_r - d_C w^0 = hv^1 + O(h^2)
\end{equation}
for some $v^1 \in V$. The expression on either side of \eqref{eq:s} is $h$ times some cocycle in $C^*(\bZ/2;V)$. We can apply the same argument to that cocycle, and then proceed iteratively, which constructs $\gamma_1,\dots,\gamma_r \in \bK[[h]]$ and a $w \in C^*(\bZ/2;V)$ such that $v = \gamma_1 u_1 + \cdots + \gamma_r u_r + d_Cw$.

(iii) can be proved by a similar order-by-order argument, whose details we omit. 

(iv) is obvious, since the complexes $C^*(\bZ/2;V_k)$ themselves form a short exact sequence (inspection of the standard argument shows that the boundary operator is a $\bK[[h]]$-linear map).
\end{proof}

\begin{remark} \label{th:spectral-sequence}
The acyclicity result (i) is an instance of a much more general principle. Namely, take any ($\bZ$-graded or $\bZ/2$-graded) chain complex of vector spaces $(V,d_V)$. Suppose that on $V[[h]]$, we have a $\bK[[h]]$-linear differential of the form $dv = d_Vv + O(h)$. Then, if $(V,d_V)$ is acyclic, the same holds for $(V[[h]],d)$. The proof is the same as in the previously considered special case. Alternatively, one can think in terms of spectral sequences: $(V[[h]],d)$ carries a complete decreasing filtration (by powers of $h$), and the differential on the associated graded space is given by $d_V$ (at each level of the filtration). Under our assumption, the $E_1$ page of the spectral sequence is zero, which implies the acyclicity of $(V[[h]],d)$.

There is a similar generalization of (ii). Abstractly, one should be able think of it as a vanishing result parallel to (i), by working modulo the Serre subcategory of finitely generated $\bK[[h]]$-modules \cite{serre} (but we have not checked the details of this approach; in any case, the proof we have given also works in this more general context).

A similar observation applies to part (iii). Take chain complexes $V_k$ ($k = 1,2$; with no group actions). Suppose that we have differentials $d_k = d_{V_k} + O(h)$ on $V_k[[h]]$. Consider a $\bK[[h]]$-linear chain map $V_1[[h]] \rightarrow V_2[[h]]$. Then, if the $h = 0$ reduction of our map is a quasi-isomorphism $V_1 \rightarrow V_2$, the original map is also a quasi-isomorphism. Abstractly, one can think of this as an application of the spectral sequence comparison theorem (see e.g.\  \cite[Theorem 5.5.11]{weibel}, and note that convergence of the spectral sequence is not necessary for this).
\end{remark}

\begin{remark} \label{th:colimit}
It may also be useful to note one property that group cohomology does not have. Namely, it is not compatible with direct limits. One could cure that deficiency by replacing $V[[h]]$ with $V \otimes \bK[[h]]$ in the definition (recall that $V[[h]]$ is the space of power series with coefficients in $V$, while $V \otimes \bK[[h]]$ is the subspace of those series whose coefficients span a finite-dimensional subspace of $V$). This yields a different theory, but one which no longer satisfies properties (i)--(iii) above (of course, the two theories agree if $V$ is finite-dimensional).
\end{remark}

The Tate version $\hat{H}^*(\bZ/2;V)$ generalizes to the case when $V$ is a chain complex in the same way, and is related to $H^*(\bZ/2;V)$ as in \eqref{eq:tate-relation}. As a consequence, all the properties in Lemma \ref{th:3-lemma} have counterparts for the Tate version. 

\begin{example} \label{th:tate-acyclic}
Let $V$ be a $\bZ$-graded and bounded chain complex with $\bZ/2$-action, such that each $V^i$ has a basis on which $\bZ/2$ acts freely. By truncating it at a fixed degree $j$, one forms a short exact sequence (of complexes with $\bZ/2$-actions)
\begin{equation} \label{eq:trunc}
0 \rightarrow V^{\geq j} \longrightarrow V \longrightarrow V^{\leq j-1} \rightarrow 0.
\end{equation}
Define the ``length'' of $V$ to be the difference between the top and bottom nonzero degrees, plus one. If $V$ has length $>1$, one can arrange that both truncations in \eqref{eq:trunc} have less length. Arguing by induction on length (using the long exact sequence associated to \eqref{eq:trunc}, and Example \ref{th:tate-0} as the base case), one shows that the Tate cohomology of $V$ vanishes.
\end{example}

\begin{remark}
With the generalization to chain complexes, we have moved beyond the first historical framework for group cohomology (as in Remark \ref{th:history-lesson}) to a more abstract viewpoint, where group cohomology is defined as a morphism space in an appropriate derived category (this also works for the Tate version, see e.g.\ \cite{rickard}).
\end{remark}

There is a short exact sequence of complexes
\begin{equation}
0 \rightarrow C^{*-1}(\bZ/2;V) \stackrel{h}{\longrightarrow} C^*(\bZ/2;V) \longrightarrow V \rightarrow 0,
\end{equation}
which induces a long exact sequence 
\begin{equation} \label{eq:u-sequence}
\cdots \rightarrow H^{*-1}(\bZ/2;V) \stackrel{h}{\longrightarrow} H^*(\bZ/2;V) \longrightarrow H^*(V) \rightarrow \cdots
\end{equation}
This sequence includes the map \eqref{eq:set-h-to-zero}. Note that this map lands in the $\bZ/2$-invariant part of $H^*(V)$. Hence
\begin{equation} \label{eq:smith-0}
\mathrm{dim}\, H^*(V)^{\bZ/2} \geq \mathrm{dim}\, H^*(\bZ/2; V)/h H^*(\bZ/2;V).
\end{equation}
If $H^*(V)$ is finite-dimensional, $H^*(\bZ/2;V)$ is a finitely generated $\bK[[h]]$-module by Lemma \ref{th:3-lemma}(ii), and $H^*(\bZ/2;V)/h H^*(\bZ/2;V)$ is the space of generators (the resulting version of \eqref{eq:smith-0} was already implicit in our proof of finite generation). As a (weaker) consequence, we find that in this case,
\begin{equation} \label{eq:smith-1}
\mathrm{dim} \, H^*(V)^{\bZ/2} \geq
\mathrm{rank}_{\bK[[h]]} \, H^*(\bZ/2;V) = 
\mathrm{dim}_{\bK((h))} \, \hat{H}^*(\bZ/2;V).
\end{equation}

Given an arbitrary chain complex $V$ (with no given group action), one can equip $V \otimes V$ with the involution which exchanges the two factors, and consider the associated equivariant cohomology $H^*(\bZ/2; V \otimes V)$. Since $V$ is quasi-isomorphic to $H^*(V)$ (in a way that is unique up to chain homotopy), $V \otimes V$ is equivariantly quasi-isomorphic to $H^*(V) \otimes H^*(V)$ (in a way which which is unique up to equivariant chain homotopy). Hence, we have a canonical isomorphism
\begin{equation} \label{eq:quasi-square}
H^*(\bZ/2; V \otimes V) \iso H^*(\bZ/2; H^*(V) \otimes H^*(V)).
\end{equation}

There is also a canonical (but nonlinear in general) degree-doubling map
\begin{equation} \label{eq:tate}
H^*(V) \longrightarrow H^{2*}(\bZ/2;V \otimes V).
\end{equation}
On cocycles, this is given by $v \mapsto v \otimes v$. Well-definedness on the cohomology level is established by observing that
\begin{equation}
(v + d_V w) \otimes (v + d_V w) - v \otimes v = d_C\big(v \otimes w + w \otimes v + w \otimes d_V w + h(w \otimes w) \big).
\end{equation}
Even though \eqref{eq:tate} is not linear, it becomes linear after multiplying by $h$, since for cocycles $v_1,v_2$ one has
\begin{equation}
h\big( (v_2+v_1) \otimes (v_2+v_1) - v_1 \otimes v_1 - v_2 \otimes v_2 \big) = d_C(v_1 \otimes v_2).
\end{equation}
Let's take \eqref{eq:tate} and compose it with the map from equivariant cohomology to the Tate version. This yields a degree-doubling map
\begin{equation} \label{eq:hat-tate}
H^*(V) \longrightarrow \hat{H}^{2*}(\bZ/2;V \otimes V).
\end{equation}
We know that this becomes linear after multiplying by $h$, but since $h$ acts invertibly on Tate cohomology, it follows that \eqref{eq:hat-tate} is itself linear. One can extend it uniquely to a $\bK((h))$-module homomorphism
\begin{equation} \label{eq:tate-isomorphism}
H^*(V)((h)) \longrightarrow \hat{H}^*(\bZ/2;V \otimes V)
\end{equation}
(we have omitted the $2$ in the superscript, since \eqref{eq:tate-isomorphism} is no longer degree-doubling for the standard choice of grading on $H^*(V)((h))$; it is best thought of as a map of ungraded $\bK((h))$-modules).

\begin{lemma}[\protect{\cite[Lemma 2.3]{kaledin09}}] \label{th:kaledin}
The map \eqref{eq:tate-isomorphism} is an isomorphism of $\bK((h))$-modules.
\end{lemma}

\subsection{Topology background}
Let $M$ be a smooth compact manifold (possibly with boundary) with a $\bZ/2$-action. The equivariant cohomology $H^*_{\bZ/2}(M)$ is most commonly defined through the Borel construction \cite{borel60}, but there is also an equivalent algebraic version (see e.g.\ \cite[Section VII.7]{brown}), which suits our discussion better. Namely, let $C^*(M)$ be the singular cochain complex with $\bK$-coefficients, which carries an induced action of $\bZ/2$. The equivariant cochain complex is $C^*_{\bZ/2}(M) = C^*(\bZ/2;C^*(M))$, and the equivariant cohomology is correspondingly $H^*_{\bZ/2}(M) = H^*(\bZ/2; C^*(M))$. There is also a parallel Tate version $\hat{H}^*_{\bZ/2}(M) = \hat{H}^*(\bZ/2; C^*(M))$ (see e.g.\ \cite[Section VII.10]{brown}).

Let $M^{\bZ/2} \subset M$ be the fixed point set of the $\bZ/2$-action. Since the action is trivial when restricted to it, we have $H^*_{\bZ/2}(M^{\bZ/2}) = H^*(M^{\bZ/2})[[h]]$. The standard restriction map on cocycles, $C^*(M) \rightarrow C^*(M^{\bZ/2})$, is clearly equivariant, hence induces a restriction map on equivariant cohomology, which is a homomorphism of graded $\bK[[h]]$-modules
\begin{equation} \label{eq:fixed-point-restriction}
H^*_{\bZ/2}(M) \longrightarrow H^*(M^{\bZ/2})[[h]].
\end{equation}

\begin{theorem}[Localization theorem \protect{\cite[Chapter IV, Proposition 3.6]{borel60}}]
\label{th:localization}
The map \eqref{eq:fixed-point-restriction} becomes an isomorphism after tensoring with $\bK((h))$. In other words, restriction to the fixed point set induces an isomorphism on the Tate version of equivariant cohomology.
\end{theorem}

This theorem and \eqref{eq:smith-1} imply the Smith inequality
\begin{equation} \label{eq:classical-smith}
\mathrm{dim} \, H^*(M)^{\bZ/2} \geq \mathrm{dim}\, H^*(M^{\bZ/2}).
\end{equation}
The localization theorem is not hard to prove. It is technically convenient to use Morse cochains rather than singular cochains, since the Morse complexes are finite-dimensional (compare e.g.\ \cite[Proposition VII.10.1]{brown} or \cite[Theorem 2.6]{lipshitz-treumann12}, which both use equivariant cell decompositions, for the same reason). Equip the pair $(M,M^{\bZ/2})$ with a suitable Morse function and metric \cite[Definition 4.27]{schwarz-morse}, so that the Morse cochain complex $\mathit{CM}^*(M)$ comes with a projection to its counterpart $\mathit{CM}^*(M^{\bZ/2})$, implementing the Morse homology analogue of the restriction map. One can do this invariantly with respect to the $\bZ/2$-action \cite[Example 4]{seidel-smith10}, and the induced map on group cohomology is the Morse-theoretic counterpart of \eqref{eq:fixed-point-restriction}. The kernel of the projection, which is the relative Morse complex $\mathit{CM}^*(M,M^{\bZ/2})$, has generators which are the non-$\bZ/2$-invariant critical points of our Morse function. Hence, it satisfies the conditions from Example \ref{th:tate-acyclic}, which means that $\hat{H}^*(\bZ/2; \mathit{CM}^*(M,M^{\bZ/2})) = 0$. In view of the Tate analogue of the long exact sequence \eqref{eq:equivariant-les-0}, this implies Theorem \ref{th:localization}.

In parallel with the previous algebraic discussion, let's take an arbitrary $M$ (with no given action), and consider the $\bZ/2$-action on $M \times M$ which exchanges the two factors. While the Eilenberg-Zilber \cite{eilenberg-zilber} isomorphism $H^*(M \times M) \iso H^*(M) \otimes H^*(M)$ is $\bZ/2$-equivariant, the underlying chain map is not. However, there is a refinement of its construction \cite{dold59} which yields the following:

\begin{theorem} \label{th:dold} There is a canonical isomorphism
\begin{equation} \label{eq:equi-kunneth}
H^*_{\bZ/2}(M \times M) \iso H^*(\bZ/2; C^*(M) \otimes C^*(M)).
\end{equation}
\end{theorem}

Because of \eqref{eq:quasi-square}, this means that $H^*_{\bZ/2}(M \times M)$ depends only on $H^*(M)$. By combining \eqref{eq:equi-kunneth} with the restriction map \eqref{eq:fixed-point-restriction}, one gets a map of graded $\bK[[h]]$-modules
\begin{equation} \label{eq:equi-cup}
H^*(\bZ/2; C^*(M) \otimes C^*(M)) \longrightarrow H^*(M)[[h]].
\end{equation}
We should add that the construction from \cite{dold59} fits into a commutative diagram
\begin{equation}
\xymatrix{
\ar[d] H^*_{\bZ/2}(M \times M) \ar[rr]^-{\iso} && H^*(\bZ/2; C^*(M) \otimes C^*(M)) \ar[d]
\\
H^*(M \times M) \ar[rr]^-{\iso} && H^*(M) \otimes H^*(M)
}
\end{equation}
where the bottom $\rightarrow$ is the ordinary Eilenberg-Zilber map. From this, it follows that \eqref{eq:equi-cup} fits into a commutative diagram
\begin{equation} \label{eq:cup-diagram}
\xymatrix{
H^*(\bZ/2; C^*(M) \otimes C^*(M)) \ar[rr] \ar[d] && H^*(M)[[h]] \ar[d]^-{\text{set $h$ to zero}} \\
H^*(M) \otimes H^*(M) \ar[rr] && H^*(M)
}
\end{equation}
where the bottom $\rightarrow$ is the ordinary cup product. With that in mind, we call \eqref{eq:equi-cup} the equivariant cup product. By combining it with \eqref{eq:tate}, we get a map
\begin{equation} \label{eq:total-steenrod}
H^*(M) \longrightarrow H^*(M)[[h]],
\end{equation}
called the total Steenrod operation. Here, the grading on $H^*(M)[[h]]$ combines that on $H^*(M)$ and on $\bK[[h]]$; with respect to that combined grading, \eqref{eq:total-steenrod} is degree-doubling. We know from our discussion of \eqref{eq:tate} that \eqref{eq:total-steenrod} becomes linear after multiplying by $h$, and since the target has no $h$-torsion, the map itself must be linear. From \eqref{eq:cup-diagram} one sees that the constant ($h^0$) component of \eqref{eq:total-steenrod} is the ordinary cup square. The higher order parts are the Steenrod squares (this is essentially Steenrod's construction of cohomology operations \cite{steenrod-epstein}). Concretely, in those terms \eqref{eq:total-steenrod} is given by
\begin{equation}  \label{eq:total-sq}
x \mapsto x^2 + h\, \mathit{Sq}^{|x|-1}(x) + h^2 \, \mathit{Sq}^{|x|-2}(x) + \cdots
\end{equation}
By Lemma \ref{th:kaledin} and Theorem \ref{th:localization}, \eqref{eq:total-steenrod} induces an automorphism of $H^*(M)((h))$ as an ungraded $\bK((h))$-module. This is a weak version of the classical fact that $\mathit{Sq}^i = 0$ for $i<0$, and $\mathit{Sq}^0 = \mathit{id}$ (which means that \eqref{eq:total-sq} can be written as $x \mapsto h^{|x|}x + \text{lower powers of $h$}$). 
\begin{figure}
\begin{centering}
\begin{picture}(0,0)%
\includegraphics{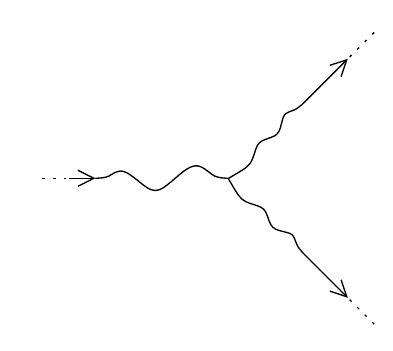}%
\end{picture}%
\setlength{\unitlength}{3355sp}%
\begingroup\makeatletter\ifx\SetFigFont\undefined%
\gdef\SetFigFont#1#2#3#4#5{%
  \reset@font\fontsize{#1}{#2pt}%
  \fontfamily{#3}\fontseries{#4}\fontshape{#5}%
  \selectfont}%
\fi\endgroup%
\begin{picture}(2205,1969)(961,-2084)
\put(976,-1141){\makebox(0,0)[lb]{\smash{{\SetFigFont{10}{12.0}{\rmdefault}{\mddefault}{\updefault}{\color[rgb]{0,0,0}$y$}%
}}}}
\put(2851,-1636){\makebox(0,0)[lb]{\smash{{\SetFigFont{10}{12.0}{\rmdefault}{\mddefault}{\updefault}{\color[rgb]{0,0,0}$-\nabla f$}%
}}}}
\put(3151,-2011){\makebox(0,0)[lb]{\smash{{\SetFigFont{10}{12.0}{\rmdefault}{\mddefault}{\updefault}{\color[rgb]{0,0,0}$x^+$}%
}}}}
\put(3151,-286){\makebox(0,0)[lb]{\smash{{\SetFigFont{10}{12.0}{\rmdefault}{\mddefault}{\updefault}{\color[rgb]{0,0,0}$x^-$}%
}}}}
\put(2851,-661){\makebox(0,0)[lb]{\smash{{\SetFigFont{10}{12.0}{\rmdefault}{\mddefault}{\updefault}{\color[rgb]{0,0,0}$-\nabla f$}%
}}}}
\put(1276,-1336){\makebox(0,0)[lb]{\smash{{\SetFigFont{10}{12.0}{\rmdefault}{\mddefault}{\updefault}{\color[rgb]{0,0,0}$-\nabla f$}%
}}}}
\end{picture}%
\caption{\label{fig:y-graph}}
\end{centering}
\end{figure}

For the purposes of translating it to Floer theory, it is instructive to mention the Morse-theoretic version of \eqref{eq:equi-cup}, which was introduced by Betz-Cohen \cite{betz93, betz-cohen94} (see \cite{cohen-norbury12} for a more detailed account). Fix a Morse function $k$ and metric on $M$, so as to define the associated Morse complex $\mathit{CM}^*(M)$. This comes with a product structure (a version of that in \cite{fukaya93})
\begin{equation} \label{eq:morse-product} 
\mathit{CM}^*(M) \otimes \mathit{CM}^*(M) \longrightarrow \mathit{CM}^*(M),
\end{equation}
defined by counting perturbed graph flow lines (Figure \ref{fig:y-graph}). More precisely, one chooses a time-dependent vector field $Y(s)$ ($s \leq 0$), which agrees with $\nabla k$ for $s \ll 0$; and similarly vector fields $X^{\pm}(s)$ ($s \geq 0$), which agree with $\nabla k$ for $s \gg 0$. All are subject to suitable (generically satisfied) transversality conditions. The relevant perturbed gradient flow equation is then
\begin{equation} \label{eq:y-graph-equation}
\left\{
\begin{aligned}
& v: (-\infty,0] \longrightarrow M, && dv_0/ds + Y(s) = 0, && \textstyle\lim_{s \rightarrow -\infty} v(s) = y, \\
& u^+: [0,\infty) \longrightarrow M, && du^+/ds + X^+(s) = 0, && \textstyle\lim_{s \rightarrow +\infty} u^+(s) = x^+, \\
& u^-: [0,\infty) \longrightarrow M, && du^-/ds + X^-(s) = 0, && \textstyle\lim_{s \rightarrow +\infty} u^-(s) = x^-,\\
& v(0) = u^+(0) = u^-(0),
\end{aligned}
\right.
\end{equation}
where $y,x^{\pm}$ are critical points of $k$. Even though the underlying graph admits a $\bZ/2$-action, the perturbations introduced in \eqref{eq:y-graph-equation} will destroy that symmetry, because one may not usually choose $X^+ = X^-$. Hence, \eqref{eq:morse-product} is not strictly commutative. However, in view of the general fact that different choices lead to chain homotopic products, it is commutative up to chain homotopy. That chain homotopy is the first term (in $h$) of a refinement of \eqref{eq:morse-product}, the equivariant Morse product, which is a graded $\bK[[h]]$-module map
\begin{equation} \label{eq:y-equi}
C^*(\bZ/2; \mathit{CM}^*(M) \otimes \mathit{CM}^*(M)) \longrightarrow \mathit{CM}^*(M)[[h]].
\end{equation}
On a technical level, the chain homotopy is defined by a version of \eqref{eq:y-graph-equation} involving an additional parameter. Similarly, the higher order terms of \eqref{eq:y-equi} involve higher-dimensional parameter spaces.

\begin{remark}
The correspondence between \eqref{eq:y-equi} and \eqref{eq:equi-cup} may not be immediately obvious, because we have described the latter as the composition of \eqref{eq:equi-kunneth} and the restriction map; it becomes clearer if one adopts a one-step description of \eqref{eq:equi-cup}, as in \cite[p.~271]{spanier}.
\end{remark}

\subsection{Symplectic fixed points}
Returning to our main topic of symplectic automorphisms, we begin by stating more precisely the situation we are addressing.

\begin{setup} \label{th:setup-1}
Let $(M,\omega_M,\theta_M)$ be a Liouville domain. This means that $M$ is a compact mani\-fold with boundary, with an exact symplectic form $\omega_M = d\theta_M$, such that the dual Liouville vector field $Z_M$ points transversally outwards along the boundary. Let $r_M \in \smooth(M,\bR)$ be a function satisfying
\begin{equation} \label{eq:r-function}
r_M| \partial M = 1, \quad \text{ and } Z_M.r_M = r_M \text{ near $\partial M$.}
\end{equation}
This is unique as a germ near $\partial M$. We will consider only those symplectic automorphisms $\phi$ which are exact in the strict sense, meaning that 
\begin{equation} \label{eq:exact-phi}
\phi^*\theta_M - \theta_M = dG_\phi
\end{equation}
for some function $G_\phi$ which vanishes near $\partial M$. This implies that $\phi$ preserves $Z_M$ near the boundary, hence that
\begin{equation} \label{eq:r-1}
\phi^*r_M = r_M \text{ near $\partial M$.}
\end{equation}
We require that $\phi$ should have no fixed points on $\partial M$. Finally, we require nondegeneracy of its fixed points.
\end{setup}

Recall that a fixed point $x$ of $\phi$ is called nondegenerate if $1$ is not an eigenvalue of $D\phi_x$, which means that $\mathrm{det}(I-D\phi_x) \neq 0$. It is elementary to show that any symplectic automorphism satisfying \eqref{eq:exact-phi} and with no fixed points on the boundary can be perturbed (by a Hamiltonian perturbation supported in the interior of $M$) so that its fixed points become nondegenerate. In this sense, nondegeneracy is a generic condition within the class we are considering.
%In our case, nondegeneracy of the fixed points actually implies that there are no fixed points on $\partial M$, because near the boundary, $\phi$ preserves the Liouville vector field (hence, any fixed point on the boundary would belong to a one-parameter family of such points). Nevertheless, we have preferred to state the two conditions separately.
The Floer cochain complex $\mathit{CF}^*(\phi)$ associated to such a $\phi$ is a finite-dimensional $\bZ/2$-graded complex of vector spaces over $\bK$. Its cohomology $\mathit{HF}^*(\phi)$ is the fixed point Floer cohomology of $\phi$, in the sense of \cite{floer88, dostoglou-salamon93}. Formally, the definition can be interpreted as Morse theory applied to the action functional on the twisted free loop space $\scrL_\phi$ (see Section \ref{subsec:floer}).
%\begin{remark}
%The Euler characteristic of $\mathit{HF}^*(\phi)$ is the classical Lefschetz number $\Lambda(\phi)$. The supertrace of the involution on $\mathit{HF}^*(\phi^2)$ (the Euler characteristic
%\end{remark}

\begin{setup} \label{th:setup-2}
Let $\phi$ be as in Setup \ref{th:setup-1}. Additionally, assume that $\phi^2$ has no fixed points on $\partial M$, and that all its fixed points are nondegenerate (then, $\phi^2$ satisfies all the conditions from Setup \ref{th:setup-2}, since the rest are consequences of the corresponding properties of $\phi$). 
\end{setup}

Nondegeneracy of the fixed points of $\phi^2$ is generic within our class of $\phi$, in the same sense as before. This is a version of the more general nondegeneracy result for periodic points from \cite{robinson70}. One can now define $\mathit{HF}^*(\phi^2)$. As mentioned before, this carries a $\bZ/2$-action, arising from a symmetry (half-rotation) of $\scrL_{\phi^2}$. Due to transversality issues, there is no underlying $\bZ/2$-action on Floer cochains. Nevertheless, one can still define an analogue of the equivariant complex \eqref{eq:c-complex}, which has the form
\begin{equation} \label{eq:equi-cf}
\mathit{CF}^*_{\mathit{eq}}(\phi^2) = \mathit{CF}^*(\phi^2)[[h]].
\end{equation}
The differential on \eqref{eq:equi-cf} consists of the ordinary Floer differential plus an a priori infinite number of additional terms (of increasingly higher powers in $h$). The resulting equivariant Floer cohomology $\mathit{HF}^*_{\mathit{eq}}(\phi^2)$ is a finitely generated $\bZ/2$-graded $\bK[[h]]$-module. It fits into a long exact sequence analogous to \eqref{eq:u-sequence}, hence one gets a counterpart of \eqref{eq:smith-1}:
\begin{equation} \label{eq:smith-2}
\mathrm{dim}\, \mathit{HF}^*(\phi^2)^{\bZ/2} \geq \mathrm{rank}_{\bK[[h]]} \, \mathit{HF}^*_{\mathit{eq}}(\phi^2) = 
\mathrm{dim}_{\bK((h))} \, \mathit{HF}^*_{\mathit{eq}}(\phi^2) \otimes_{\bK[[h]]} \bK((h)).
\end{equation}
So far, none of this is fundamentally new: equivariant Floer cohomology, in various forms, has a long history both in gauge theory \cite{austin-braam95, donaldson02b, froyshov10} and in symplectic geometry \cite{viterbo97a, hutchings08, seidel-smith10, bourgeois-oancea09b}. The treatment in this paper follows the initial part of \cite{seidel-smith10}, see also \cite{hutchings08}.

Fixed point Floer cohomology has a product structure, the pair-of-pants product \cite{schwarz95, salamon99}, which in particular gives rise to a map
\begin{equation} \label{eq:pair-of-pants}
\mathit{HF}^*(\phi) \otimes \mathit{HF}^*(\phi) \longrightarrow \mathit{HF}^*(\phi^2).
\end{equation}
If one equips $\mathit{HF}^*(\phi) \otimes \mathit{HF}^*(\phi)$ with the $\bZ/2$-action which exchanges the two factors, then \eqref{eq:pair-of-pants} becomes $\bZ/2$-equivariant, which means that the following diagram commutes:
\begin{equation} \label{eq:commutativity}
\xymatrix{
\mathit{HF}^*(\phi) \otimes \mathit{HF}^*(\phi) \ar[d]_-{\text{exchange factors}} \ar[rr]^-{\text{pair-of-pants}} && \mathit{HF}^*(\phi^2) \ar[d]^-{\text{involution}} \\
\mathit{HF}^*(\phi) \otimes \mathit{HF}^*(\phi) \ar[rr]^-{\text{pair-of-pants}} && \mathit{HF}^*(\phi^2).
}
\end{equation}
Given that, it is natural to look for a refinement on the level of equivariant cohomology, and that is our equivariant pair-of-pants product \eqref{eq:equi-pants}. The construction of the product, and the proof of its main property (Theorem \ref{th:main}), are the principal results of this paper.

By combining \eqref{eq:equi-pants} with \eqref{eq:tate}, one gets a degree-doubling map
\begin{equation} \label{eq:quantum-steenrod}
\mathit{HF}^*(\phi) \longrightarrow \mathit{HF}^{2*}_{\mathit{eq}}(\phi^2).
\end{equation}
One can extend this uniquely to a map of $\bK((h))$-modules (not preserving the $\bZ/2$-grading)
\begin{equation} \label{eq:tate-floer}
\mathit{HF}^*(\phi)((h)) \longrightarrow \mathit{HF}^*_{\mathit{eq}}(\phi^2) \otimes_{\bK[[h]]} \bK((h)).
\end{equation}
From Lemma \ref{th:kaledin} and Theorem \ref{th:main}, it follows that \eqref{eq:tate-floer} is an isomorphism, which proves Corollary \ref{th:localized-isomorphism}. In view of \eqref{eq:smith-2}, Corollary \ref{th:floer-smith} follows.

\begin{example} \label{th:pss}
Let $(\phi_t)$ be the Hamiltonian flow of a function \eqref{eq:r-function}, assumed to be Morse. Consider $\phi_t$ for sufficiently small $t>0$ (the fixed points correspond to the critical points of our function, and are nondegenerate). One has 
\begin{equation} \label{eq:pss}
\mathit{HF}^*(\phi_t) \iso H^*(M), 
\end{equation}
and the same applies to $\phi_t^2 = \phi_{2t}$. The $\bZ/2$-action on $\mathit{HF}^*(\phi_t^2)$ is trivial, and in fact, there is a canonical isomorphism 
\begin{equation} \label{eq:equi-trivial}
\mathit{HF}^*_{\mathit{eq}}(\phi_t^2) \iso H^*(M)[[h]]
\end{equation}
(but we will not prove that here). The isomorphism \eqref{eq:pss} relates the pair-of-pants product to the standard cup product. In parallel, one expects that under \eqref{eq:equi-trivial}, the equivariant pair-of-pants product will correspond to \eqref{eq:equi-cup}. This becomes particularly plausible when one compares the Morse-theoretic version \eqref{eq:y-equi} with our construction of \eqref{eq:equi-pants} (Section \ref{sec:con}).
\end{example}

With the above example in mind, one can think of \eqref{eq:quantum-steenrod} as a Steenrod squaring operation in Floer cohomology. Of course, for general $\phi$ its formal structure is not really analogous to that of Steenrod squares, since it relates different Floer cohomology groups. We postpone further discussion of this issue to Section \ref{subsec:quantum-steenrod}, and consider some simple applications, in which Corollary \ref{th:floer-smith} plays the main role.
%
%
% question: (x,y) -> (\phi(y),x). Its fixed points are x = \phi(y), y = x, hence equal those of \phi
% the square is (x,y) -> (\phi(y),x) -> (\phi(x),\phi(y)), so the Smith inequality is obvious
% in this case, one thinks that the equivariant pair-of-pants map should be an isomorphism
% 
% (x,y,z) -> (\phi(z),x,y). Its fixed points are those of \phi
% if I square it, I should get the fixed points of \phi^2, so nothing new.
% Suppose \phi is \Z/2-equivariant. We can take \phi \circ \iota.
%
% Take X -> D Lefschetz fibration with fibres conics (which degenerate)
% Lift a diffeomorphism of (D,marked points), obtained as a braid.
% These are precisely the S^1-invariant diffeomorphisms. They have a circle of fixed
% points over each original fixed point. One can try to move all fixed points into the
% fixed point set of the involution, but minimality is an issue. 
% General conic fibrations xy = ?. 
% E.g. powers of Dehn twists. We show that their symplectic fixed point theory
% grows? Or, pseudo-Anosovs in 4 dimensions.

\begin{application} \label{th:powers-of-2}
Let $\scrS$ be the group of exact symplectic automorphisms of $M$ which are the identity near the boundary. Take $\phi \in \scrS$, and perturb it to $\tilde{\phi} = \phi \circ \phi_t$, using the same 
$\phi_t$ as in Example \ref{th:pss}. Suppose that
\begin{equation}
\mathrm{dim}\, \mathit{HF}^*(\tilde{\phi}) > \mathrm{dim}\, H^*(M).
\end{equation}
Then, the same holds for $\tilde{\phi}^2$, by Corollary \ref{th:floer-smith}. Now, $\tilde{\phi}^2$ is isotopic (rel boundary) to $\phi^2 \circ \phi_{2t}$. Using the isotopy invariance of Floer cohomology and \eqref{eq:pss}, it follows that $[\phi^2] \in \pi_0(\scrS)$ is nontrivial. Moreover, this argument can be iterated, hence the classes 
\begin{equation}
[\phi^2], [\phi^4], [\phi^8],\dots \in \pi_0(\scrS)
\end{equation}
are all nontrivial. Under the additional Assumption \ref{th:stably-trivial}, this was proved in \cite[Corollary 1.3]{hendricks14}, by the same argument.

As an aside, note that if we had an analogue of our theory for all primes $p$, there would be similar statements about powers $\phi^{p^k}$. However, since the theory would use Floer cohomology with coefficients in a characteristic $p$ field, the arguments for different primes can't be combined. It is not clear to the author how to address all iterates in this way.
\end{application}

One can compare the previous application with a classical (purely topological) statement, which says that if the Lefschetz number $\Lambda(\phi)$ satisfies
\begin{equation}
|\Lambda(\phi)| > \mathrm{dim}\, H^*(M;\bQ), 
\end{equation}
then $\phi$ has infinite order up to homotopy (because the action of $\phi$ on rational cohomology must have an eigenvalue with norm $>1$). The connection between the two statements is given by the elementary fact that $\Lambda(\phi)$ is the Euler characteristic of $\mathit{HF}^*(\phi)$. The two kinds of arguments can also be combined fruitfully:

\begin{application}
Suppose that $M$ has nontrivial rational homology only in degrees $0$ and $n$, where $n$ is odd. 
Take an automorphism $\phi$ which satisfies \eqref{eq:exact-phi}, and which acts as minus the identity on $H_n(M;\bQ)$. Then, for any $d$ such that $\phi^d$ has no fixed points on $\partial M$, we have
\begin{equation}
\mathrm{dim}\, \mathit{HF}^*(\phi^d) \geq \mathrm{dim}\,H_*(M;\bQ).
\end{equation}
If $d$ is odd, this is an Euler characteristic computation, $\mathrm{dim}\, \mathit{HF}^*(\phi^d) \geq \Lambda(\phi^d) = \mathrm{dim}\, H_*(M;\bQ)$. The case of even $d$ then follows by applying Corollary \ref{th:floer-smith} to $\phi^{d/2}$.
\end{application}

\begin{application}
Suppose that $M$ admits an involution $\iota$ (compatible with its Liouville structure). Consider an automorphism $\phi$ which satisfies \eqref{eq:exact-phi}, which commutes with $\iota$, and such that $\phi^2$ has no fixed points on $\partial M$. Let $\bar\phi$ be the induced map on the quotient $\bar{M} = M/\iota$. By considering the splitting of cohomology into $\iota$-eigenspaces, one gets
\begin{equation}
\Lambda(\phi) + \Lambda(\iota \circ \phi) = 2 \Lambda(\bar\phi).
\end{equation}
Using the fact that $\phi^2 = (\iota \circ \phi)^2$ and Corollary \ref{th:floer-smith}, one gets
\begin{equation} \label{eq:square-inequality}
\mathrm{dim}\, \mathit{HF}^*(\phi^2) \geq \half (\mathrm{dim}\, \mathit{HF}^*(\phi) + \mathrm{dim}\, \mathit{HF}^*(\iota \circ \phi) )
\geq \half |\Lambda(\phi) + \Lambda(\iota \circ \phi)| = |\Lambda(\bar\phi)|.
\end{equation}
A concrete case of interest is where $M$ is the Milnor fibre of a hypersurface singularity which has multiplicity $m=2$, and which therefore can be written as $x_0^2 + p(x_1,\dots,x_n) = 0$ in local holomorphic coordinates. One takes $\iota$ to be the involution which reverses $x_0$, and $\phi$ the monodromy (perturbed as in Application \ref{th:powers-of-2}). Here, $\bar{M}$ is contractible, so $\Lambda(\bar\phi) = 1$, hence $\mathit{HF}^*(\phi^2) \neq 0$. This nonvanishing statement (which one can also obtain using \cite{hendricks14}, since Assumption \ref{th:stably-trivial} holds here) confirms the first instance of a general conjecture, according to which the $m$-th power of the monodromy has nontrivial fixed point Floer cohomology.
\end{application}

One can compare \eqref{eq:square-inequality} to the elliptic relation \cite{tonkonog14}, which in the same context yields
\begin{equation} \label{eq:tonkonog}
\mathrm{dim} \, \mathit{HF}^*(\phi^2) \geq |\Lambda(\phi^2|M^\iota)|.
\end{equation}
% This applies with Q-coefficients, hence bounds the rank of the Z-coefficient Floer cohomology,
% but by the universal coefficient sequence, we get the same for Z/2-coefficients.
%
% Supertrace of the action of \phi on \phi^2 = supertrace of the action of \phi on \phi^2.
%
% 2\bar{M} - M^\iota = M ?
% 2 Lambda(\bar{\phi}) - \Lambda(\phi|M^\iota) = 
%
% rank \geq |Str(\iota|HF^*)| = |Str(\phi acting on M^\iota)|
%where $\phi^{\mathit{fix}} = \phi|M^{\mathit{fix}}$. More precisely, what one gets first from \cite{tonkonog14} is a lower bound for the rank of the Floer cohomology with $\bZ$-coefficients; but the corresponding bound with $\bK$-coefficients follows from that via the long exact coefficient sequence. Corollary \ref{th:floer-smith} then implies that
%\begin{equation}
%\mathrm{rank} \, \mathit{HF}^*(\phi^2) \geq |\Lambda(\phi^{\mathit{fix}})|.
%\end{equation}
%Depending on circumstances, this may or may not be better than the bound one gets from applying \eqref{eq:tonkonog} itself to $\phi^2$.
%\end{example}
%
% \chi(n)-1. What can one say about the rank of Floer cohomology?
% trace of a symplectic matrix??
%
% monodromy. Trace is even dimensions must be -1, hence there are negative
% eigenvalues, but not necessarily real ones.
% 

\subsection{Symplectic cohomology\label{subsec:symplectic}}
Theorem \ref{th:main} has potential structural implications for $S^1$-equivariant symplectic cohomology, in its ``uncompleted'' or ``finitely supported'' version (``finitely supported'' is the terminology from \cite{zhao14}, which in terms of \cite{albers-cieliebak-frauenfelder14} corresponds to $\underrightarrow{H}\underleftarrow{T}$; however, unlike those two references, we do not a priori invert the equivariant parameter). These implications rely on some compatibility statements (made explicit below), which seem natural but are not proved in this paper. Nevertheless, we discuss the argument briefly here, since it sheds light on the rather remarkable outcome of the computations in \cite{albers-cieliebak-frauenfelder14, zhao14}.

As before, let $M$ be a Liouville domain, and $(\phi_t)$ the Hamiltonian flow of a function \eqref{eq:r-function}. This time, we consider it for large times, and define symplectic cohomology \cite{viterbo97a} as
\begin{equation} \label{eq:symplectic-co}
\mathit{SH}^*(M) = \underrightarrow{\lim}_t \, \mathit{HF}^*(\phi_t).
\end{equation}
The homomorphisms in the direct system are suitable continuation maps. Bearing in mind that $\phi_t = \phi_{t/2}^2$, one can define the $\bZ/2$-equivariant Floer cohomology of $\phi_t$. Let's denote this by $\mathit{HF}^*_{\bZ/2}(\phi_t)$ rather than our usual $\mathit{HF}^*_{\mathit{eq}}(\phi_t)$. 

When defining the equivariant analogue of \eqref{eq:symplectic-co}, one is faced with two different possibilites (because of the issue pointed out in Remark \ref{th:colimit}). Both versions yield $\bZ/2$-graded $\bK[[h]]$-modules, and both fit into long exact sequences
\begin{equation}
\cdots \rightarrow \mathit{SH}^{*-1}_{\bZ/2}(M) \stackrel{h}{\longrightarrow} \mathit{SH}^*_{\bZ/2}(M) \longrightarrow \mathit{SH}^*(M) \rightarrow \cdots 
\end{equation} 
However, otherwise they are quite differently behaved. The first possibility is to build a theory based on cochain spaces which are complete with respect to the filtration by powers of $h$. Concretely, if $\mathit{SC}^*(M)$ is the cochain space underlying $\mathit{SH}^*(M)$ (let's say, defined using a quadratically growing Hamiltonian), then the equivariant version would use $\mathit{SC}^*(M)[[h]]$, with a differential that modifies that on $\mathit{SC}^*(M)$ by terms of order $\geq 1$ in $h$ (compare \cite[Remark 8.1]{biased} for the $S^1$-equivariant theory). From an algebraic perspective, this puts us in the situation of Remark \ref{th:spectral-sequence}. In particular, this version of equivariant symplectic cohomology vanishes whenever $\mathit{SH}^*(M) = 0$. 

However, here we will adopt the other possibility, which is this:
\begin{equation} \label{eq:2-sh}
\mathit{SH}^*_{\bZ/2}(M) \stackrel{\mathrm{def}}{=} \underrightarrow{\lim}_t \, \mathit{HF}^*_{\bZ/2}(\phi_t).
\end{equation}
Obviously, to make that rigorous, one needs equivariant continuation maps. Suppose that such maps have been defined, and that they commute with the equivariant pair-of-pants product. After applying \eqref{eq:tate}, one would then have commutative diagrams ($s<t$)
\begin{equation} \label{eq:sh-tate-1}
\xymatrix{
\mathit{HF}^*(\phi_{s/2}) \ar[rrrr]^-{\text{continuation map}} \ar[d] &&&&
\mathit{HF}^*(\phi_{t/2}) \ar[d] \\
\mathit{HF}^{2*}_{\bZ/2}(\phi_s)
\ar[rrrr]^-{\text{equivariant continuation map}} &&&&
\mathit{HF}^{2*}_{\bZ/2}(\phi_t),
}
\end{equation}
hence in the direct limit a map
\begin{equation}
\mathit{SH}^*(M) \longrightarrow \mathit{SH}^{2*}_{\bZ/2}(M).
\end{equation}
Theorem \ref{th:main} implies that the vertical maps in \eqref{eq:sh-tate-1} induce isomorphisms (of ungraded $\bK((h))$-modules) $\mathit{HF}^*(\phi_{t/2}) \otimes \bK((h)) \iso \mathit{HF}^*_{\bZ/2}(\phi_t) \otimes_{\bK[[h]]} \bK((h))$. Passing to the direct limit (and noting that taking the tensor product with $\bK((h))$ commutes with the direct limit) yields
\begin{equation} \label{eq:wow-sh}
\mathit{SH}^*(M) \otimes \bK((h)) \iso \mathit{SH}^*_{\bZ/2}(M) \otimes_{\bK[[h]]} \bK((h)).
\end{equation}

Because $(\phi_t)$ is a flow, the $\bZ/2$-symmetry on the twisted loop space is the restriction of an $S^1$-symmetry. The analogue of \eqref{eq:2-sh} is a version of $S^1$-equivariant symplectic cohomology \cite[Section 5]{viterbo97a}, defined as
\begin{equation} \label{eq:s1-group}
\mathit{SH}^*_{S^1}(M) = \underrightarrow{\lim}_t \, \mathit{HF}^*_{S^1}(\phi_t).
\end{equation}
This is a module over $\bK[[u]]$, where the formal variable $u$ has degree $2$ (of course, this is not particularly meaningful since we consider $\bZ/2$-gradings only, but we say it to keep the connection with classical equivariant cohomology). It sits in a long exact sequence \cite{bourgeois-oancea09b}
\begin{equation} \label{eq:s1-les}
\cdots \rightarrow \mathit{SH}^{*-2}_{S^1}(M) \stackrel{u}{\longrightarrow} \mathit{SH}^*_{S^1}(M) \longrightarrow \mathit{SH}^*(M) \rightarrow \cdots
\end{equation}

\begin{example} \label{th:d}
In the definition \eqref{eq:s1-group} of $S^1$-equivariant symplectic cohomology, one can use Floer cohomology with coefficients in any commutative ring $R$. Let's denote the outcome, which is a module over $R[[u]]$, by $\mathit{SH}^*_{S^1}(M;R)$. The computation in \cite[Section 8.1]{zhao14} and \cite[Section 5.1]{albers-cieliebak-frauenfelder14} shows that for the two-dimensional disc $D$, 
\begin{equation} \label{eq:z-coefficient}
\mathit{SH}^*_{S^1}(D;\bZ) \iso \textstyle \bQ((u)).
\end{equation}
This implies (using the universal coefficient theorem) that
\begin{align}
& \mathit{SH}^*_{S^1}(D;\bQ) \iso \bQ((u)), \label{eq:q-coefficient} \\ 
& \mathit{SH}^*_{S^1}(D;\bF_p) = 0 \quad \text{for any prime $p$.} \label{eq:prime-coefficient}
\end{align}
\end{example}

We now return to our usual coefficient field $\bK = \bF_2$. In that situation, there is a general relation between $S^1$-equivariant cohomology and $\bZ/2$-equivariant cohomology. In classical topological terms, this means that if we are given a manifold $M$ with a circle action, and consider the action of the subgroup $\bZ/2 \subset S^1$, then 
\begin{equation} \label{eq:discretize}
H^*_{\bZ/2}(M) \iso H^*_{S^1}(M) \oplus H^{*-1}_{S^1}(M).
\end{equation}
This is an isomorphism of graded modules over $\bK[[u]]$, where the module structure on the left is defined by setting $u = h^2$. The simplest proof of \eqref{eq:discretize} uses the Borel construction; write $H^*_G(M) = H^*(EG \times_G M)$ for both $G = \bZ/2$ and $G = S^1$. The inclusion $\bZ/2 \subset S^1$ induces a map
\begin{equation} \label{eq:borel-map}
E{\bZ/2} \times_{\bZ/2} M \longrightarrow ES^1 \times_{S^1} M,
\end{equation}
which is a circle bundle whose Chern class is $2u = 0 \in H^2_{S^1}(M)$. The Gysin sequence with $\bK$-coefficients therefore splits, yielding \eqref{eq:discretize}. Even though we will not prove that here, there is a parallel result for symplectic cohomology:
\begin{equation}
\mathit{SH}^*_{\bZ/2}(M) \iso \mathit{SH}^*_{S^1}(M) \oplus \mathit{SH}^{*-1}_{S^1}(M).
\end{equation}
By combining this with \eqref{eq:wow-sh}, one gets
\begin{equation} \label{eq:periodic-sh}
\mathit{SH}^*(M) \otimes \bK((h)) \iso (\mathit{SH}^*_{S^1}(M) \oplus \mathit{SH}^{*-1}_{S^1}(M)) \otimes_{\bK[[u]]} \bK((u)).
\end{equation}
%In the terminology of \cite{zhao14}, $\mathit{SH}^*_{S^1}(M) \otimes_{\bK[[u]]} \bK((u))$ is the periodic symplectic cohomology $\mathit{PSH}(M)$. 

Suppose for instance that $\mathit{SH}^*(M) = 0$. Then \eqref{eq:periodic-sh} vanishes, which means that $u$ acts nilpotently on each element of $\mathit{SH}^*_{S^1}(M)$. By combining this with \eqref{eq:s1-les}, one sees that in fact, $\mathit{SH}^*_{S^1}(M) = 0$, which agrees with \eqref{eq:prime-coefficient}.

\begin{remark} \label{th:homotopy-fixed}
One expects a corresponding result for $\bF_p$-coefficients for any $p$, using $\bZ/p$-equivariant Floer cohomology. This would explain why, when we used integer coefficients in Example \ref{th:d}, the outcome \eqref{eq:z-coefficient} was already a $\bQ((u))$-module: the same should happen whenever ordinary symplectic cohomology (with $\bZ$-coefficients) vanishes.
\end{remark}

\subsection{Related work\label{subsec:quantum-steenrod}}
The general idea of ``quantum Steenrod operations'' is not new. Two distinct approaches had been proposed in the mid-1990s. The first approach was outlined in \cite[Section 2]{fukaya93b}. It is essentially a deformation of the Morse-theoretic picture (Figure \ref{fig:y-graph}) which adds ``quantum'' contributions from pseudo-holomorphic spheres. This is closely related to the idea in this paper, if one took the symplectic manifold to be closed rather than exact, and the symplectic automorphism to be the identity. More precisely, the relation between the two theories would then be parallel to that between the quantum product and the pair-of-pants product.

\begin{remark}
Especially if one considers the analogues for primes $p>2$, there is no a priori reason to expect that the operations from \cite{fukaya93b} would have all the formal properties of the classical topological Steenrod operations. The first relevant question would be whether the action of the symmetric group $S_p$ on the Deligne-Mumford space $\overline{\scrM}_{0,p+1}$ (by permuting the first $p$ marked points) has a homotopy fixed point; which means, whether there is an equivariant map $ES_p \rightarrow \overline{\scrM}_{0,p+1}$, where the notation $ES_p$ is as in \eqref{eq:borel-map}.
\end{remark}

The second approach is based on homotopy theory, hence requires Floer theory to show behaviour close to ordinary Morse theory. Taking $M$ and $\phi$ as in Setup \ref{th:setup-1}, let's impose the following:

\begin{assumption} \label{th:stably-trivial}
$TM$ is stably trivial (as a symplectic vector bundle) and, with respect to that stable trivialization, the map $D\phi: M \rightarrow \mathit{Sp}(\infty)$ is nullhomotopic. 
\end{assumption}

The twisted loop space $\scrL_\phi$ carries a polarization class, an element of $\mathit{KO}^1(\scrL_\phi)$ \cite[Section 2]{cohen-jones-segal95}. Assumption \ref{th:stably-trivial} implies that the polarization class vanishes; in fact, from this perspective the assumption is unnecessarily strong (it would be enough to reduce the structure groups involved from unitary to orthogonal groups), but we use it since it fits in well with the discussion later on. As proposed in \cite{cohen-jones-segal95}, vanishing of the polarization class should allow one to define a Floer stable homotopy type (a spectrum) whose cohomology with $\bK$-coefficients is $\mathit{HF}^*(\phi)$ (to make sense of this, note that Assumption \ref{th:stably-trivial} implies that the Floer cohomology groups can be equipped with a $\bZ$-grading). This requires certain smoothness results for compactified moduli spaces; assuming those, the construction of the homotopy type is described in \cite{cohen09} (a closely related version of Floer homotopy type is discussed in \cite{cohen07}; for constructions in other types of Floer theories, see \cite{manolescu03, lipshitz-sarkar11}). In particular, this equips Floer cohomology with Steenrod operations. For instance, $\mathit{Sq}^1$ would then be the Bockstein operator. Of course, the Bockstein exists even if Assumption \ref{th:stably-trivial} fails (since fixed point Floer cohomology can always be defined with $\bZ$-coefficients). However, one does not expect the same to hold for the general Steenrod operations arising from the Floer homotopy type. Moreover, these operations may depend on additional data that is implicit in using Assumption \ref{th:stably-trivial} (the choice of stable trivialization, and that of the nullhomotopy for $D\phi$).
%
%
%\begin{remark}
%Generally speaking, there is no reason to believe that these operations should have the same formal properties as the classical topological ones: the classical theory makes use of equivariance with respect to symmetric groups, while the ``quantum'' version only admits a cyclic group of symmetries. In other words, it seems likely that the answer to \cite[Problem 2.11]{fukaya93b} should be negative, but we will not pursue that further here.
%\end{remark}

To see how Floer homotopy type might be related to our construction, we need to discuss the localization theorem for symplectic automorphisms proved in \cite{hendricks14}. The basic starting point is the well-known relation between fixed point Floer cohomology and Lagrangian intersection Floer cohomology. This says that
\begin{equation}
\mathit{HF}^*(\phi) \iso \mathit{HF}^*(\Gamma,\Delta),
\end{equation}
where the right hand side is Lagrangian Floer cohomology in $\bar{M} \times M$ (the notation $M \mapsto \bar{M}$ indicates reversal of the sign of the symplectic form), and the Lagrangian submanifolds involved are the graph $\Gamma = \{(x,y) \;:\; y = \phi(x)\}$ as well as the diagonal $\Delta$. Similarly, one has \cite[Proposition 1.6]{hendricks14}
\begin{equation} \label{eq:square-lagrangian}
\mathit{HF}^*(\phi^2) \iso \mathit{HF}^*(\Gamma_2,\Delta_2),
\end{equation}
where now the right hand side takes place in $\bar{M} \times M \times \bar{M} \times M$, for the Lagrangian submanifolds $\Gamma_2 = \{(x_1,y_1,x_2,y_2) \;:\; y_k = \phi(x_k)\}$ and $\Delta_2 = \{(x_1,y_1,x_2,y_2) \;:\; x_2 = y_1, \, x_1 = y_2\}$. Consider the symplectic involution $(x_1,y_1,x_2,y_2) \mapsto (x_2,y_2,x_1,y_1)$. Its fixed point set can be identified with $\bar{M} \times M$, and the fixed parts of $(\Gamma_2,\Delta_2)$ with $(\Gamma,\Delta)$. 
A suitable adaptation of the arguments from \cite{seidel-smith10} (the main issue having to do with the fact that the Lagrangian submanifolds are not closed) shows that, if Assumption \ref{th:stably-trivial} holds, one can define a stabilized localization map
\begin{equation} \label{eq:localization-map}
\mathit{HF}^*_{\mathit{eq}}(\phi^2) \longrightarrow \mathit{HF}^{*+m}(\phi)[[h]]
\end{equation}
(for some large $m$; increasing $m$ amounts to multiplying the localization map with $h$), which becomes an isomorphism after tensoring with $\bK((h))$. Assumption \ref{th:stably-trivial} appears here because, as shown in \cite{hendricks14}, it implies the ``stable normal triviality'' condition on the Lagrangian submanifolds which is a requirement in \cite{seidel-smith10} (as a consequence of this relation, one expects that \eqref{eq:localization-map} depends on choices that are implicit in using Assumption \ref{th:stably-trivial}).

\begin{example} \label{th:111}
If we take our symplectic automorphism to be the identity, then $\Gamma_2 \cap \Delta_2 = M$, and the $\bZ/2$-action on it is trivial. While this is not admissible in our context, one can perturb it as in Example \ref{th:pss}, in which case it seems reasonable to think that \eqref{eq:localization-map} should be multiplication with $h^m$,
\begin{equation}
\mathit{HF}^*_{\mathit{eq}}(\phi^2) = H^*(M)[[h]] \longrightarrow
\mathit{HF}^{*+m}(\phi)[[h]] = H^{*+m}(M)[[h]]
\end{equation}
(more precisely, this should be the case if the nullhomotopy $D\phi \htp \mathit{id}$ is chosen to be the constant one). Recall that in contrast, the map \eqref{eq:quantum-steenrod} gives the total Steenrod operation.
\end{example}

As should be clear from our discussion of Example \ref{th:111}, we don't expect \eqref{eq:quantum-steenrod} and \eqref{eq:localization-map} to be inverses of each other. Instead, one should think of the general situation as follows. In general, there is no Floer stable homotopy type, and correspondingly there are no Steenrod operations which would act on $\mathit{HF}^*(\phi)$ as in classical topology. Instead, we have \eqref{eq:quantum-steenrod} which lands in a different group, namely $\mathit{HF}^*_{\bZ/2}(\phi^2)$. However, if Assumption \ref{th:stably-trivial} holds, we do have \eqref{eq:localization-map} which brings us back to $\mathit{HF}^*(\phi)$, and we then also have a Floer stable homotopy type (moreover, both depend on the same choices). Concretely, this leads to the conjecture that the composition of \eqref{eq:quantum-steenrod} and \eqref{eq:localization-map}, which yields a degree-doubling map
\begin{equation} \label{eq:total-phi-steenrod}
\mathit{HF}^*(\phi) \longrightarrow \mathit{HF}^*(\phi)((h)),
\end{equation}
agrees with the total Steenrod square (in the topological sense) associated to the Floer stable homotopy type. It seems that any attempt to prove this would require one first to revisit \cite{seidel-smith10}, with the aim of finding a more direct construction of \eqref{eq:localization-map}.

The other motivation for this work is the study \cite{lipshitz-treumann12} of $\bZ/2$-localization for the Hochschild homology of bimodules (with applications to Heegaard-Floer theory). Take a dg algebra $\scrA$ and an $\scrA$-bimodule $\scrP$ (both are assumed to be defined over $\bK$, and $\bZ$-graded). The associated Hochschild complex is
\begin{equation} \label{eq:hochschild}
\mathit{CC}_*(\scrA,\scrP) = T(\scrA[1]) \otimes \scrP,
\end{equation}
where $T(\scrA[1])$ is the tensor algebra over the shifted vector space $\scrA[1]$ (for the differential, see e.g.\ \cite[Definition 3.2]{lipshitz-treumann12}, where our choice corresponds to that of the standard bar resolution of the diagonal bimodule; the case where $\scrP$ is also the diagonal bimodule is the most classical one, see e.g.\ \cite[Section 5.3.2]{loday}). Its homology is the Hochschild homology $\mathit{HH}_*(\scrA,\scrP)$. One can consider the derived tensor product
\begin{equation}
\scrP \otimes_{\scrA}^L \scrP = \scrP \otimes T(\scrA[1]) \otimes \scrP,
\end{equation}
(where the differential is again obtained from that on the bar resolution of the diagonal bimodule), and then
\begin{equation} \label{eq:double-tensor}
\mathit{CC}_*(\scrA, \scrP \otimes_{\scrA}^L \scrP)  = T(\scrA[1]) \otimes \scrP \otimes T(\scrA[1]) \otimes \scrP
\end{equation}
carries a $\bZ/2$-action, which cyclically permutes the factors in \eqref{eq:double-tensor}. It is important to note that as a chain complex, \eqref{eq:double-tensor} is not the tensor product of two copies of \eqref{eq:hochschild}. Lipshitz and Treumann take the Tate complex $\hat{C}^*(\bZ/2;\mathit{CC}_*(\scrA, \scrP \otimes_{\scrA}^L \scrP))$ and filter it by the grading in \eqref{eq:double-tensor}. Applying \eqref{eq:tate} to the associated graded space shows that the resulting spectral sequence has
\begin{equation} \label{eq:spectral-sequence}
E_1 \iso \mathit{CC}_{*/2}(\scrA,\scrP)((h)).
\end{equation}
The $E_1$ differential vanishes, and that on the $E_2$ page can be identified with the Hochschild differential for $\scrP$ (our notation is somewhat rough; we refer to \cite[Propositions 3.10 and 3.12]{lipshitz-treumann12} for precise statements and proofs). Convergence of the spectral sequence can be taken care of by suitable homological boundedness assumptions ($\scrA$ should be smooth and proper, and $\scrP$ bounded) \cite[Proposition 3.8]{lipshitz-treumann12}. We will assume from now on that these assumptions hold. More importantly, one would like the spectral sequence to degenerate at the $E_3$ page, in order to derive an isomorphism (at least non-canonically) between $\mathit{HH}_*(\scrA,\scrP)((h))$ and $\hat{H}^*(\bZ/2;\mathit{CC}_*(\scrA, \scrP \otimes_{\scrA}^L \scrP))$. A key result says that it is enough to show this for the case when $\scrP = \scrA^!$ \cite[Theorem 5]{lipshitz-treumann12}. Further investigation of this ``$\pi$-formality'' condition leads to interesting relations with noncommutative geometry \cite{kaledin08}, which are beyond the scope of our discussion here. Assuming $\pi$-formality, one obtains a Smith-type inequality \cite[Theorem 4]{lipshitz-treumann12}
\begin{equation} \label{eq:l-t}
\mathrm{dim}\, \mathit{HH}_*(\scrA,\scrP) \leq \mathrm{dim}\, \mathit{HH}_*(\scrA,\scrP \otimes_{\scrA} \scrP).
\end{equation}

The connection with symplectic geometry concerns the case where $\scrA$ describes the Fukaya category of a (closed) symplectic manifold, and $\scrP$ is the graph bimodule of a symplectic automorphism. Assuming the existence of a suitable diagonal decomposition in the Fukaya category, $\mathit{HH}(\scrA,\scrP)$ agrees with fixed point Floer homology \cite[Conjecture 1.4]{lipshitz-treumann12}. Even though the goals are quite close, as one can see by comparing Corollary \ref{th:floer-smith} and \eqref{eq:l-t}, the Lipshitz-Treumann approach seems to be substantially different from the one in this paper; in particular, it is not clear what the geometric interpretation of \eqref{eq:spectral-sequence} should be.

\begin{remark}
Another direction for future work, which is natural from the viewpoint of \cite{lipshitz-treumann12, hendricks14}, would be to generalize our pair-of-pants product from symplectic automorphisms to closed chains of Lagrangian correspondences, replacing fixed point Floer cohomology with quilted Floer cohomology \cite{wehrheim-woodward10}.
\end{remark}

\section{Two parameter spaces\label{sec:morse}}

This section introduces certain manifolds with corners, which will be later used as parameter spaces for appropriate families of Cauchy-Riemann equations. Even though these manifolds could be defined purely combinatorially, we prefer to construct them geometrically using Morse theory.

\subsection{Morse theory for real projective space}
Take the infinite-dimensional sphere
\begin{equation}
S^\infty = \textstyle \bigcup_i S^i. 
\end{equation}
Points of $S^\infty$ are sequences $v = (\nu_0,\nu_1,\dots)$ with almost all $\nu_k \in \bR$ vanishing, and such that $\nu_0^2 + \nu_1^2 + \cdots = 1$. We consider $S^\infty$ as the union of the finite-dimensional sub-spheres $S^i = \{\nu_{i+1} = \nu_{i+2} = \cdots = 0\}$. Taking the quotient by the involution $v = (\nu_0,\nu_1,\dots) \mapsto -v = (-\nu_0,-\nu_1,\dots)$ gives rise to the infinite-dimensional real projective space $\bR P^\infty$. We will also use the shift self-embedding $\tau: S^\infty \rightarrow S^\infty$, $\tau(\nu_0,\nu_1,\dots) = (0,\nu_0,\nu_1,\dots)$. 

Take a standard Morse function on $S^\infty$,
\begin{equation}
f(v) = \textstyle\sum_k k \nu_k^2.
\end{equation}
Its critical points are $v^{i,\pm} = \{\nu_i = \pm 1, \; \nu_j = 0 \text{ for } j \neq i\}$, of value and Morse index $i$ (both have the same image $v^i$ in $\bR P^\infty$). As usual in Morse theory, we want to consider the negative gradient flow of $f$.

\begin{data} \label{th:setup-metric}
Choose a Riemannian metric on $S^\infty$ (that is to say, a sequence of mutually compatible metrics on the spheres $S^i$) such that: reversing the sign of any coordinate(s) is an isometry; and $\tau$ is an isometry. 
\end{data}

As a consequence of the symmetry condition, $-\nabla f$ is tangent to each sub-sphere $S^i$. This implies that it has a well-defined flow, which can be analyzed by finite-dimensional methods. 

\begin{lemma} \label{th:stable-unstable}
The unstable and stable manifolds of $-\nabla f$ are
\begin{align}
& W^u(v^{i,\pm}) = \{\pm \nu_i > 0, \; \nu_{i+1} = \nu_{i+2} = \cdots = 0\}, \label{eq:stable} \\ 
& W^s(v^{i,\pm}) = \{\nu_0 = \cdots = \nu_{i-1} = 0, \; \pm \nu_i > 0\}. \label{eq:unstable}
\end{align}
\end{lemma}

The answers are independent of the choice of metric (within the class from Data \ref{th:setup-metric}); in particular, one sees that $\nabla f$ is always Morse-Smale. Note that $\tau^*f = f+1$. Because of this and the assumptions on the metric, $\tau$ induces a map between the space of trajectories connecting $v^{i,\pm}$ and $v^{j,\pm}$, and the corresponding space for $v^{i+1,\pm}$ and $v^{j+1,\pm}$; the explicit description shows that this map is a diffeomorphism. Of course, there is also the involution, which exchanges the critical points $v^{i,+}$ and $v^{i,-}$, and acts correspondingly on the spaces of gradient flow lines.

\begin{proof}[Proof of Lemma \ref{th:stable-unstable}]
Each point in $S^i$ which is sufficiently close to $v^{i,\pm}$ asymptotically goes to that critical point if we flow up the gradient (because $v^{i,\pm}$ is a local maximum for $f|S^i$). For dimension reasons, this fully describes $W^u(v^{i,\pm})$ locally near the critical point. Since this local part lies entirely inside $S^i \setminus S^{i-1}$, and that set is invariant under the flow of $\nabla f$, it follows that 
\begin{equation} \label{eq:stable-1}
W^u(v^{i,\pm}) \subset S^i \setminus S^{i-1}.
\end{equation}
A point of $S^i \setminus S^{i-1}$ can't asymptotically flow to any critical point $v^{j,\pm}$ with $j<i$, because that would contradict \eqref{eq:stable-1} for that critical point. Hence, it must converge to $v^{i,\pm}$, where the sign is determined by the connected component of $S^i \setminus S^{i-1}$ in which it lies. This shows \eqref{eq:stable}, and the proof of \eqref{eq:unstable} is similar.
\end{proof}
\begin{figure}
\begin{centering}
\begin{picture}(0,0)%
\includegraphics{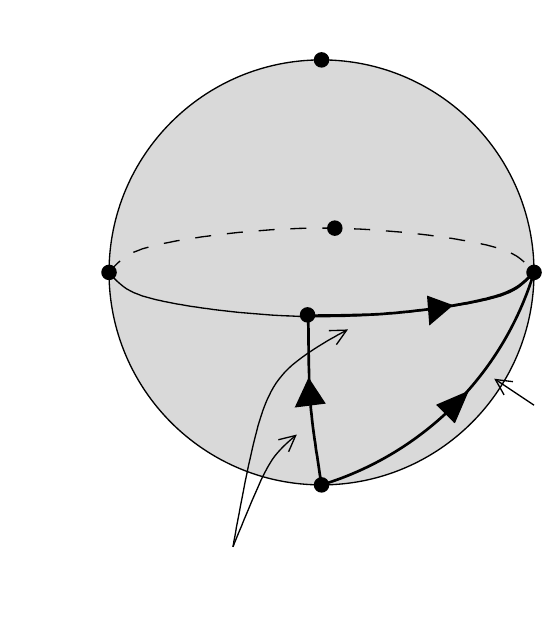}%
\end{picture}%
\setlength{\unitlength}{3355sp}%
\begingroup\makeatletter\ifx\SetFigFont\undefined%
\gdef\SetFigFont#1#2#3#4#5{%
  \reset@font\fontsize{#1}{#2pt}%
  \fontfamily{#3}\fontseries{#4}\fontshape{#5}%
  \selectfont}%
\fi\endgroup%
\begin{picture}(3105,3532)(586,-2522)
\put(602,-2266){\makebox(0,0)[lb]{\smash{{\SetFigFont{10}{12.0}{\rmdefault}{\mddefault}{\updefault}{\color[rgb]{0,0,0}a point in the stratum}%
}}}}
\put(3706,-511){\makebox(0,0)[lb]{\smash{{\SetFigFont{10}{12.0}{\rmdefault}{\mddefault}{\updefault}{\color[rgb]{0,0,0}$v^{0,+}$}%
}}}}
\put(751,-511){\makebox(0,0)[lb]{\smash{{\SetFigFont{10}{12.0}{\rmdefault}{\mddefault}{\updefault}{\color[rgb]{0,0,0}$v^{0,-}$}%
}}}}
\put(2401,-136){\makebox(0,0)[lb]{\smash{{\SetFigFont{10}{12.0}{\rmdefault}{\mddefault}{\updefault}{\color[rgb]{0,0,0}$v^{1,+}$}%
}}}}
\put(2401,-1936){\makebox(0,0)[lb]{\smash{{\SetFigFont{10}{12.0}{\rmdefault}{\mddefault}{\updefault}{\color[rgb]{0,0,0}$v^{2,-}$}%
}}}}
\put(2401,839){\makebox(0,0)[lb]{\smash{{\SetFigFont{10}{12.0}{\rmdefault}{\mddefault}{\updefault}{\color[rgb]{0,0,0}$v^{2,+}$}%
}}}}
\put(3676,-1336){\makebox(0,0)[lb]{\smash{{\SetFigFont{10}{12.0}{\rmdefault}{\mddefault}{\updefault}{\color[rgb]{0,0,0}a point in $\scrQ^{2,-}$}%
}}}}
\put(2101,-661){\makebox(0,0)[lb]{\smash{{\SetFigFont{10}{12.0}{\rmdefault}{\mddefault}{\updefault}{\color[rgb]{0,0,0}$v^{1.-}$}%
}}}}
\put(605,-2490){\makebox(0,0)[lb]{\smash{{\SetFigFont{10}{12.0}{\rmdefault}{\mddefault}{\updefault}{\color[rgb]{0,0,0}$\scrQ^{1,+} \times \scrQ^{1,-} \subset \bar{\scrQ}^{2,-}$}%
}}}}
\end{picture}%
\end{centering}
\caption{\label{fig:flow-lines}}
\end{figure}

For $i > 0$ and $\sigma \in \{+,-\}$, we define $\scrQ^{i,\sigma}$ to be the space of (unparametrized) trajectories of $-\nabla f$ connecting $v^{i,\sigma}$ to $v^{0,+}$ (see Figure \ref{fig:flow-lines}). To clarify the terminology, this is the space of solutions
\begin{equation}
\left\{
\begin{aligned}
& w: \bR \rightarrow S^\infty, \\
& dw/ds + \nabla f = 0, \\
& \textstyle \lim_{s \rightarrow -\infty} w(s) = v^{i,\sigma}, \\
& \textstyle \lim_{s \rightarrow \infty} w(s) = v^{0,+}, 
\end{aligned}
\right.
\end{equation}
modulo translation in $s$-direction. Equivalently in terms of \eqref{eq:stable} and \eqref{eq:unstable},
\begin{equation} \label{eq:rh-space}
\scrQ^{i,\sigma} \iso \big(W^u(v^{i,\sigma}) \cap W^s(v^{0,+})\big) / \bR
= \{\nu_{i+1} = \nu_{i+2} = \cdots = 0, \; \nu_0 > 0, \; \sigma \nu_i > 0 \} / \bR.
\end{equation}
The corresponding space of trajectories on $\bR P^\infty$, connecting $v^i$ to $v^0$, can be identified with the disjoint union of $\scrQ^{i,+}$ and $\scrQ^{i,-}$. The spaces $\scrQ^{i,\sigma}$ have standard compactifications $\bar{\scrQ}^{i,\sigma}$, obtained by adding broken flow lines. In our case, this can be written as 
\begin{equation} \label{eq:stratification}
\bar{\scrQ}^{i,\sigma} = \textstyle \bigsqcup\, \scrQ^{i_1,\sigma_1} \times \cdots \times \scrQ^{i_d,\sigma_d}.
\end{equation}
The union is over all partitions $i = i_1 + \cdots + i_d$ and collections of signs $\sigma_1,\dots, \sigma_d$ with $\sigma_1 \cdots \sigma_d = \sigma$. Just like $\scrQ^{i,\sigma}$, the compactification is independent of the choice of metric. To be more precise, let's say that $\bar\scrQ^{i,\sigma}$ is metric-independent as a compact topological space which comes with a decomposition into strata, and a smooth structure on each stratum (the word stratum is used here in an informal way, to refer to the subsets in \eqref{eq:stratification}; the topology of the space, its decomposition, and the smooth structure on each stratum are all independent of the choice of metric).

The $\bR$-action on the space on the right in \eqref{eq:rh-space} depends on the metric, and doesn't usually admit an elementary description. However, suppose that we specialize to the standard round metric on $S^\infty$. In that case, the gradient flow is a normalized linear flow: the unique flow line of $-\nabla f$ with $w(0) = (\nu_0,\nu_1,\dots) \in S^\infty$ is 
\begin{equation} \label{eq:exponential}
w(s) = (\nu_0,e^{-2s}\nu_1,e^{-4s}\nu_2,\dots)/\|(\nu_0,e^{-2s}\nu_1,e^{-4s}\nu_2,\dots)\|.
\end{equation}
Hence, every flow line $[w] \in \scrQ^{i,\sigma}$ can be parametrized in a unique way so that the coordinates of the point $w(0)$ satisfy $\nu_i = \sigma \nu_0$. By mapping $[w]$ to $(\nu_1/\nu_0,\dots,\nu_{i-1}/\nu_0)$, one gets an explicit diffeomorphism
\begin{equation}
\scrQ^{i,\sigma} \longrightarrow \bR^{i-1}.
\end{equation}

Even though this is the most elementary choice of metric, there is another possibility which offers some advantages. Let's say that the metric is standard near the critical points if the following holds:
\begin{equation} \label{eq:wehrheim}
\parbox{35em}{%\linespread{1.2}
Near each point $v^{i,\pm}$, there are local coordinates $\xi_j$ in which the metric is standard, and in which $f = \mathit{const} - \xi_1^2 - \cdots - \xi_i^2 + \xi_{i+1}^2 + \cdots$.
} 
\end{equation}
Such coordinates are easy to find in our case: on the (pairwise disjoint) subsets where $\pm \nu_i > 3/4$ for some $i$, use $|j-i|^{1/2} \nu_j$, $j \neq i$, as coordinates, and take the standard metric in this coordinates; and then extend that metric to the rest of $S^\infty$. As explained in \cite{burghelea-haller01,wehrheim12}, one can use such a metric to equip the spaces $\bar{\scrQ}^{i,\sigma}$ with the structure of a smooth manifold with corners. This is technically highly convenient: for instance, it allows one to construct strictly associative gluing maps which describe the neighbourhoods of the closure of each boundary stratum \cite{qin11, wehrheim12}.

\subsection{Parametrized flow lines}
In the same situation as before, consider the spaces $\scrP^{i,\sigma}$ of parametrized flow lines with limits $v^{i,\sigma}$ and $v^{0,+}$. Equivalently, one can view a parametrized flow line as an unparametrized flow line with one marked point on it (since then, there is a unique parametrization $w$ such that $w(0)$ is the marked point). This identifies $\scrP^{i,\sigma}$ with the intersection $W^u(v^{i,\sigma}) \cap W^s(v^{0,+})$. This time, $i$ is allowed to be zero, in which case $\scrP^{0,-} = \emptyset$ and $\scrP^{0,+} = \mathit{point}$ (corresponding to the constant flow line $w(s) = v^{0,+}$). For $i>0$, $\scrP^{i,\sigma}/\bR = \scrQ^{i,\sigma}$. The spaces of parametrized flow lines have standard compactifications
\begin{equation} \label{eq:stratification-2}
\bar{\scrP}^{i,\sigma} = \textstyle \bigsqcup\, \scrQ^{i_1,\sigma_1} \times \cdots \times \scrP^{i_j,\sigma_j} \times \cdots \times \scrQ^{i_d,\sigma_d}.
\end{equation}
Here, the union is over all partitions and signs as before, but with an additional distinguished choice of $j \in \{1,\dots,d\}$, and where $i_j$ can be zero. The zero-dimensional (corner) strata are  parametrized by $(\sigma_1,\dots,\sigma_{d+1}) \in \{\pm\}^d$ with $\sigma_1\cdots \sigma_{d+1} = \sigma$, together with a choice of $j \in \{1,\dots,d+1\}$ such that $\sigma_j = +$: there are $2^{d-1}(d+1)$ of them. The two-dimensional cases are shown in Figure \ref{fig:2d}, where the $\oplus$ in the labeling of the corners denotes the position of $j$. Similarly, Figure \ref{fig:3d} shows one of the three-dimensional cases (the other one can be obtained from that by switching the $+$ and $-$ labels, but keeping the $\oplus$). 
\begin{figure}
\begin{picture}(0,0)%
\includegraphics{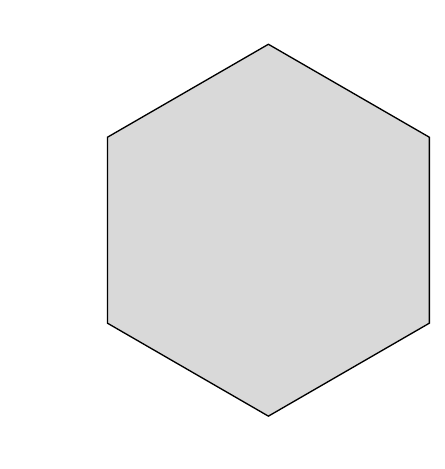}%
\end{picture}%
\setlength{\unitlength}{3355sp}%
\begingroup\makeatletter\ifx\SetFigFont\undefined%
\gdef\SetFigFont#1#2#3#4#5{%
  \reset@font\fontsize{#1}{#2pt}%
  \fontfamily{#3}\fontseries{#4}\fontshape{#5}%
  \selectfont}%
\fi\endgroup%
\begin{picture}(2505,2644)(286,-2009)
\put(1576,-661){\makebox(0,0)[lb]{\smash{{\SetFigFont{10}{12.0}{\rmdefault}{\mddefault}{\updefault}{$\bar{\scrP}^{2,+}$}%
}}}}
\put(2776,-1336){\makebox(0,0)[lb]{\smash{{\SetFigFont{10}{12.0}{\rmdefault}{\mddefault}{\updefault}{$\oplus ++$}%
}}}}
\put(2776,-136){\makebox(0,0)[lb]{\smash{{\SetFigFont{10}{12.0}{\rmdefault}{\mddefault}{\updefault}{$+\oplus+$}%
}}}}
\put(301,-136){\makebox(0,0)[lb]{\smash{{\SetFigFont{10}{12.0}{\rmdefault}{\mddefault}{\updefault}{$--\oplus$}%
}}}}
\put(301,-1336){\makebox(0,0)[lb]{\smash{{\SetFigFont{10}{12.0}{\rmdefault}{\mddefault}{\updefault}{$-\oplus-$}%
}}}}
\put(1576,-1936){\makebox(0,0)[lb]{\smash{{\SetFigFont{10}{12.0}{\rmdefault}{\mddefault}{\updefault}{$\oplus --$}%
}}}}
\put(1576,464){\makebox(0,0)[lb]{\smash{{\SetFigFont{10}{12.0}{\rmdefault}{\mddefault}{\updefault}{$++\oplus$}%
}}}}
\end{picture}%
\quad \quad \quad \quad \quad \quad
\begin{picture}(0,0)%
\includegraphics{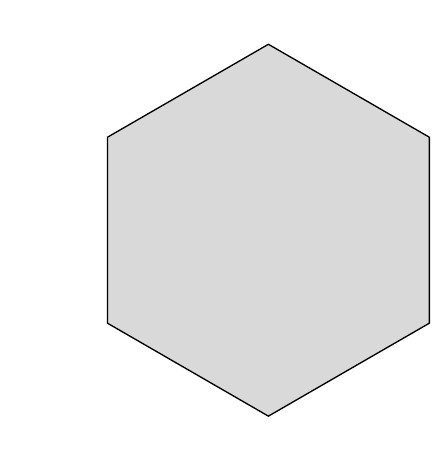}%
\end{picture}%
\setlength{\unitlength}{3355sp}%
\begingroup\makeatletter\ifx\SetFigFont\undefined%
\gdef\SetFigFont#1#2#3#4#5{%
  \reset@font\fontsize{#1}{#2pt}%
  \fontfamily{#3}\fontseries{#4}\fontshape{#5}%
  \selectfont}%
\fi\endgroup%
\begin{picture}(2505,2644)(286,-2009)
\put(2776,-1336){\makebox(0,0)[lb]{\smash{{\SetFigFont{10}{12.0}{\rmdefault}{\mddefault}{\updefault}{$\oplus -+$}%
}}}}
\put(2776,-136){\makebox(0,0)[lb]{\smash{{\SetFigFont{10}{12.0}{\rmdefault}{\mddefault}{\updefault}{$-\oplus+$}%
}}}}
\put(1576,-661){\makebox(0,0)[lb]{\smash{{\SetFigFont{10}{12.0}{\rmdefault}{\mddefault}{\updefault}{$\bar{\scrP}^{2,-}$}%
}}}}
\put(301,-136){\makebox(0,0)[lb]{\smash{{\SetFigFont{10}{12.0}{\rmdefault}{\mddefault}{\updefault}{$+-\oplus$}%
}}}}
\put(301,-1336){\makebox(0,0)[lb]{\smash{{\SetFigFont{10}{12.0}{\rmdefault}{\mddefault}{\updefault}{$+\oplus-$}%
}}}}
\put(1576,-1936){\makebox(0,0)[lb]{\smash{{\SetFigFont{10}{12.0}{\rmdefault}{\mddefault}{\updefault}{$\oplus +-$}%
}}}}
\put(1576,464){\makebox(0,0)[lb]{\smash{{\SetFigFont{10}{12.0}{\rmdefault}{\mddefault}{\updefault}{$-+\oplus$}%
}}}}
\end{picture}%
\caption{\label{fig:2d}}
\end{figure}%
\begin{figure}
\begin{picture}(0,0)%
\includegraphics{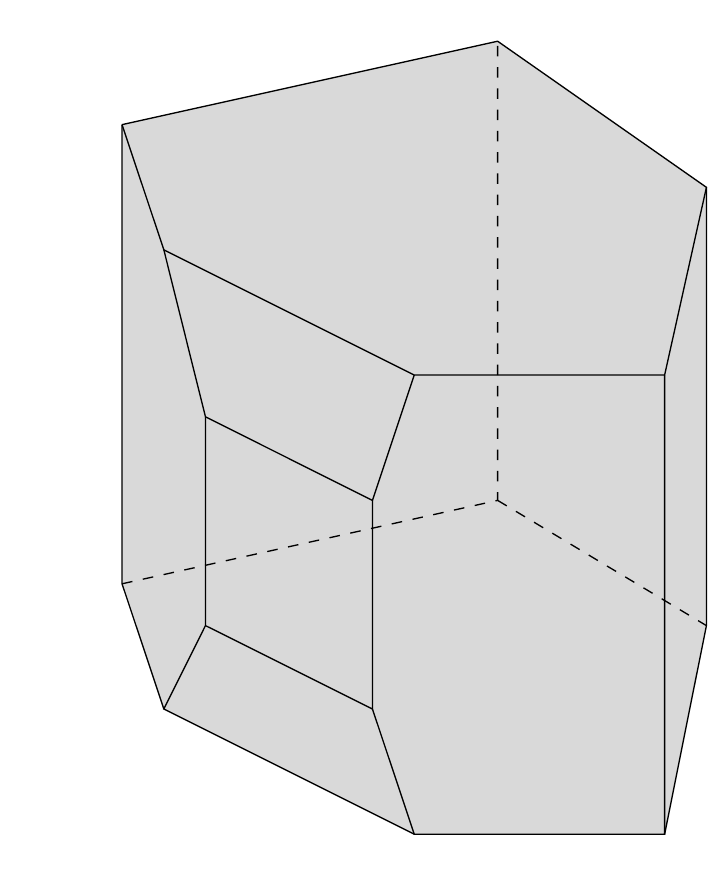}%
\end{picture}%
\setlength{\unitlength}{3355sp}%
\begingroup\makeatletter\ifx\SetFigFont\undefined%
\gdef\SetFigFont#1#2#3#4#5{%
  \reset@font\fontsize{#1}{#2pt}%
  \fontfamily{#3}\fontseries{#4}\fontshape{#5}%
  \selectfont}%
\fi\endgroup%
\begin{picture}(4080,4920)(961,-3121)
\put(3826,-1036){\makebox(0,0)[lb]{\smash{{\SetFigFont{8}{9.6}{\rmdefault}{\mddefault}{\updefault}{$+--\oplus$}%
}}}}
\put(2176,-436){\makebox(0,0)[lb]{\smash{{\SetFigFont{8}{9.6}{\rmdefault}{\mddefault}{\updefault}{$\oplus-+-$}%
}}}}
\put(1426,-2386){\makebox(0,0)[lb]{\smash{{\SetFigFont{8}{9.6}{\rmdefault}{\mddefault}{\updefault}{$+\oplus--$}%
}}}}
\put(4726,-3061){\makebox(0,0)[lb]{\smash{{\SetFigFont{8}{9.6}{\rmdefault}{\mddefault}{\updefault}{$++\oplus+$}%
}}}}
\put(5026,-1786){\makebox(0,0)[lb]{\smash{{\SetFigFont{8}{9.6}{\rmdefault}{\mddefault}{\updefault}{$+++\oplus$}%
}}}}
\put(5026,689){\makebox(0,0)[lb]{\smash{{\SetFigFont{8}{9.6}{\rmdefault}{\mddefault}{\updefault}{$--+\oplus$}%
}}}}
\put(3751,1664){\makebox(0,0)[lb]{\smash{{\SetFigFont{8}{9.6}{\rmdefault}{\mddefault}{\updefault}{$-+-\oplus$}%
}}}}
\put(1126,1214){\makebox(0,0)[lb]{\smash{{\SetFigFont{8}{9.6}{\rmdefault}{\mddefault}{\updefault}{$-+\oplus-$}%
}}}}
\put(3001,614){\makebox(0,0)[lb]{\smash{{\SetFigFont{10}{12.0}{\rmdefault}{\mddefault}{\updefault}{$\bar{\scrP}^{3,+}$}%
}}}}
\put(976,-1561){\makebox(0,0)[lb]{\smash{{\SetFigFont{8}{9.6}{\rmdefault}{\mddefault}{\updefault}{$+-\oplus-$}%
}}}}
\put(2401,-1111){\makebox(0,0)[lb]{\smash{{\SetFigFont{8}{9.6}{\rmdefault}{\mddefault}{\updefault}{$\oplus--+$}%
}}}}
\put(1951,464){\makebox(0,0)[lb]{\smash{{\SetFigFont{8}{9.6}{\rmdefault}{\mddefault}{\updefault}{$-\oplus+-$}%
}}}}
\put(3301,-211){\makebox(0,0)[lb]{\smash{{\SetFigFont{8}{9.6}{\rmdefault}{\mddefault}{\updefault}{$-\oplus-+$}%
}}}}
\put(4801,-361){\makebox(0,0)[lb]{\smash{{\SetFigFont{8}{9.6}{\rmdefault}{\mddefault}{\updefault}{$--\oplus+$}%
}}}}
\put(3151,-2236){\makebox(0,0)[lb]{\smash{{\SetFigFont{8}{9.6}{\rmdefault}{\mddefault}{\updefault}{$\oplus+++$}%
}}}}
\put(2176,-1636){\makebox(0,0)[lb]{\smash{{\SetFigFont{8}{9.6}{\rmdefault}{\mddefault}{\updefault}{$\oplus+--$}%
}}}}
\put(3226,-3061){\makebox(0,0)[lb]{\smash{{\SetFigFont{8}{9.6}{\rmdefault}{\mddefault}{\updefault}{$+\oplus++$}%
}}}}
\end{picture}%
\caption{\label{fig:3d}}
\end{figure}%

As before, $\bar\scrP^{i,\sigma}$ is independent of the metric (chosen as in Data \ref{th:setup-metric}). If additionally \eqref{eq:wehrheim} is satisfied, one can equip that space with the structure of a smooth manifold with corners. 

\begin{remark}
Strictly speaking, spaces of parametrized flow lines do not appear in the literature we have quoted previously. However, one can use the following trick to reduce the discussion to the unparametrized case. Consider $\bR \times S^\infty$ with a Morse function 
\begin{equation} \label{eq:prod-f}
(r,v) \longmapsto \psi(r) + f(v), 
\end{equation}
where $\psi$ has a nondegenerate minimum at $r = 0$ and maximum at $r = 1$, and with the product metric. Then, $\scrP^{i,\sigma}$ can be thought of as the space of unparametrized negative gradient flow lines for \eqref{eq:prod-f} connecting $(1,v^{i,\sigma})$ with $(0,v^{0,+})$ (the marked point on each such gradient flow line is the unique point where $r = 1/2$); and this identification extends to the compactifications.
\end{remark}

%\newpage
\section{Constructions\label{sec:con}}
This section introduces the main objects, namely $\mathit{HF}^*_{\mathit{eq}}(\phi^2)$ and the equivariant pair-of-pants product \eqref{eq:equi-pants}. Both constructions are based on parametrized moduli spaces. Generally speaking, the analytic aspects of such moduli spaces are quite well-known. Hence, we will only include a small amount of details, keeping the technical discussion focused on issues that are specific to this particular application.

\subsection{Review of Floer cohomology\label{subsec:floer}}
The following material is classical, and included in order to make the exposition self-contained. Take $\phi$ as in Setup \ref{th:setup-1}. Formally, the fixed point Floer cohomology of $\phi$ is the Morse cohomology of the action functional on the twisted free loop space. With $G_\phi$ as in \eqref{eq:exact-phi}, this is
\begin{align} \label{eq:phi-action-1}
& \scrL_\phi = \{x \in \smooth(\bR,M) \;:\; x(t) = \phi(x(t+1)) \}, \\
& A_\phi: \scrL_\phi \rightarrow \bR, \quad A_\phi(x) = \textstyle -\int_0^1 x^*\theta_M - G_\phi(x(1)). \label{eq:phi-action-2}
\end{align}
The critical points are constant $x \in \scrL_\phi$, which correspond to fixed points of $\phi$.

\begin{setup} \label{th:setup-3}
Throughout, we will use compatible almost complex structures $J$ on $M$ which satisfy 
\begin{equation} \label{eq:r-2}
dr_M\circ J = -\theta_M
\end{equation}
near the boundary. Property \eqref{eq:exact-phi} and its consequence \eqref{eq:r-1} ensure that \eqref{eq:r-2} is preserved under pushforward by $\phi$.
\end{setup}

Denote by $\scrJ_\phi$ the space of all families $J = (J_t)$ of almost complex structures parametrized by $t \in \bR$, which satisfy \eqref{eq:r-2} for all $t$, as well as the periodicity condition
\begin{equation} \label{eq:j-periodicity}
J_t = \phi_*(J_{t+1}).
\end{equation}
Formally, each such family defines an $L^2$ metric on $\scrL_\phi$, which one uses to define the gradient of the action functional. Choose a $J_\phi \in \scrJ_\phi$, and consider negative gradient flow lines connecting two fixed points $y$ and $x$. These are solutions of the Cauchy-Riemann equation (Floer's equation)
\begin{equation} \label{eq:floer}
\left\{
\begin{aligned}
& u: \bR^2 \longrightarrow M, \\
& u(s,t) = \phi(u(s,t+1)), \\
& \partial_s u + J_{\phi,t} \,\partial_t u = 0, \\
& \textstyle \lim_{s \rightarrow -\infty} u(s,\cdot) = y, \\
& \textstyle \lim_{s \rightarrow +\infty} u(s,\cdot) = x,
\end{aligned}
\right.
\end{equation}
up to translation in $s$-direction. Given a solution, consider the function $r_M(u)$. At all points where $u(s,t)$ is sufficiently close to $\partial M$, this function is $1$-periodic in $t$, and subharmonic. Given that, the maximum principle shows that $u$ can't reach $\partial M$, hence the fact that $M$ has a boundary is effectively irrelevant. Assuming that $J_\phi$ has been chosen generically, the moduli spaces $\scrM(y,x)$ of unparametrized Floer trajectories (non-constant solutions of \eqref{eq:floer}, up to translation in $s$-direction) are regular, hence smooth finite-dimensional manifolds. These manifolds can have connected components of different dimensions, but the parity of the dimension is always given by
\begin{equation}
\mathrm{dim}\, \scrM(y,x) \equiv |y|-|x|-1 \;\; \text{mod } 2, 
\end{equation}
where $|x| \in \bZ/2$ is determined by the local sign
\begin{equation}
(-1)^{|x|} = \mathrm{sign}\big(\mathrm{det}(I - D\phi_x)\big). \label{eq:mod2-degree}
\end{equation}
Moreover, each $\scrM(y,x)$ has only finitely many zero-dimensional components (isolated points). Denote the number of such points (mod 2) by $\#\scrM(y,x) \in \bK$.

\begin{definition}
The Floer cochain space is $\mathit{CF}^*(\phi) = \bigoplus_x \bK x$, where the sum is over fixed points, and the degree (mod $2$) of each generator is as in \eqref{eq:mod2-degree}. The differential is
\begin{equation} \label{eq:d-floer}
d_{J_\phi}(x) = \textstyle\sum_y \#\scrM(y,x) \, y.
\end{equation}
\end{definition}

For the application to $\phi^2$ (in Setup \ref{th:setup-2}), we find it convenient to slightly tweak this framework (the outcome is still equivalent to the original one). Given $G_\phi$, there is a natural choice of a corresponding function for $\phi^2$,
\begin{equation}
G_{\phi^2} = \phi^*G_\phi + G_\phi.
\end{equation}
We use the twisted loop space with period $2$, so the counterparts of \eqref{eq:phi-action-1}, \eqref{eq:phi-action-2} are
\begin{align} 
& \scrL_{\phi^2} = \{x \in \smooth(\bR,M) \;:\; x(t) = \phi^2(x(t+2)) \}, \\
& 
\begin{aligned}
A_{\phi^2}: \scrL_{\phi^2} \rightarrow \bR, \quad
A_{\phi^2}(x) & = \textstyle -\int_0^2 x^*\theta_M - G_{\phi^2}(x(2)) \\ &
=\textstyle -\int_0^1 x^*\theta_M - G_\phi(x(1)) - \int_1^2 x^*\phi^*\theta_M - G_\phi(\phi(x(2))).
\end{aligned}
\end{align}
The $\phi^2$-twisted loop space admits an involution
\begin{equation} \label{eq:rotate-loop}
\rho: \scrL_{\phi^2} \rightarrow \scrL_{\phi^2}, \quad
(\rho x)(t) = \phi(x(t+1)),
\end{equation} %\phi^2(x(t+2))
which preserves the action functional. The fixed point set of $\rho$ is exactly $\scrL_\phi$, and
\begin{equation}
A_{\phi^2}\,|\,\scrL_\phi = 2  A_\phi.
\end{equation}
There is a corresponding action on families of almost complex structures, 
\begin{align} \label{eq:j-periodicity-2}
& \scrJ_{\phi^2} = \{ J = (J_t) \;:\; J_{t} = \phi^2_*(J_{t+2}) \}, \\
& \rho_*: \scrJ_{\phi^2} \longrightarrow \scrJ_{\phi^2}, \quad
(\rho_*J)_t = \phi_*J_{t+1},
\end{align}
whose fixed point set is $\scrJ_\phi$. 

To define $\mathit{HF}^*(\phi^2)$, one chooses a generic $J_{\phi^2} \in \scrJ_{\phi^2}$, and then repeats the previous construction, except of course that the periodicity condition in \eqref{eq:floer} must be replaced by one involving $(s,t+2)$. In general, the genericity requirement means that it is impossible to choose $J_{\phi^2}$ to be invariant under \eqref{eq:j-periodicity-2}, so that choice breaks the existing symmetry. More concretely, while the space $\mathit{CF}^*(\phi^2)$ carries an involution given by $\rho$, or equivalently by the action of $\phi$ on the fixed points of $\phi^2$, that action will not usually be compatible with the differential. However, there is an involution on Floer cohomology, which we denote by 
\begin{equation} \label{eq:coh-iota}
\iota: \mathit{HF}^*(\phi^2) \longrightarrow \mathit{HF}^*(\phi^2).
\end{equation}
It is induced by the composition
\begin{equation} \label{eq:involution-on-hf}
(\mathit{CF}^*(\phi^2), d_{J_{\phi^2}}) \stackrel{\htp}{\longrightarrow}
(\mathit{CF}^*(\phi^2), d_{\rho_*J_{\phi^2}}) \stackrel{\rho}{\iso} (\mathit{CF}^*(\phi^2), d_{J_{\phi^2}}).
\end{equation}
Here, the middle group is the cohomology of the Floer complex formed with respect to the family $\rho_*J_{\phi^2}$ of almost complex structures. That complex is isomorphic to that for $J_{\phi^2}$, by applying $\rho$, which is the second part of \eqref{eq:involution-on-hf}. The first part is a continuation map, which is a quasi-isomorphism relating Floer complexes for different choices of almost complex structures: it is unique up to chain homotopy, hence induces a canonical isomorphism of cohomology groups. One can check (based on concatenation properties of continuation maps) that \eqref{eq:coh-iota} is indeed an involution. This is not true of \eqref{eq:involution-on-hf}, whose square is in general only chain homotopic to the identity.

\subsection{Equivariant Floer cohomology\label{subsec:equi-define}}
To define equivariant Floer cohomology, one introduces a family of almost complex structures which interpolates between $J_{\phi^2}$ and $\rho_*J_{\phi^2}$, and then extends that to higher-dimensional families. We choose to carry out the entire process in a single step, using the classical Borel construction as a model, as in \cite{seidel-smith10}.

\begin{data} \label{th:eq-data}
For each $v \in S^\infty$ choose a $J_{\mathit{eq},v} \in \scrJ_{\phi^2}$. This should depend smoothly on $v$, and have the following properties:
\begin{align}
\label{eq:j-symmetry} & J_{\mathit{eq},-v} = \rho_* J_{\mathit{eq},v}, \\
\label{eq:j-critical} & J_{\mathit{eq},v} = J_{\phi^2} \quad \text{if $v$ lies in a neighbourhood of $v^{i,+}$, for any $i$}, \\
& J_{\mathit{eq},\tau(v)} = J_{\mathit{eq},v}. \label{eq:j-shift}
\end{align}
\end{data}

Suppose that $w: \bR \rightarrow S^\infty$ is a non-constant negative gradient flow line (of the function $f$, with a metric as in Data \ref{th:setup-metric}), representing a point $[w] \in \scrQ^{i,\sigma}$. Our choice associates to $w$ a family of almost complex structures, namely
\begin{equation} \label{eq:j-trick}
J_{s,t} = J_{\mathit{eq},w(s),t}.
\end{equation}
This family satisfies
\begin{align}
& J_{s,t} = J_{\phi^2,t} \quad \text{for $s \gg 0$}, \\
& J_{s,t} = \left\{
\begin{aligned}
& J_{\phi^2,t} \;\; \text{if $\sigma = +$} \\
& (\rho_*J_{\phi^2})_t = \phi_*J_{\phi^2,t+1} \;\; \text{if $\sigma = -$}
\end{aligned} \right. \quad
\text{for $s \ll 0$.}
\end{align}
Using $J_{s,t}$, we write down a Cauchy-Riemann equation:
\begin{equation} \label{eq:parametrized-equation}
\left\{
\begin{aligned}
& u: \bR^2 \longrightarrow M, \\
& u(s,t) = \phi^2(u(s,t+2)), \\
& \partial_s u + J_{s,t} \, \partial_t u = 0, \\
& \textstyle \lim_{s \rightarrow +\infty} u(s,t) = x, \\
& \textstyle \lim_{s \rightarrow -\infty} u(s,t) = 
\left\{
\begin{aligned}
& y \;\; \text{if $\sigma = +$}, \\
& \phi(y) \;\; \text{if $\sigma = -$.}  
\end{aligned}
\right.
\end{aligned}
\right.
\end{equation}
Here, the limits $x$ and $y$ are fixed points of $\phi^2$. Note that \eqref{eq:parametrized-equation} is not invariant under $s$-translation of $u$, since the almost complex structures are $s$-dependent. However, it is compatible with simultaneous translation of $w$ and $u$. After dividing out by such translations, we get a moduli space of pairs $[w,u]$, denoted by $\scrM_{\mathit{eq}}^{i,\sigma}(y,x)$, which comes with a forgetful map
\begin{equation} \label{eq:eq-space}
\scrM_{\mathit{eq}}^{i,\sigma}(y,x) \longrightarrow \scrQ^{i,\sigma}.
\end{equation}
Since $\scrQ^{i,\sigma}$ is an $(i-1)$-manifold, and the fibre of \eqref{eq:eq-space} is the space of solutions of \eqref{eq:parametrized-equation} for a choice of almost complex structure determined by $w$, $\scrM_{\mathit{eq}}^{i,\sigma}(y,x)$ is a moduli space of pseudo-holomorphic maps depending on $(i-1)$ auxiliary parameters. For generic choice of almost complex structures, this space will be regular. As before, it can have components of different dimensions, but the parity of the dimension satisfies
\begin{equation} \label{eq:dim-mod-2}
\mathrm{dim}\, \scrM_{\mathit{eq}}^{i,\sigma}(y,x) \equiv |y| - |x| + i-1 \;\; \text{mod } 2.
\end{equation}
Proving generic regularity requires a transversality argument of a familiar kind. The other, and more substantial, technical part of any Floer-type construction are compactness and gluing arguments. Temporarily postponing the discussion of how those arguments work out in our situation, we jump ahead to the outcome:

\begin{definition}
By counting isolated points in the parametrized moduli spaces, define (for each $i>0$ and sign $\sigma$) maps
\begin{equation} \label{eq:i-differential}
\begin{aligned}
& d_{\mathit{eq}}^{i,\sigma}: \mathit{CF}^*(\phi^2) \longrightarrow \mathit{CF}^{*+1-i}(\phi^2), \\
& d_{\mathit{eq}}^{i,\sigma}(x) = \textstyle \sum_y \# \scrM_{\mathit{eq}}^{i,\sigma}(y,x)\, y.
\end{aligned}
\end{equation}
Set $d_{\mathit{eq}}^i = d_{\mathit{eq}}^{i,+} + d_{\mathit{eq}}^{i,-}$, and use that to define the differential on \eqref{eq:equi-cf}, by the formula
\begin{equation} \label{eq:equi-d}
d_{\mathit{eq}} = d_{{\phi^2}} + \sum_{i \geq 1} h^i d_{\mathit{eq}}^i.
\end{equation}
\end{definition}
%
%\begin{remark}
%As the reader may have noticed, we use the letter $u$ for the (purely algebraic) formal parameter arising in equivariant cohomology, as well as for pseudo-holomorphic maps. Because of the rather different nature of the two notions, this will hopefully not lead to confusion.
%\end{remark}

The operations \eqref{eq:i-differential} satisfy a series of equations, one for each $i>0$:
\begin{align} \label{eq:plus-relation}
& d_{J_{\phi^2}} d_{\mathit{eq}}^{i,+} + d_{\mathit{eq}}^{i,+} d_{J_{\phi^2}} = 
\sum_{\substack{i_1+i_2=i \\ i_1,i_2>0}} d_{\mathit{eq}}^{i_1,+} d_{\mathit{eq}}^{i_2,+} + d_{\mathit{eq}}^{i_1,-} d_{\mathit{eq}}^{i_2,-}, \\
\label{eq:minus-relation}
& d_{J_{\phi^2}} d_{\mathit{eq}}^{i,-} + d_{\mathit{eq}}^{i,-} d_{J_{\phi^2}} = 
\sum_{\substack{i_1+i_2=i \\ i_1,i_2>0}} d_{\mathit{eq}}^{i_1,-} d_{\mathit{eq}}^{i_2,+} + d_{\mathit{eq}}^{i_1,+} d_{\mathit{eq}}^{i_2,-}.
\end{align}
These imply that
\begin{equation} \label{eq:i-differential-sum}
d_{J_{\phi^2}} d_{\mathit{eq}}^i + d_{\mathit{eq}}^i d_{J_{\phi^2}} =
\sum_{\substack{i_1+i_2=i \\ i_1,i_2>0}} d_{\mathit{eq}}^{i_1} d_{\mathit{eq}}^{i_2},
\end{equation}
which is precisely the condition needed to show that \eqref{eq:equi-d} squares to zero. As an immediate consequence of the formal structure of \eqref{eq:equi-d}, one gets the desired analogue of \eqref{eq:u-sequence}, a long exact sequence of $\bK[[h]]$-modules
\begin{equation} \label{eq:floer-u-sequence}
\cdots \rightarrow \mathit{HF}^{*-1}_{\mathit{eq}}(\phi^2) \stackrel{h}{\longrightarrow} \mathit{HF}^*_{\mathit{eq}}(\phi^2) \longrightarrow \mathit{HF}^*(\phi^2) \rightarrow \cdots
\end{equation}

To understand \eqref{eq:i-differential}, it is instructive to look at the first order term in $h$. Lemma \ref{th:stable-unstable} implies that each space $\scrQ^{1,\pm}$ consists of a single unparametrized flow line, which means that we are looking at the space of solutions of a single equation \eqref{eq:parametrized-equation}. This is known as a continuation map equation \cite{salamon-zehnder92}, and a count of its solutions gives rise to a chain map between Floer complexes. More specifically, for $\sigma = +$ we get an endomorphism
\begin{equation} \label{eq:plus-1-map}
d_{\mathit{eq}}^{1,+}: (\mathit{CF}^*(\phi^2),d_{J_{\phi^2}}) \longrightarrow
(\mathit{CF}^*(\phi^2),d_{J_{\phi^2}}).
\end{equation}
Because of the uniqueness of continuation maps up to chain homotopy \cite[Lemma 6.3]{salamon-zehnder92}, this map is homotopic to the identity. %In fact, one could have chosen the almost complex structures $J_{\mathit{eq},v}$ so that $J_{s,t}^+$ is constant in $s$ everywhere, and then \eqref{eq:plus-1-map} would be strictly equal to the identity; but that is not necessary for our purpose.
In the other case $\sigma = -$, the continuation map provides the quasi-isomorphism from \eqref{eq:involution-on-hf}, which means that $d_{\mathit{eq}}^{1,-}$ is a chain map inducing the involution $\iota$ on $\mathit{HF}^*(\phi^2)$. We have therefore shown the following:

\begin{lemma} \label{th:u-adic-spectral-sequence}
Consider the spectral sequence associated to the $h$-adic filtration of $\mathit{CF}^*_{\mathit{eq}}(\phi^2)$. The $E_1$ page is $\mathit{HF}^*(\phi^2)[[h]]$, and the differential on it is $h(\mathit{id} + \iota)$. Hence, the $E_2$ page is $H^*(\bZ/2;\mathit{HF}^*(\phi^2))$. \qed
\end{lemma}

%Here, the term ``spectral sequence'' should be treated with a little caution, since we are dealing with a filtration of a $\bZ/2$-graded complex. 
The edge homomorphisms of the spectral sequence are canonical maps from $\mathit{HF}^*_{\mathit{eq}}(\phi^2)$ to the leftmost column $E_r^{0*}$ of each page ($r \geq 1$; the existence of these maps is independent of convergence issues for the spectral sequence). Specializing to $r = 2$, we get a map
\begin{equation}
\mathit{HF}^*_{\mathit{eq}}(\phi^2) \longrightarrow H^0(\bZ/2;\mathit{HF}^*(\phi^2)) = \mathit{HF}^*(\phi^2)^{\bZ/2}.
\end{equation}
By construction, this is a refinement of the forgetful map in \eqref{eq:floer-u-sequence}. This shows that the forgetful map lands in the $\bZ/2$-invariant part of $\mathit{HF}^*(\phi^2)$, a fact we have previously used to derive \eqref{eq:smith-2} (the language of spectral sequences is not really necessary in order to arrive at this conclusion; one can readily translate the argument into a more elementary form).

Let's turn to the more technical aspects, starting with transversality. Standard transversality arguments (compare e.g.\ \cite[Proposition 6.7.7]{mcduff-salamon}) suffice to prove the regularity of $\scrM_{\mathit{eq}}^{i,\sigma}(y,x)$ except at constant solutions, which have to be treated separately. The linearization of \eqref{eq:parametrized-equation} at a constant solution $u(s,t) = x$ is the operator
\begin{equation} \label{eq:constant-linearization}
D_u: \scrE^1 \rightarrow \scrE^0, \quad D_u(\xi) = \partial_s \xi + J_{s,t}\, \partial_t \xi.
\end{equation}
Here, the domain $\scrE^1$ is the space of maps $\xi: \bR^2 \rightarrow TM_x$ which are: locally $W^{k,p}$; globally $W^{k,p}$ when restricted to any strip $\bR \times (t_0,t_1)$; and satisfy $\xi(s,t) = D\phi_x^2(\xi(s,t+2))$. The range $\scrE^0$ is the same space with $W^{k-1,p}$ regularity. Because $1$ is not an eigenvalue of $D\phi_x^2$, $D_u$ is an elliptic operator. By (an easy special case of) the spectral flow formula, it has index $0$. Let's suppose for concreteness that $(k,p) = (2,2)$. One has (with respect to the metrics on $TM_x$ induced by $J_{s,t}$)
\begin{equation}\textstyle \
\int_{\bR \times [0,1]} \half |D_u \xi|^2 + \int_{\bR \times [0,1]} \xi^*\omega_{M,x} = \int_{\bR \times [0,1]} \half (|\partial_s \xi|^2 + |\partial_t \xi|^2). \label{eq:zero-energy-0}
\end{equation}
The second term on the left hand side integrates over the pullback of the constant two-form $\omega_{M,x}$ on $TM_x$, and one can show by a Stokes argument that it vanishes. With this in mind, \eqref{eq:zero-energy-0} implies that $D_u$ is injective, and therefore invertible. This shows that constant solutions of \eqref{eq:parametrized-equation} are always regular in the ordinary sense, hence a fortiori also regular in the parametrized sense.

\begin{addendum} \label{th:low-energy-differential}
For \eqref{eq:parametrized-equation} to have solutions, we must have
\begin{equation}
A_{\phi^2}(x) \leq A_{\phi^2}(y).
\end{equation}
More precisely: if equality holds, then the only solutions are constant ones (which means that necessarily $x = y$); whereas if the inequality is strict, all solutions are non-constant. Since the constant solutions exist for any choice of $J_{s,t}$, they form isolated points in $\scrM^{i,\sigma}_{\mathit{eq}}(y,x)$ only if $i = 1$. Hence,
\begin{equation} \label{eq:0-energy-term}
d_{\mathit{eq}} = h(\mathit{id} + \rho) + (\text{\it terms which increase the action}).
\end{equation}
A suitable filtration by action yields a spectral sequence converging to $\mathit{HF}^*_{\mathit{eq}}(\phi^2)$, whose $E_1$ page is $H^*(\bZ/2; \mathit{CF}^*(\phi^2))$, the group cohomology for the ``naive'' $\bZ/2$-action $\rho$ on $\mathit{CF}^*(\phi^2)$ (convergence of this spectral sequence is automatic, because the filtration is a finite one).

Even more interesting is the Tate version of the same spectral sequence, which converges to $\mathit{HF}^*_{\mathit{eq}}(\phi^2) \otimes_{\bK[[h]]} \bK((h))$. Let's divide $\mathit{CF}^*(\phi^2)$ into two pieces, one generated by  the fixed points of $\phi$, and the other by the points that have period exactly two. The Tate cohomology of the second summand vanishes by Example \ref{th:tate-0}. Hence, the $E_1$ page of this spectral sequence can be written as
\begin{equation} \label{eq:tate-e1page}
\hat{H}^*(\bZ/2; \mathit{CF}^*(\phi^2)) \iso \mathit{CF}^*(\phi)((h)).
\end{equation}
This isomorphism does not preserve the $\bZ/2$-grading, since the parity of $x$ as a fixed point of $\phi$ does not determine its counterpart for $\phi^2$ (see Section \ref{sec:linear} for more discussion of this). Similarly, a priori there appears to be no relation between the higher order differentials in the spectral sequence, acting on the left hand side of \eqref{eq:tate-e1page}, and the Floer differential on the right hand side. However, our proof of Theorem \ref{th:main} will show that they are related (but not in a way that's easy to describe explicitly).
% Our proof of Theorem \ref{th:main} will show that the higher differentials in this spectral sequence agree with the Floer differential for $\phi$.
\end{addendum}

Our final topic is compactness, where the argument is a version of that underlying the composition theorem for continuation maps \cite[Lemma 6.4]{salamon-zehnder92}. Suppose that we have a sequence $[w^k, u^k] \in \scrM^{i,\sigma}_{\mathit{eq}}(y,x)$, such that the Morse-theoretic gradient flow lines $[w^k]$ converge to a point of the compactification \eqref{eq:stratification}. Let's denote the components of the limit point by $([w^\infty_1],\dots,[w^\infty_d])$. In more geometric terms, this limit would be the broken Morse trajectory consisting of
\begin{equation} \label{eq:broken-morse}
(\sigma_2 \cdots \sigma_d) \tau^{i_2+\cdots+i_d}(w^\infty_1),\,\dots, \, (\sigma_{d-1} \sigma_d) \tau^{i_{d-1}+i_d}(w^\infty_{d-2}),\, \sigma_d \tau^{i_d}(w_{d-1}^\infty), \; w^\infty_d
\end{equation}
(here, the $(\pm)$ sign denotes the $\bZ/2$-action on $S^\infty$, and $\tau$ the shift; the special case $d = 1$ corresponds to convergence inside $\scrQ^{i,\sigma}$ itself). Even more explicitly, for each component $[w_j^\infty]$ of the limit, we have a sequence $s^k_j \in \bR$ such that the reparametrized gradient flow lines $\tilde{w}^k_j = w^k(s-s^k_j)$ satisfy
\begin{equation}
\tilde{w}^k_j(s) \longrightarrow (\sigma_{j+1} \cdots \sigma_d) \tau^{i_{j+1} + \cdots + i_d}(w_j^\infty(s))
\end{equation}
(uniformly on compact subsets). Suppose first that $\sigma_{j+1} \cdots \sigma_d = +$. If we consider the corresponding sequence of reparametrized solutions $\tilde{u}^k_j(s,t) = u(s - s^k_j,t)$, they satisfy an equation
\begin{equation}
\partial_s \tilde{u}^k_j + \tilde{J}_{j,s,t}^k \partial_t \tilde{u}^k_j = 0,
\end{equation}
where $\tilde{J}_{j,s,t}^k = J_{\mathit{eq},\tilde{w}^k_j(s),t}$ converges (on compact subsets) to the family of almost complex structures defining the Cauchy-Riemann equation \eqref{eq:parametrized-equation} associated to $w_j^\infty$. Bubbling being ruled out by the exactness assumptions, it follows that a subsequence of the $\tilde{u}^k_j$ converges to a $u_j^\infty$ such that $[w_j^\infty,u_j^\infty] \in \scrM_{\mathit{eq}}^{i_k,\sigma_k}(y_j,x_j)$ (for some limits $y_j$ and $x_j$). In the other case $\sigma_{j+1} \cdots \sigma_d = -$, the same convergence result applies up to an involution (replacing $\tilde{J}_{j,s,t}^k$ by $(\rho_*\tilde{J}_{j,s}^k)_t$, and $\tilde{u}^k_j$ by $\rho(\tilde{u}^k_j)$.

In general, the components $[w^\infty_j,u^\infty_j]$ obtained in this way do not characterize the limiting behaviour completely. There will be further components, which are ordinary Floer trajectories \eqref{eq:floer}, appearing either before the $j = 1$ component, after the $j = d$ component, or in between any two such components. After including such Floer trajectories, one obtains the desired compactification $\bar{\scrM}_{\mathit{eq}}^{i,\sigma}(y,x)$, to which a parametrized version of the Floer-theoretic gluing theory can be applied (see \cite[Section 4.4]{schwarz95} or \cite[Section 3.3]{salamon} for the gluing theorem; the parametrized version, where families of Cauchy-Riemann equations are considered, appeared first in the proof of uniqueness up to homotopy of continuation maps, \cite[Lemma 6.3]{salamon-zehnder92} or \cite[Lemma 3.12]{salamon}).

The compactness theorem (together with transversality) implies that $\scrM_{\mathit{eq}}^{i,\sigma}(y,x)$ has only finitely many isolated points. The other relevant special case is that of a sequence of points $[w^k,u^k]$ which lie in the one-dimensional part of $\scrM_{\mathit{eq}}^{i,\sigma}(y,x)$. Here, the only possible limits in $\bar\scrM_{\mathit{eq}}^{i,\sigma}(y,x) \setminus \scrM_{\mathit{eq}}^{i,\sigma}(y,x)$ are of the following kinds. One can have convergence in $\scrQ^{i,\sigma}$ and exactly one Floer trajectory appearing, which accounts for the terms on the left-hand side of \eqref{eq:plus-relation}, \eqref{eq:minus-relation}. Or else, one can have convergence to a codimension one stratum of $\bar\scrQ^{i,\sigma}$, which means $d = 2$ in \eqref{eq:broken-morse}, with no Floer trajectories appearing. In the latter case, the two pieces of the limit have the form
\begin{equation}
[w_1^\infty,u_1^\infty] \in \scrM_{\mathit{eq}}^{i_1,\sigma_1}(y,z), \;\; [w_2^\infty,u_2^\infty] \in \scrM_{\mathit{eq}}^{i_2,\sigma_2}(z,x)
\end{equation}
for $i_1+i_2 = i$ and $\sigma_1\sigma_2 = \sigma$. Moreover, they must be isolated points of their respective moduli spaces. The resulting contributions (for the two possible choices of $\sigma_1$, $\sigma_2$) make up the right hand side of \eqref{eq:plus-relation}, \eqref{eq:minus-relation}.

\subsection{The equivariant product\label{subsec:product}}
We will work with a specific model for the pair-of-pants (the three-punctured sphere) $S$, as the double cover
\begin{equation} \label{eq:s-projection}
\pi: S \longrightarrow \bR \times S^1 = \bR \times \bR/\bZ
\end{equation}
branched over the point $(0,0) \in \bR \times S^1$. To fully specify \eqref{eq:s-projection}, we should say that the covering must be trivial over the end $s > 0$ of $\bR \times S^1$ (hence nontrivial over the other end $s < 0$). Denote the covering involution by $\gamma: S \rightarrow S$. By assumption, we can find two embeddings
\begin{equation} \label{eq:epsilon-in}
\begin{aligned}
& \delta^{\pm}: [1,\infty) \times S^1 \longrightarrow S, \\
& \pi(\delta^{\pm}(s,t)) = (s,t), \quad
\gamma(\delta^{\pm}(s,t)) = \delta^{\mp}(s,t).
\end{aligned}
\end{equation}
%whose images are disjoint, and together cover $\pi^{-1}([1,\infty) \times S^1) \subset S$. 
Similarly, there is an embedding
\begin{equation} \label{eq:epsilon-out}
\begin{aligned}
& \epsilon^+: (-\infty,-1] \times \bR/2\bZ \longrightarrow S, \\
&
\pi(\epsilon^+(s,t)) = (s,t), \quad
\gamma(\epsilon^+(s,t)) = \epsilon^+(s,t+1),
\end{aligned}
\end{equation}
%whose image is $\pi^{-1}((-\infty,-1] \times S^1) \subset S$. 
For symmetry reasons, we also consider $\epsilon^-(s,t) = \epsilon^+(s,t+1)$, which gives a different parametrization of the same end. The embeddings \eqref{eq:epsilon-in}, \eqref{eq:epsilon-out} are not quite unique (one could exchange $\delta^+$ with $\delta^-$, and correspondingly for the $\epsilon$'s), but we assume that a choice has made been made once and for all.

% zeta^2 = x(x-1) = x^2-x
% is branched at x = 0,1
% equivalently x^2\zeta^2 = x^2 - x => \zeta^2 = 1- 1/x 
\begin{remark} \label{th:coordinates}
If one prefers explicit coordinates, one can set
\begin{align}
& S = \big\{(s,t,\zeta) \in \bR \times S^1 \times \bC \;:\; \zeta^2 = 1 - \exp(-2\pi(s+it))\big\}, \\
& \gamma(s,t,\zeta) = (s,t,-\zeta).
\end{align}
Then
\begin{align}
\label{eq:explicit-delta} & \delta^{\pm}(s,t) = \big(s,t,\pm \sqrt{1-\exp(-2\pi(s+it))} \big), \\
\label{eq:explicit-epsilon} & \epsilon^{\pm}(s,t) = \big(s,t, \pm e^{-\pi (s+it)} \sqrt{\exp(2 \pi (s+it)) - 1} \big).
\end{align}
In \eqref{eq:explicit-delta}, we have arbitrarily chosen a branch of the complex square root on the
open unit disc around $1$; and in \eqref{eq:explicit-epsilon}, the same for $-1$.
\end{remark}

Take the covering $\bR^2 \rightarrow \bR \times S^1$, and pull it back via \eqref{eq:s-projection}. The outcome is a covering $\tilde{S} \rightarrow S$, whose covering group is generated by an automorphism $\theta$. Then, \eqref{eq:s-projection} lifts to a double branched covering
\begin{equation} \label{eq:tilde-pi-map}
\tilde{\pi}: \tilde{S} \longrightarrow \bR^2,
\end{equation}
with covering involution $\tilde{\gamma}$ which commutes with $\theta$. The previously defined maps $\delta^\pm$, $\epsilon^\pm$ admit lifts
\begin{equation} \label{eq:tilde-delta-properties}
\begin{aligned}
& \tilde{\delta}^{\pm}: [1,\infty) \times \bR \longrightarrow \tilde{S}, \\ &
\tilde{\pi}(\tilde{\delta}^{\pm}(s,t)) = (s,t), \quad
\theta(\tilde{\delta}^{\pm}(s,t)) = \tilde{\delta}^{\pm}(s,t+1), \quad
\tilde{\gamma}(\tilde{\delta}^{\pm}(s,t)) = \tilde{\delta}^{\mp}(s,t),
\end{aligned}
\end{equation}
and
\begin{equation} \label{eq:tilde-epsilon-properties}
\begin{aligned}
& \tilde{\epsilon}^{\pm}: (-\infty,-1] \times \bR \longrightarrow \tilde{S}, \\
& \tilde{\pi}(\tilde{\epsilon}^{\pm}(s,t)) = (s,t), \quad
\theta(\tilde{\epsilon}^{\pm}(s,t)) = \tilde{\epsilon}^{\mp}(s,t+1), \quad
\tilde{\gamma}(\tilde{\epsilon}^{\pm}(s,t)) = \tilde{\epsilon}^{\mp}(s,t).
\end{aligned}
\end{equation}
Note that $\tilde{\epsilon}^+$ and $\tilde{\epsilon}^-$ have disjoint images, which together cover the preimage of the end \eqref{eq:epsilon-out} under $\tilde{\pi}$.
% \theta\tilde{epsilon}^+(s,t) = (s,t+1, e^{-...}) = \tilde{\theta}^- !
% \theta\tilde{epsilon}^-(s,t) = (s,t+2,...) = 

\begin{remark}
In the model from Remark \ref{th:coordinates}, 
\begin{align}
& \tilde{S} = \big\{(s,t,\zeta) \in \bR^2 \times \bC \;:\; \zeta^2 = 1 - \exp(-2\pi(s+it))\big\}, \\
& \theta(s,t,\zeta) = (s,t+1,\zeta), \\
& \tilde\gamma(s,t,\zeta) = (s,t,-\zeta).
\end{align}
The maps \eqref{eq:tilde-delta-properties} and \eqref{eq:tilde-epsilon-properties} are defined by the same formulae \eqref{eq:explicit-delta}, \eqref{eq:explicit-epsilon} as before.
\end{remark}

\begin{data} \label{th:j-prod-data}
For each $v \in S^\infty$ and $s < 1$, choose almost complex structures $J_{\mathit{left},v,s} \in \scrJ_{\phi^2}$ with the following properties:
\begin{align}
& J_{\mathit{left},-v,s} = \rho_* J_{\mathit{left},v,s}, \label{eq:jeq-1} \\
& J_{\mathit{left},\tau(v),s} = J_{\mathit{left},v,s}, \label{eq:left-tau} \\
& J_{\mathit{left},v,s} = J_{\mathit{eq},v} \quad \text{if $s \leq -2$}, \label{eq:prod-eq} \\
& J_{\mathit{left},v,s} \in \scrJ_\phi \quad \text{if $s \geq -1$}. \label{eq:minus1-bound}
\end{align}
In addition, for $v \in S^\infty$ and $s>-1$, choose $J_{\mathit{right},v,s}^{\pm} \in \scrJ_{\phi}$, such that:
\begin{align}
& J_{\mathit{right},-v,s}^{\pm} = J_{\mathit{right},v,s}^{\mp}, \label{eq:right-swap} \\
& J_{\mathit{right},\tau(v),s}^{\pm} = J_{\mathit{right},v,s}^{\pm}, \label{eq:right-tau} \\
& J_{\mathit{right},v,s}^{\pm} = J_\phi \quad \text{ if $s \geq 2$}, \label{eq:2-bound} \\
& J_{\mathit{right},v,s}^{\pm} = J_{\mathit{left},v,s} \quad \text{if $s \leq 1$}. \label{eq:plus1-bound}
\end{align}
\end{data}

Let $w: \bR \rightarrow S^\infty$ be a negative gradient trajectory of $f$ which corresponds to a point in $\scrP^{i,\sigma}$, meaning that it connects $v^{i,\sigma}$ to $v^{0,+}$. To this, we associate a family $J_z$ of almost complex structures parametrized by $z \in \tilde{S}$, as follows:
\begin{align} 
\label{eq:middleregion} 
& \text{if $\tilde{\pi}(z) = (s,t)$ with $-1 \leq s \leq 1$, set } 
J_z = J_{\mathit{left},w(s),s,t} = J_{\mathit{right},w(s),s,t}^{\pm}\,; \\
\label{eq:left+region} 
& \text{if $z = \tilde{\epsilon}^+(s,t)$, set } 
J_z = J_{\mathit{left},w(s),s,t}\,; \\
\label{eq:left-region} 
& \text{if $z = \tilde{\epsilon}^-(s,t)$, set } 
J_z = J_{\mathit{left},-w(s),s,t}\,; \\
\label{eq:right+region} 
& \text{if $z = \tilde{\delta}^+(s,t)$, set }
J_z = J_{\mathit{right},w(s),s,t}^+\,; \\
\label{eq:right-region}
& \text{if $z = \tilde{\delta}^-(s,t)$, set }
J_z = J_{\mathit{right},w(s),s,t}^-\,.
\end{align}
This makes sense thanks to \eqref{eq:minus1-bound} and \eqref{eq:plus1-bound}, which imply that along $s = \pm 1$, \eqref{eq:middleregion} matches up smoothly with the other prescriptions. 

\begin{lemma} \label{th:st}The family $(J_z)$ has the following properties:
\begin{align}
& \label{eq:twisted-j-period} J_z = \phi_*J_{\theta(z)}, \\
& \label{eq:delta-j} J_{\tilde{\delta}^{\pm}(s,t)} = J_{\phi,t} \quad \text{for $s \geq 2$}, \\
\label{eq:epsilon-j-plus} 
& J_{\tilde{\epsilon}^{\sigma}(s,t)} = J_{\phi^2,t} \quad \text{for $s \ll 0$,} \\
\label{eq:epsilon-j-minus}
& J_{\tilde{\epsilon}^{-\sigma}(s,t)} = (\rho_*J_{\phi^2})_t \quad \text{for $s \ll 0$.}
\end{align}
\end{lemma}

\begin{proof}
The verification of \eqref{eq:twisted-j-period} breaks up into the following cases:
\begin{align}
& \text{if $\tilde{\pi}(z) = (s,t)$ with $-1 \leq s \leq 1$, }
\phi_*J_{\theta(z)} = \phi_*J_{\mathit{right},w(s),s,t+1}^{\pm} = J_{\mathit{right},w(s),s,t}^{\pm} \label{eq:verify-1}\,; \\
\label{eq:verify-2}
& \parbox{35em}{%\linespread{1.2}
if $z = \tilde{\epsilon}^\pm(s,t)$, \newline
$\phi_*J_{\theta\tilde{\epsilon}^\pm(s,t)} = \phi_*J_{\tilde{\epsilon}^{\mp}(s,t+1)} = \phi_*J_{\mathit{left},\mp w(s),s,t+1} = (\rho_*J_{\mathit{left}, \mp w(s),s})_t = J_{\mathit{left},\pm w(s),s,t}$;} \\
& \text{if $z = \tilde{\delta}^{\pm}(s,t)$, }
\phi_*J_{\theta\tilde{\delta}^{\pm}(s,t)} = \phi_*J_{\tilde{\delta}^{\pm}(s,t+1)} = \phi_*J_{\mathit{right},w(s),s,t+1}^{\pm} =
J_{\mathit{right},w(s),s,t}^{\pm}. 
\label{eq:verify-3}
\end{align}
Here, \eqref{eq:verify-1} and \eqref{eq:verify-3} use the fact that $J_{\mathit{right},v,s}^{\pm} \in \scrJ_\phi$, while \eqref{eq:verify-2} uses \eqref{eq:jeq-1}. Next, \eqref{eq:delta-j} is a direct consequence of \eqref{eq:2-bound}. As for \eqref{eq:epsilon-j-plus}, note that for $s \ll 0$, $w(s)$ is close to $v^{i,\sigma}$, hence $\sigma w(s)$ is close to $v^{i,+}$. Using \eqref{eq:prod-eq}, \eqref{eq:j-shift} and \eqref{eq:j-critical}, one therefore gets
\begin{equation}
J_{\tilde{\epsilon}^{\sigma}(s,t)} = J_{\mathit{left},\sigma w(s),s,t} = J_{\mathit{eq},\sigma w(s), t} = J_{\phi^2,t}.
\end{equation}
The final property \eqref{eq:epsilon-j-minus} follows from \eqref{eq:epsilon-j-plus} and \eqref{eq:twisted-j-period}.
\end{proof}

Given any family of almost complex structures $(J_z)$ satisfying the properties from Lemma \ref{th:st}, one can consider the pair-of-pants product equation
\begin{equation} \label{eq:pants-equation}
\left\{
\begin{aligned}
& u: \tilde{S} \longrightarrow M, \\
& u(z) = \phi(u(\theta(z))), \\
& du \circ j = J_z \circ du, \\
& \textstyle \lim_{s \rightarrow -\infty} u({\tilde\epsilon}^\sigma(s,t)) = y, \\
& \textstyle \lim_{s \rightarrow +\infty} u({\tilde\delta}^+(s,t)) = x^+, \\
& \textstyle \lim_{s \rightarrow +\infty} u(\tilde{\delta}^-(s,t)) = x^-.
\end{aligned}
\right.
\end{equation}
Here, $j$ is the complex structure on $S$, $y$ is a fixed point of $\phi^2$, and the $x^\pm$ are fixed points of $\phi$. Note that
\begin{equation} \label{eq:u-sigma}
\phi(u(\tilde\epsilon^{\sigma}(s,t+1))) = \phi(u(\theta(\tilde\epsilon^{-\sigma}(s,t)))) = u(\epsilon^{-\sigma}(s,t)).
\end{equation}
In particular, one also has
\begin{equation}
\textstyle \lim_{s \rightarrow -\infty} u({\tilde\epsilon}^{-\sigma}(s,t)) = \phi(y).
\end{equation}

Unfortunately, transversality fails for solutions of \eqref{eq:pants-equation}. The culprit is the constant map $u(z) = x$, where $x$ is a fixed point of $\phi$. This is a solution of \eqref{eq:pants-equation} for any choice of $J_z$. Unlike the constant solutions of \eqref{eq:parametrized-equation}, these ones may have negative virtual dimension (we will discuss the relevant index theory in more detail later on, see Lemma \ref{th:index-theory}), hence won't be regular in general. While one could remedy this by applying virtual perturbation theory, we prefer the older approach using an explicit inhomogeneous term.

\begin{data} \label{th:inhomogeneous-data} 
Denote by $\scrH_{\phi}$ the space of all functions $H = H_t(x): \bR \times M \rightarrow \bR$ which vanish near $\partial M$, and which satisfy $H_t = \phi_*H_{t+1}$, meaning that
\begin{equation}
H_t(x) = H_{t+1}(\phi^{-1}(x)).
\end{equation}
Choose a family $H_s \in \scrH_{\phi}$ depending on another parameter $s \in \bR$, and whose support in $s$-direction lies inside the interval $(1,2)$. Write $X_{s,t}$ for the Hamiltonian vector field of $H_{s,t}$.
\end{data}

This choice equips the surface $\tilde{S}$ with an inhomogeneous term $Y$, which is a one-form on $\tilde{S}$ with values in Hamiltonian vector fields on $M$. Namely, $Y$ vanishes outside the image of $\tilde{\delta}^{\pm}$, and satisfies
\begin{equation}
(\tilde{\delta}^{\pm})^* Y = X_{s,t} \otimes dt.
\end{equation}
Note that by definition,
\begin{align}
& Y = \phi_* (\theta^*Y), \\
& Y = \tilde{\gamma}^*Y. \label{eq:gamma-y}
\end{align}
Given this, we perturb \eqref{eq:pants-equation} to an inhomogeneous Cauchy-Riemann equation
\begin{equation} \label{eq:perturbed-equation}
(du - Y_z) \circ j = J_z \circ (du - Y_z). 
\end{equation}
More concretely, this means that $u \circ \tilde{\delta}^{\pm}: [1,\infty) \times \bR \rightarrow M$ are solutions of 
\begin{equation} \label{eq:inhomogeneous-pullback}
\partial_s (u \circ \tilde{\delta}^{\pm}) + J_{\tilde{\delta}^{\pm}(s,t)}\big(
\partial_t (u \circ \tilde{\delta}^{\pm}) - X_{s,t}\big) = 0,
\end{equation}
while over the rest of the Riemann surface the equation remains as before. We should explain how this solves the transversality problem mentioned above. Note that inside the region $s \in (1,2)$, one can vary the almost complex structures 
\begin{align}
& J_{\tilde{\delta}^+(s,t)} = J_{\mathit{right},w(s),s,t}^+, \\
& J_{\tilde{\delta}^-(s,t)} = J_{\mathit{right},w(s),s,t}^- = J_{\mathit{right},-w(s),s,t}^+
\end{align}
freely, and independently of each other in the $+$ and $-$ cases (independence holds since $(w(s^+),s^+) \neq (-w(s^-),s^-)$ for any $s^\pm$). The only solutions $u$ for which transversality can't be achieved by such a variation of almost complex structure are those which satisfy
\begin{equation} \label{eq:zero-energy}
\partial_s (u \circ \tilde{\delta}^{\pm}) = 0 \quad \text{for all } (s,t) \in (1,2) \times \bR,
\end{equation}
or equivalently
\begin{equation}
\partial_t (u \circ \tilde{\delta}^{\pm}) = X_{s,t} \quad \text{for all } (s,t) \in (1,2) \times \bR.
\end{equation}
By continuity, such a solution $u$ is constant along the circles $s = 1,2$, hence (by unique continuation) constant over the part of the Riemann surface $\tilde{S}$ where \eqref{eq:zero-energy} does not apply. It follows that $u$ must be constant overall, with its value being a fixed point of $\phi$. But one can choose $H$ so that $X_{s,t}$ does not vanish identically at any of those fixed points, and then there are no such solutions.

\begin{addendum} \label{th:flip-the-sign}
Let's temporarily write $\tilde{J}_z$ for the family given by applying the same formulae to $\tilde{w}(s) = -w(s)$ (which is a flow line of $-\nabla f$ going from $v^{i,-\sigma}$ to $v^{0,-}$). Then,
\begin{equation}
\tilde{J}_{\tilde{\gamma}(z)} = J_z.
\end{equation}
To see this, note that
\begin{align}
\label{eq:add-1}
& \text{if $\tilde{\pi}(z) = (s,t)$ with $-1 \leq s \leq 1$, }
\tilde{J}_{\tilde{\gamma}(z)} = J_{\mathit{left},-w(s),s,t} = (\rho_*J_{\mathit{left},w(s),s})_t
= J_{\mathit{left},w(s),s,t}\, ; \\
\label{eq:add-2}
& \text{if $z = \tilde{\epsilon}^\pm(s,t)$, we have $\gamma(z) = \tilde{\epsilon}^\mp(s,t)$, hence } 
\tilde{J}_{\tilde{\gamma}(z)} = J_{\mathit{left},\mp (-w(s)),s,t} = J_{\mathit{left},\pm w(s),s,t}\, ; \\
\label{eq:add-3}
& \text{if $z = \tilde{\delta}^{\pm}(s,t)$, we have $\gamma(z) = \tilde{\delta}^{\mp}(s,t)$, hence }
\tilde{J}_{\tilde{\gamma}(z)} = J_{\mathit{right},-w(s),s,t}^{\mp} = J_{\mathit{right},w(s),s,t}^{\pm}.
\end{align}
Here, \eqref{eq:add-1} uses \eqref{eq:jeq-1} and \eqref{eq:minus1-bound}; \eqref{eq:add-2} reduces to a tautology; and \eqref{eq:add-3} uses \eqref{eq:right-swap}. Because of this and \eqref{eq:gamma-y}, the equation \eqref{eq:perturbed-equation} for the family $J$ and its counterpart for $\tilde{J}$ are related by a coordinate change $u \mapsto u \circ \tilde{\gamma}$.
\end{addendum}

We denote by $\scrM_{\mathit{prod}}^{i,\sigma}(y,x^+,x^-)$ the moduli space of pairs $(w,u)$, where $w \in \scrP^{i,\sigma}$, and $u$ is a solution of the perturbed version \eqref{eq:perturbed-equation} of \eqref{eq:pants-equation}. These moduli spaces are generically smooth, and in the same sense as in \eqref{eq:dim-mod-2}, one has
\begin{equation} \label{eq:dimension-formula-prod}
\mathrm{dim}\,  \scrM_{\mathit{prod}}^{i,\sigma}(y,x^+,x^-) \equiv |y| - |x^+| - |x^-| + i
\;\; \mathrm{mod}\; 2.
\end{equation}
There is a natural compactification $\bar{\scrM}_{\mathit{prod}}^{i,\sigma}(y,x^+,x^-)$, whose construction proceeds along familiar lines (it is a parametrized version of the classical construction underlying the pair-of-pants product \cite{schwarz95,salamon99}). Rather than writing this out fully, we consider its implications for the operations defined by counting isolated points in our spaces.

These operations have the form
\begin{equation} 
\begin{aligned}
& \wp^{i,\sigma}: \mathit{CF}^*(\phi) \otimes \mathit{CF}^*(\phi) \longrightarrow \mathit{CF}^{*-i}(\phi^2), \\ \textstyle & 
\wp^{i,\sigma}(x^+, x^-) = \sum_y \#\scrM_{\mathit{prod}}^{i,\sigma}(y,x^+,x^-) \, y,
\end{aligned}
\end{equation}
for $i \geq 0$ and $\sigma = \pm$, with one trivial case:
\begin{equation} \label{eq:0-}
\wp^{0,-} = 0.
\end{equation}
Their fundamental properties are
\begin{align} \label{eq:p-plus-relation}
& \begin{aligned}
& d_{J_{\phi^2}} \wp^{i,+}(x^+,x^-) + \wp^{i,+}(d_{J_{\phi}} x^+, x^-) + \wp^{i,+}(x^+,d_{J_\phi} x^-) = \\ 
& \wp^{i-1,+}(x^+,x^-) + \wp^{i-1,-}(x^-,x^+) + \sum_{\substack{i_1+i_2 = i \\ i_1>0}} d_{\mathit{eq}}^{i_1,+} \wp^{i_2,+}(x^+,x^-) + d_{\mathit{eq}}^{i_1,-} \wp^{i_2,-}(x^+,x^-),
\end{aligned} \\
\label{eq:p-minus-relation} & \begin{aligned}
& d_{J_{\phi^2}} \wp^{i,-}(x^+,x^-) + \wp^{i,-}(d_{J_{\phi}} x^+, x^-) + \wp^{i,-}(x^+,d_{J_\phi} x^-) = \\
& \wp^{i-1,-}(x^+,x^-) + \wp^{i-1,+}(x^-,x^+) + \sum_{\substack{i_1+i_2 = i \\ i_1>0}} d_{\mathit{eq}}^{i_1,+} \wp^{i_2,-}(x^+,x^-) + d_{\mathit{eq}}^{i_1,-} \wp^{i_2,+}(x^+,x^-).
\end{aligned}
\end{align}
Before discussing the origin of these relations in the structure of $\bar\scrM^{i,\sigma}_{\mathit{prod}}(y,x^+,x^-)$, let's see how they are used. Setting $\wp^i = \wp^{i,+} + \wp^{i,-}$, one gets
\begin{equation}
\begin{aligned}
& d_{J_{\phi^2}} \wp^i(x^+,x^-) + \wp^i(d_{J_{\phi}} x^+, x^-) + \wp^i(x^+,d_{J_\phi} x^-) = \\ & \qquad
\wp^{i-1}(x^+,x^-) + \wp^{i-1}(x^-,x^+) + 
\sum_{\substack{i_1+i_2 = i \\ i_1>0}} d_{\mathit{eq}}^{i_1} \wp^{i_2}(x^+,x^-),
\end{aligned}
\end{equation}
which is equivalent to saying that the $\bK[[h]]$-linear map
\begin{equation} \label{eq:chain-product}
%\begin{aligned}
%& 
\wp: C^*(\bZ/2; \mathit{CF}^*(\phi) \otimes \mathit{CF}^*(\phi)) \longrightarrow \mathit{CF}^*_{\mathit{eq}}(\phi^2), \quad
\wp(x^+ \otimes x^-) = \textstyle \sum_i h^i \wp^i(x^+, x^-)
%\end{aligned}
\end{equation}
is a chain map. We define \eqref{eq:equi-pants} to be the induced cohomology level map.

The geometry behind \eqref{eq:p-plus-relation}, \eqref{eq:p-minus-relation} is especially intuitive for low values of $i$. Start with $i = 0$. If one takes $\sigma = -$, the space $\scrM_{\mathit{prod}}^{0,-}(y,x^+,x^-)$ is always empty, since there are no trajectories of $- \nabla f$ going from $v^{0,-}$ to $v^{0,+}$; this explains \eqref{eq:0-}. For the other choice of sign $\sigma = +$, there is one relevant gradient trajectory, namely the constant one $w(s) = v^{0,+}$. This means that $\scrM_{\mathit{prod}}^{0,+}(y,x^+, x^-)$ is a moduli space of perturbed pseudo-holomorphic maps, with no additional parameters. The resulting map $\wp^{0,+}$ is a standard cochain representative for the pair-of-pants product \eqref{eq:pair-of-pants}, and indeed \eqref{eq:p-plus-relation} just specializes to the statement that this is a chain map:
\begin{equation}
d_{J_{\phi^2}} \wp^{0,+}(x^+,x^-) + \wp^{0,+}(d_{J_{\phi}} x^+, x^-) + \wp^{0,+}(x^+,d_{J_\phi} x^-) = 0.
\end{equation}

\begin{remark} 
If we apply $h$-adic filtrations to both sides of \eqref{eq:chain-product}, we get a map between the associated spectral sequences. On the $E_1$ page, this has the form
\begin{equation}
C^*(\bZ/2; \mathit{HF}^*(\phi) \otimes \mathit{HF}^*(\phi)) = (\mathit{HF}^*(\phi) \otimes \mathit{HF}^*(\phi))[[h]] \longrightarrow \mathit{HF}^*(\phi^2)[[h]]. 
\end{equation}
The map is induced by $\wp^{0,+}$, hence is the ($h$-linear extension of) the pair-of-pants product. %In view of Lemma \ref{th:u-adic-spectral-sequence}, the fact that this is compatible with $E_1$ differentials is equivalent to the commutativity of \eqref{eq:commutativity}.
\end{remark}

Now consider the case $i = 1$ and $\sigma = +$, where \eqref{eq:p-plus-relation} says that
\begin{equation} \label{eq:1plus-equation}
d_{J_{\phi^2}} \wp^{1,+}(x^+,x^-) + \wp^{1,+}(d_{J_{\phi}} x^+, x^-) + \wp^{1,+}(x^+,d_{J_\phi} x^-) = 
\wp^{0,+}(x^+,x^-) + d_{\mathit{eq}}^{1,+} \wp^{0,+}(x^+,x^-)
\end{equation}
(it is a priori clear that the right hand side is nullhomotopic, since $d_{\mathit{eq}}^{1,+}$ is chain homotopic to the identity, as previously discussed). There is a unique unparametrized flow line $[w]$ of $-\nabla f$ going from $v^{1,+}$ to $v^{0,+}$. The space $\scrP^{1,+} \iso \bR$ consists of all its possible parametrizations, and gives rise to a one-parameter family of inhomogeneous Cauchy-Riemann equations for maps $\tilde{S} \rightarrow M$. Following the general description in \eqref{eq:stratification-2}, the two boundary points of the compactification $\bar\scrP^{1,+}$ are as follows. 

One boundary point is $\scrP^{0,+} \times \scrQ^{1,+}$, which in terms of broken flow lines means that the limit consists of a constant parametrized flow line $w^\infty(s) = v^{1,+}$, combined with the unparametrized flow line $[w]$. Sequences in $\scrP^{1,+}$ converging to this limit are reparametrizations 
\begin{equation} \label{eq:reparametrize}
w^k(s) = w(s-s^k), 
\end{equation}
with $s^k \rightarrow \infty$. Let $J_z^k$ be the family of almost complex structures on $\tilde{S}$ associated to $w^k$. As $k \rightarrow \infty$, this family has a limit $J_z^\infty$ (in the sense of uniform convergence on compact subsets), which is precisely that associated to the constant gradient flow line $w^\infty$. In fact, the convergence behaviour is better than that: outside the preimage of a compact subset of $S$, one has $J_z^k = J^\infty_z$, since
\begin{align}
\label{eq:converges-1}
& \parbox{35em}{
if $z = \tilde{\epsilon}^+(s,t)$ and $s$ is sufficiently negative, $w(s-s^k)$ is close to $v^{1,+}$ for all $k$, hence $J_z^k = J_{\mathit{left},w(s-s^k),s,t} = J_{\mathit{eq},w(s-s^k),s} = J_{\phi^2,t}$;} \\
\label{eq:converges-2}
& 
\text{if $z = \tilde{\epsilon}^-(s,t)$ and $s$ is sufficiently negative, one similarly has }
J_z^k = (\rho_*J_{\phi^2})_t; \\
\label{eq:converges-3}
& \text{if $z = \tilde{\delta}^{\pm}(s,t)$ and $s \geq 2$, $J^k_z = J^{\pm}_{\mathit{right},w(s-s^k),s,t} = J_{\phi,t}$.} 
\end{align}
Here, \eqref{eq:converges-1} and \eqref{eq:converges-2} use \eqref{eq:prod-eq} as well as \eqref{eq:j-critical}, while \eqref{eq:converges-3} uses \eqref{eq:2-bound}. As a final point, note that even though we have characterized the family $J^\infty_z$ as being associated to the constant flow line at $v^{1,+}$, it is the same as that for the constant flow line at $v^{0,+}$, because of \eqref{eq:left-tau} and \eqref{eq:right-tau}. Note also that the inhomogeneous term in \eqref{eq:perturbed-equation} is the same for all $k$. Given that, a standard Gromov compactness argument shows that if we have a sequence $(w^k,u^k) \in \scrM_{\mathit{prod}}^{1,+}(y,x^+,x^-)$ with $w^k$ as in \eqref{eq:reparametrize}, then a subsequence of the $u^k$ converges on compact subsets to some $u^\infty$ which, together with the constant flow line at $v^{0,+}$, yields an element of one of the moduli spaces $\scrM_{\mathit{prod}}^{0,+}$. In the case when the original $(w^k,u^k)$ belonged to the one-dimensional part of $\scrM_{\mathit{prod}}^{1,+}(y,x^+,x^-)$, one can show that the limit point belongs to $\scrM_{\mathit{prod}}^{0,+}(y,x^+,x^-)$. This, together with a suitable gluing result, explains the appearance of the first term on the right hand side of \eqref{eq:1plus-equation}.

The other boundary point is $\scrQ^{1,+} \times \scrP^{0,+}$, which consists of $[w]$ together with a constant parametrized flow line $w^\infty = v^{0,+}$. A sequence converging to this limit can be written as in \eqref{eq:reparametrize}, but where $s^k \rightarrow -\infty$. The associated families of almost complex structures $J^k_z$ converge to the same limit $J^\infty_z$ as before (uniformly on compact subsets). Correspondingly, if $u^k$ are such that $(w^k,u^k)$ is a sequence in $\scrM_{\mathit{prod}}^{1,+}(y,x^+,x^-)$, a subsequence of the $u^k$ will converge (on compact subsets) to a limit $u^\infty$ such that $(w^\infty,u^\infty)$ belongs to one of the moduli spaces $\scrM_{\mathit{prod}}^{0,+}$. Note that over the ends $\tilde{\delta}^{\pm}$, one still has \eqref{eq:converges-3}, but over the other ends $\tilde{\epsilon}^{\pm}$, the behaviour of the $J^k_z$ is no longer as simple as in \eqref{eq:converges-1}, \eqref{eq:converges-2}. Instead, with a suitable reparametrization, one has
\begin{equation}
J^k_{\tilde{\epsilon}^{\pm}(s+s_k,t)} = J_{\mathit{left},\pm w(s),s+s_k,t} = J_{\mathit{eq},w(s),t} \;\; \text{ if $s \leq -2-s_k$,}
\end{equation}
by \eqref{eq:prod-eq}. As a consequence, a subsequence of the maps $\tilde{u}^k(s,t) = u^k(\tilde{\epsilon}^+(s+s_k,t))$ converges on compact subsets to some $\tilde{u}^\infty$ such that $[w,\tilde{u}^\infty]$ is an element in one of the moduli spaces $\scrM_{\mathit{eq}}^{1,+}$. One now has two components of the limit: the principal component $(w^\infty,u^\infty)$, and the non-principal component $[w,\tilde{u}^\infty]$. This explains the second term on the right hand side of \eqref{eq:1plus-equation}.

Next, let's look at the parallel situation for $i = 1$ and $\sigma = -$, where \eqref{eq:p-minus-relation} specializes to
\begin{equation} \label{eq:1minus-equation}
d_{J_{\phi^2}}\wp^{1,-}(x^+, x^-) + \wp^{1,-}(d_{J_\phi}x^+, x^-) + \wp^{1,-}(x^+, d_{J_\phi}x^-) = \wp^{0,+}(x^-,x^+) + d_{\mathit{eq}}^{1,-} \wp^{0,+}(x^+,x^-)
\end{equation}
(since $d_{\mathit{eq}}^{1,-}$ induces the involution on $\mathit{HF}^*(\phi^2)$, the commutativity of \eqref{eq:commutativity} is equivalent to the fact that the right hand side of \eqref{eq:1minus-equation} is nullhomotopic). As in the previously discussed case, there is a single unparametrized flow line $[w]$ from $v^{1,-}$ to $v^{0,+}$. Consider the limit \eqref{eq:reparametrize} with $s^k \rightarrow \infty$. In this case, the Cauchy-Riemann equations on $S$ converge to that associated to the constant gradient flow line $v^{1,-}$ (and there are counterparts of \eqref{eq:converges-1}--\eqref{eq:converges-3} as well). As shown in Addendum \ref{th:flip-the-sign}, the family of almost complex structures associated to (the constant flow line at) $v^{1,-}$ is related to that for $v^{1,+}$ by the action of the involution $\tilde{\gamma}$ on $\tilde{S}$; and the inhomogeneous term is invariant under that involution. If we then define $u^\infty$ as before, it follows that $(v^{0,+}, u^\infty \circ \tilde{\gamma})$ is an element of one of the moduli spaces $\scrM_{\mathit{prod}}^{1,+}$. Recall from \eqref{eq:tilde-delta-properties} that $\tilde{\gamma}$ exchanges the two ends $\tilde{\delta}^{\pm}$. In the case where the original $(w^k,u^k)$ belonged to the one-dimensional part of $\scrM_{\mathit{prod}}^{1,-}(y,x^+,x^-)$, one finds that 
\begin{equation}
\textstyle
\lim_{s \rightarrow +\infty} u^\infty(\tilde{\gamma}(\tilde{\delta}^{\pm}(s,t))) =
\lim_{s \rightarrow + \infty} u^\infty(\tilde{\delta}^\mp(s,t)) = x^{\mp},
\end{equation}
where the effect of the $\tilde{\gamma}$ is to swap the roles of the limits $x^{\pm}$. Similarly, using \eqref{eq:tilde-epsilon-properties}, and taking into account the way in which the ends $\tilde{\epsilon}^{\pm}$ appear in \eqref{eq:pants-equation}, one gets
\begin{equation} \textstyle
\lim_{s \rightarrow -\infty} u^\infty(\tilde{\gamma}(\tilde{\epsilon}^+(s,t))) =
\lim_{s \rightarrow -\infty} u^\infty(\tilde{\epsilon}^-(s,t)) = y.
\end{equation}
Hence, $(v^{0,+}, u^\infty \circ \tilde{\gamma})$ is actually an element of $\scrM_{\mathit{prod}}^{0,+}(y,x^-,x^+)$, which explains the first term on the right hand side of \eqref{eq:1minus-equation}. The second term arises exactly in the same way as its counterpart in \eqref{eq:1plus-equation}.
\begin{figure}
\begin{picture}(0,0)%
\includegraphics{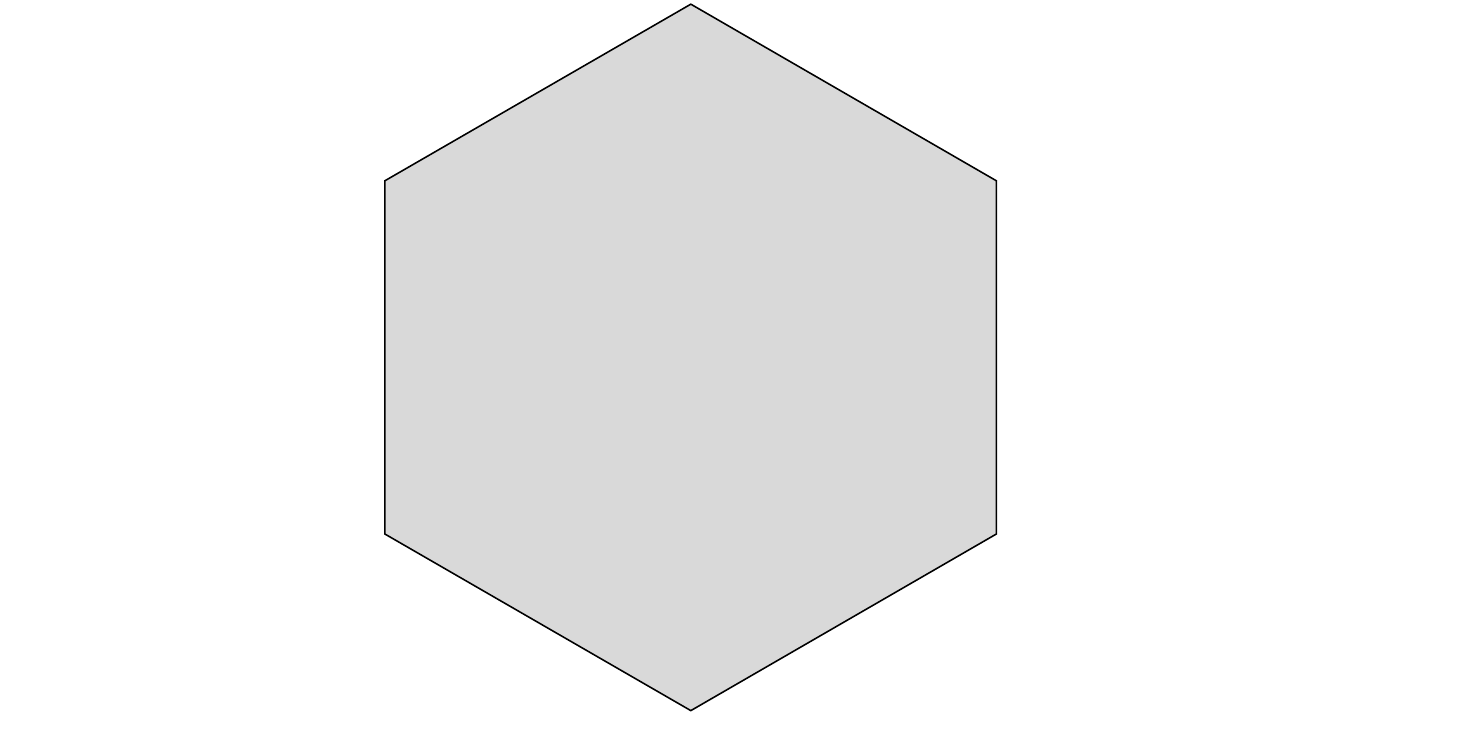}%
\end{picture}%
\setlength{\unitlength}{3947sp}%
\begingroup\makeatletter\ifx\SetFigFont\undefined%
\gdef\SetFigFont#1#2#3#4#5{%
  \reset@font\fontsize{#1}{#2pt}%
  \fontfamily{#3}\fontseries{#4}\fontshape{#5}%
  \selectfont}%
\fi\endgroup%
\begin{picture}(6991,3526)(-914,-3125)
\put(3151, 89){\makebox(0,0)[lb]{\smash{{\SetFigFont{10}{12.0}{\rmdefault}{\mddefault}{\updefault}{$v^{2,+} \noarrow v^{1,+} \yesarrow v^{0,+}$}%
}}}}
\put(3226,-2761){\makebox(0,0)[lb]{\smash{{\SetFigFont{10}{12.0}{\rmdefault}{\mddefault}{\updefault}{$v^{2,+} \yesarrow v^{2,+} \noarrow v^{0,+}$}%
}}}}
\put(-899,-1336){\makebox(0,0)[lb]{\smash{{\SetFigFont{10}{12.0}{\rmdefault}{\mddefault}{\updefault}{$v^{2,+} \yesarrow v^{1,-} \noarrow v^{0,+}$}%
}}}}
\put(-149,-2761){\makebox(0,0)[lb]{\smash{{\SetFigFont{10}{12.0}{\rmdefault}{\mddefault}{\updefault}{$v^{2,+} \noarrow v^{1,-} \yesarrow v^{0,+}$}%
}}}}
\put(3976,-1261){\makebox(0,0)[lb]{\smash{{\SetFigFont{10}{12.0}{\rmdefault}{\mddefault}{\updefault}{$v^{2,+} \yesarrow v^{1,+} \noarrow v^{0,+}$}%
}}}}
\put(151, 89){\makebox(0,0)[lb]{\smash{{\SetFigFont{10}{12.0}{\rmdefault}{\mddefault}{\updefault}{$v^{2,+} \noarrow v^{0,+} \yesarrow v^{0,+}$}%
}}}}
\put(3601,-211){\makebox(0,0)[lb]{\smash{{\SetFigFont{10}{12.0}{\rmdefault}{\mddefault}{\updefault}{$d_{\mathit{eq}}^{1,+} \wp^{1,+}$}%
}}}}
\put(451,-211){\makebox(0,0)[lb]{\smash{{\SetFigFont{10}{12.0}{\rmdefault}{\mddefault}{\updefault}{$d_{\mathit{eq}}^{2,+} \wp^{0,+}$}%
}}}}
\put(526,-3061){\makebox(0,0)[lb]{\smash{{\SetFigFont{10}{12.0}{\rmdefault}{\mddefault}{\updefault}{$d_{\mathit{eq}}^{1,-} \wp^{1,-}$}%
}}}}
\put(4276,-1561){\makebox(0,0)[lb]{\smash{{\SetFigFont{10}{12.0}{\rmdefault}{\mddefault}{\updefault}{$\wp^{1,+}$}%
}}}}
\put(-1000,-1636){\makebox(0,0)[lb]{\smash{{\SetFigFont{10}{12.0}{\rmdefault}{\mddefault}{\updefault}{$\wp^{1,-}$ (inputs exchanged)}%
}}}}
\put(3526,-3061){\makebox(0,0)[lb]{\smash{{\SetFigFont{10}{12.0}{\rmdefault}{\mddefault}{\updefault}{$0$ (contribution vanishes)}%
}}}}
\end{picture}%
\caption{\label{fig:2rel}}
\end{figure}%

The final example we want to look at is $i = 2$ and $\sigma = +$. Figure \ref{fig:2rel} shows the broken flow lines associated to the boundary faces of $\bar{\scrP}^{2,+}$ (where the dotted arrows mark the parametrized flow lines), together with the corresponding terms in the relevant instance of \eqref{eq:p-plus-relation}:
\begin{equation} \label{eq:2-plus-relation}
\begin{aligned}
& d_{J_{\phi^2}} \wp^{2,+}(x^+,x^-) + \wp^{2,+}(d_{J_{\phi}} x^+, x^-) + \wp^{2,+}(x^+,d_{J_\phi} x^-) = \\ 
& \quad \wp^{1,+}(x^+,x^-) + \wp^{1,-}(x^-,x^+) + d_{\mathit{eq}}^{1,+} \wp^{1,+}(x^+,x^-) + d_{\mathit{eq}}^{1,-} \wp^{1,-}(x^+,x^-) + d^{2,+}_{\mathit{eq}} \wp^{0,+}(x^+,x^-).
\end{aligned}
\end{equation}
Note that one codimension $1$ boundary face, namely 
\begin{equation} \label{eq:hidden-face}
\scrP^{0,+} \times \scrQ^{2,+} \subset \partial\bar\scrP^{2,+}, 
\end{equation}
yields a trivial contribution. Even though this may appear to be a new phenomenon, it is actually due to the same mechanism which produces the first terms on the right hand sides of \eqref{eq:1plus-equation} and \eqref{eq:1minus-equation}. As one approaches any point in this boundary face, the Cauchy-Riemann equations on $\tilde{S}$ converge to the same limit, which is the equation that underlies $\wp^{0,+}$; that convergence is locally uniform, and also uniform on any subset of the form $\tilde{\delta}^{\pm}([1,\infty) \times (t_0,t_1))$ or $\tilde{\epsilon}^{\pm}([1,\infty) \times (t_0,t_1))$. This means that the principal component of the limit is an element of $\scrM^{0,+}_{\mathit{prod}}(y,x^+,x^-)$, independently of which point of \eqref{eq:hidden-face} one approaches. Because of the extra $\scrQ^{2,+}$ parameter, there are no isolated points in the resulting part of $\bar\scrM^{2,+}_{\mathit{prod}}(y,x^+,x^-)$.

The examples above already contain all the issues one encounters in the general case. There are $(4i-2)$ codimension one boundary faces of $\bar\scrP^{i,\sigma}$, of the form
\begin{align} 
\label{eq:pq-face}
& \scrP^{0,+} \times \scrQ^{i,\sigma}, \; \cdots, \; \scrP^{i-1,+} \times \scrQ^{1,\sigma}, \;
   \scrP^{1,-} \times \scrQ^{i-1,-\sigma},\; \cdots, \; \scrP^{i-1,-} \times \scrQ^{1,-\sigma}, \\
\label{eq:qp-face}
& \scrQ^{1,\sigma} \times \scrP^{i-1,+}, \; \cdots, \; \scrQ^{i,\sigma} \times \scrP^{0,+}, \;
   \scrQ^{1,-\sigma} \times \scrP^{i-1,-}, \; \cdots, \; \scrQ^{i-1,-\sigma} \times \scrP^{1,-}.
\end{align}
Of the faces \eqref{eq:pq-face}, those of the form $\scrP^{i-1,\pm} \times \scrQ^{1,\pm \sigma} \iso \scrP^{i-1,\pm}$ contribute the first two terms of on the right hand side of \eqref{eq:p-plus-relation}, \eqref{eq:p-minus-relation}. All others contribute zero, for the same reason as in the special case \eqref{eq:hidden-face}. In contrast, all faces  \eqref{eq:qp-face} contribute, and give rise to the remaining terms on the right hand side of \eqref{eq:p-plus-relation}, \eqref{eq:p-minus-relation} (the left hand side, as usual, accounts for bubbling off of solutions of \eqref{eq:floer} over the ends).

\begin{addendum} \label{th:careful-action}
The introduction of inhomogeneous terms slightly complicates arguments about the action filtration. For any solution $u$ of the perturbed version \eqref{eq:perturbed-equation} of \eqref{eq:pants-equation}, 
\begin{equation} \label{eq:inhomogeneous-action}
A_{\phi^2}(y) - A_\phi(x^+) - A_\phi(x^-) \geq \textstyle \sum_\sigma \int_{[-2,-1] \times [0,1]} (u \circ \tilde{\delta}^\sigma)^*(\partial_s H_{s,t}).
\end{equation}
In particular, if 
\begin{equation} \label{eq:bound-h}
\textstyle \int_{[-2,-1] \times (0,1)} ||\partial_s H_{s,t}||_{L^\infty} < \epsilon
\end{equation}
for some constant $\epsilon>0$, the integrand in \eqref{eq:inhomogeneous-action} is pointwise $> -\epsilon$. Given $\phi$, there is an $\epsilon$ such that
\begin{equation} \label{eq:2h}
A_{\phi^2}(y) - A_\phi(x^+) - A_\phi(x^-) \notin (-2\epsilon,0) \quad \text{for all fixed points $y,x^\pm$.}
\end{equation}
Suppose that we've chosen $H$ in such a way that \eqref{eq:bound-h} holds for this $\epsilon$. It then follows that a solution $u$ can exist only if
\begin{equation}
A_{\phi^2}(y) - A_\phi(x^+) - A_{\phi}(x^-) \geq 0.
\end{equation}
In other words, for sufficiently small choices of inhomogeneous terms, $\wp$ will preserve the action filtration. 
\end{addendum}

%\newpage
\section{Symplectic linear algebra and index theory\label{sec:linear}}

This section collects (classical) background material, which underlies the local study of nondegenerate 2-periodic points of symplectic automorphisms.

\subsection{The Krein index}
Let $(H,\omega_H)$ be a symplectic vector space of dimension $2n$. Denote its linear automorphism group by $\mathit{Sp}(H)$, and the associated Lie algebra (often called the space of Hamiltonian endomorphisms) by $\mathfrak{sp}(H)$. Consider the open subsets
\begin{align} \label{eq:starstar}
& \mathit{Sp}^{**}(H) = \{A \in \mathit{Sp}(H) \;:\; \pm 1 \notin \mathit{spec}(A)\}, \\
\label{eq:lie-starstar}
& \mathfrak{sp}^{**}(H) = \{B \in \mathfrak{sp}(H) \;:\; 0,\pm 1 \notin \mathit{spec}(B)\}
\end{align}
(since the spectrum of $B$ is symmetric around zero, having $1$ or $-1$ as eigenvalues are equivalent conditions). The Cayley transform 
\begin{equation} \label{eq:cayley}
A = (B+I)(B-I)^{-1}
\end{equation}
yields a diffeomorphism between \eqref{eq:lie-starstar} and \eqref{eq:starstar} \cite[p.\ 18]{arnold-givental}.

Semisimple matrices form an open and dense subset of $\mathfrak{sp}^{**}(H)$ \cite[p.\ 14]{arnold-givental}, and therefore of $\mathit{Sp}^{**}(H)$ as well, by the Cayley transform. Any semisimple element of $\mathit{Sp}^{**}(H)$ can be written, with respect to some identification $(H, \omega_H) \iso (\bR^{2n}, dp_1 \wedge dq_1 + \cdots + dp_n \wedge dq_n)$, as a direct sum of blocks of the following form (see the corresponding statement for $\mathfrak{sp}^{**}(H)$ in \cite{williamson36} or \cite[p.\ 10]{arnold-givental}):
\begin{equation} \label{eq:blocks}
\begin{array}{l|l|l}
\text{type} & \text{symplectic matrix} & \text{eigenvalues} \\ \hline
\text{(i+)} 
& \left(\begin{smallmatrix} a & 0 \\ 0 & a^{-1} \end{smallmatrix}\right), \;\; a \in (0,1)
& \text{real $>0$}
 \\ \hline
\text{(i-)} 
& \text{same as (i+)}, \;\; a \in (-1,0)
& \text{real $<0$}
 \\ \hline
\text{(ii+)} 
& \left(\begin{smallmatrix} a_1 & -a_2 \\ a_2 & a_1 \end{smallmatrix}\right), \;\; a_1^2+a_2^2 = 1, \; a_2>0
& \text{unit circle}
\\ \hline
\text{(ii-)} &
\text{same as (ii+)}, \; a_2 < 0, 
& \text{unit circle}
\\ \hline
\text{(iii)} 
& \left(\begin{smallmatrix} 
a_1 & 0 & -a_2 & 0 \\
0 & a_1/(a_1^2+a_2^2) & 0 & -a_2/(a_1^2+a_2^2) \\
a_2 & 0 & a_1 & 0 \\
0 & a_2/(a_1^2+a_2^2) & 0 & a_1/(a_1^2+a_2^2)
\end{smallmatrix}\right), \;\;
\begin{matrix}
a_1 \in (-1,1) \\
a_1^2 + a_2^2 \in (0,1] 
\end{matrix}
& 
\begin{matrix}
\text{quadruple} \\
(a_1 \pm ia_2)^{\pm 1}
\end{matrix}
\end{array}
\end{equation}
where the last matrix is written in coordinates $(p_1,q_1,p_2,q_2)$. There is some overlap - the following are equal or conjugate in $\mathit{Sp}(\bR^4)$:
\begin{align}
& \text{type (iii) with $(a_1,a_2)$} \quad\sim\quad \text{type (iii) with $(a_1,-a_2)$} \\
& \text{type (iii) with $a_1>0$, $a_2 = 0$} \quad=\quad \text{direct sum of two type (i+) blocks} 
\label{eq:trade-plus} \\
& \text{type (iii) with $a_1<0$, $a_2 = 0$} \quad=\quad \text{direct sum of two type (i-) blocks} 
\label{eq:trade-minus} \\
& \text{type (iii) with $a_1^2 + a_2^2 = 1$} \quad\sim\quad \text{sum of a type (ii+) and (ii-) blocks}. \label{eq:circle-fusion}
\end{align}
 
Take $A \in \mathit{Sp}^{**}(H)$, and let $E \subset H_{\bC} = H \otimes_{\bR} \bC$ be the direct sum of all generalized eigenspaces for the eigenvalues $\lambda$ of $A$ which satisfy 
\begin{equation} \label{eq:positive-im}
|\lambda|^2 = 1, \quad \mathrm{im}(\lambda)>0. 
\end{equation}
The space $E$ comes with a nondegenerate hermitian form \cite[Chapter 1.2, Definition 8]{ekeland}
\begin{equation} \label{eq:e-form}
\langle h_1,h_2 \rangle_E = i \omega_H(\bar{h}_1,h_2).
\end{equation}

\begin{definition} 
The Krein index of $\kappa(A)$ is the signature of \eqref{eq:e-form}. In other words, if there is a isomorphism $E \iso \bC^i \times \bC^j$ which transforms our hermitian form into $d\bar{x}_1\, dx_1 + \cdots + d\bar{x}_i\, dx_i - d\bar{y}_1\,dy_1 - \cdots - d\bar{y}_j \, dy_j$, then $\kappa(A) = i-j$.
\end{definition}

Since the (generalized) eigenvalues which lie on the unit circle come in pairs $\{\lambda, \bar{\lambda}\}$ of equal multiplicity, $\mathrm{dim}(E) \leq n$. Moreover, if $\mathrm{det}(I-A) < 0$, at least one eigenvalue must lie outside the unit circle, hence the inquality of dimensions will then be a strict one. One concludes that
\begin{equation} \label{eq:conditions-kappa}
\left\{\begin{aligned}
& |\kappa(A)| \leq n && \text{if $\mathrm{det}(I-A) > 0$,} \\
& |\kappa(A)| \leq n-1 && \text{if $\mathrm{det}(I-A) < 0$.}
\end{aligned}\right.
\end{equation}
By the same consideration, the parity of $\kappa(A)$ is the dimension of $E$, or equivalently
\begin{equation}
(-1)^{\kappa(A)} = (-1)^n\, \mathrm{sign}(\mathrm{det}(I-A^2)).
\end{equation}
% (i+) + (i-) + (ii+) + (ii-) = n (mod 2)
% det(I-A) = counts i+ blocks
% det(I-A^2) = counts i+ and i- blocks
% krein = (ii+)-(ii-)

\begin{lemma} \label{th:krein}
$\kappa: \mathit{Sp}^{**}(H) \longrightarrow \bZ$ is a locally constant function.
\end{lemma}

This statement is not trivial, since $E$ can change discontinuously under deformations. It is part of Krein's stability theory (\cite{krein50,gelfand-lidsky58}; see \cite[Appendice 29]{arnold-avez} or \cite{moser58, ekeland} for expositions). We will give alternative perspectives in Lemmas \ref{th:krein-conley-zehnder} and \ref{th:index-theory} (these won't be strictly independent, since we'll use Lemma \ref{th:krein} on the way to proving them).

\begin{example}
If $A$ is semisimple, $\kappa(A)$ is the number of type (ii+) blocks minus the number of type (ii-) blocks. Indeed, for those two blocks, $E$ is spanned by $h = (1,\mp i)$, with $\langle h,h \rangle_E = \pm 2$. For all other block types, $E$ vanishes, with the obvious exception of \eqref{eq:circle-fusion} whose contribution is trivial.
\end{example}
% (1,-i) -> (a_1+ia_2,a_2-ia_1). In the + case, this is the required eigenvalue.
% + case: i \omega( (1,i), (1,-i)) = i(-i-i) = -2i^2 = 2
% - case: (1,i) leads to i\omega( (1,-i), (1,i) ) = -i .. i > 0.

\begin{example} \label{th:epsilon-example}
Take a nondegenerate quadratic form $Q$, with its associated $B \in \mathfrak{sp}(H)$, and set $A = \exp(t B)$ for small $t>0$. Then
\begin{equation} \label{eq:epsilon-path}
\kappa(A) = n-i(Q),
\end{equation}
where $i(Q)$ is the Morse index. Because $\kappa$ is locally constant, it is a priori clear that $\kappa(A)$ depends only on the Morse index. Since $\kappa$ is additive under direct sums, it is sufficient to check \eqref{eq:epsilon-path} in the case where $H = \bR^2$ and $Q(p,q) = \pm p^2 \pm q^2$, corresponding to blocks of type (i+), (ii+), (ii-).
\end{example}
% sign(Q) = -i(Q) + (2n-i(Q)) = 2n-2i(Q)

\begin{example} \label{th:rotated-epsilon-example}
Suppose that $H = \bR^{2n}$. Take $Q = p_1q_1 + ${\it (quadratic form in the other $2n-2$ variables)}, with associated $B \in \mathfrak{sp}(H)$. Set $A = R \exp(t B)$ for small $t>0$, where $R$ maps $(p_1,q_1,p_2,q_2,\dots)$ to $(-p_1,-q_1,p_2,q_2,\dots)$. Then, the Krein index is given by the same formula \eqref{eq:epsilon-path} as before. To check this, one can again use additivity, which means that it is enough to consider the case of $\bR^2$ and $Q = pq$; in that case, $A$ is of type (i-).
\end{example}

\begin{lemma} \label{th:connected}
The map
\begin{equation} \label{eq:invariant}
\pi_0(\mathit{Sp}^{**}(H)) \longrightarrow \{\pm 1\} \times \bZ, 
\quad A \mapsto ( \mathrm{sign}(\mathrm{det}(I-A)), \kappa(A) )
\end{equation}
is injective, and its image is precisely given by \eqref{eq:conditions-kappa}.
\end{lemma}

\begin{proof}
Consider first the case $n = 1$, and set $H = \bR^2$. Writing $B = \left(\begin{smallmatrix} b_1 & b_2+b_3 \\ b_2-b_3 & -b_1 \end{smallmatrix}\right)$, one has
\begin{equation}
\mathfrak{sp}^{**}(\bR^2) = \{b \in \bR^3 \;:\; b_1^2 + b_2^2 - b_3^2 \neq 0,1\}.
\end{equation}
This clearly has four connected components, which under the Cayley transform correspond to the four size $2$ blocks in \eqref{eq:blocks}. In the order given there, the values of \eqref{eq:invariant} are $(-1,0)$, $(1,0)$, $(1,1)$, and $(1,-1)$, which implies the desired result.

Now consider the case $n>1$. Any element of $\mathit{Sp}^{**}(H)$ can be perturbed to a semisimple one. Because the symplectic group is connected, any two semisimple elements which have the same kind of block decomposition \eqref{eq:blocks} can be deformed into each other inside $\mathit{Sp}^{**}(H)$. By \eqref{eq:trade-plus}, two blocks of type (i+) can be traded for a block of type (iii), and the same is true of type (i-) by \eqref{eq:trade-minus}. This reduces us to the case where there is at most one block of type (i+) and at most one block of type (i-). Similarly, given one block of type (ii+) and one block of type (ii-), one can trade them for a block of type (iii) by \eqref{eq:circle-fusion}. Hence, by applying such deformations, one can kill either the type (ii+) blocks or the type (ii-) blocks. After that, the type (ii) part of the block decomposition is determined by $\kappa(A)$. The type (i) part is determined by the sign of $\mathrm{det}(I-A)$ together with the parity of $n$. This shows injectivity. It is straightforward to see that all values allowed by \eqref{eq:conditions-kappa} are achieved.
\end{proof}

\subsection{Index theory\label{subsec:index}}
Consider the subsets
\begin{align} 
\label{eq:star}
& \mathit{Sp}^{*}(H) = \{A \in \mathit{Sp}(H) \;:\; 1 \notin \mathit{spec}(A)\}, \\
\label{eq:lie-star}
& \mathfrak{sp}^{*}(H) = \{B \in \mathfrak{sp}(H) \;:\; \pm 1 \notin \mathit{spec}(B)\},
\end{align}
which are again diffeomorphic by \eqref{eq:cayley}. This time there are only two connected components, which are distinguished by the sign of $\mathrm{det}(I-A)$. Take the universal cover $\widetilde{\mathit{Sp}}(H)$, which is again a Lie group, and consider the preimage $\widetilde{\mathit{Sp}}^*(H)$ of \eqref{eq:star}. The connected components of this are classified by the Conley-Zehnder index, which is a locally constant function
\begin{equation} \label{eq:conley-zehnder}
\mu: \widetilde{\mathit{Sp}}^*(H) \longrightarrow \bZ
\end{equation}
satisfying 
\begin{equation}
(-1)^{\mu(\tilde{A})} = \mathrm{sign}(\mathrm{det}(I-A)).
\end{equation}
The action of the standard generator of the covering group $\pi_1(\mathit{Sp}(H)) \iso \bZ$ on an element $\tilde{A}$ decreases its Conley-Zehnder index by $2$.

\begin{remark}
The Conley-Zehnder index was introduced in \cite{conley-zehnder84}. Compared to the exposition in \cite{salamon-zehnder92}, our conventions are as follows. Inside $\widetilde{\mathit{Sp}}(H)$, take a path from the identity to $\tilde{A}$, and then project that path to $\mathit{Sp}(H)$. The index of that path, as defined in \cite[Theorem 3.3]{salamon-zehnder92}, is $\mu(\tilde{A})-n$ in our notation.
% (note that, due to the definition of $J_0$ in \cite[p.\ 1314]{salamon-zehnder92}, the complex determinant appearing in \cite[Theorem 3.1]{salamon-zehnder92} is the inverse of the usual determinant of a unitary matrix).
\end{remark}

\begin{example} \label{th:lifted-epsilon-example}
Take $A$ as in Example \ref{th:epsilon-example}, and consider the lift $\tilde{A}$ which is the exponential of $t B$ inside $\widetilde{\mathit{Sp}}(H)$ (equivalently, this is the unique lift which is close to the identity element of the universal cover). Then $\mu(\tilde{A}) = i(Q)$, compare \cite[Theorem 3.3(iv)]{salamon-zehnder92}.
\end{example}
%sign = (2n-mu)-mu = 2n-2mu -> 2\mu = 2n-sign
%
\begin{example} \label{th:lifted-rotated-epsilon-example}
Take $A$ as in Example \ref{th:rotated-epsilon-example}. Consider the lift $\tilde{A}$ obtained by using the exponential as before, together with the lift $\tilde{R}$ which one gets from the path that rotates $(p_1,q_1)$ anticlockwise by $\pi$. Then $\mu(\tilde{A}) = i(Q)-1$ (this can be reduced to Example \ref{th:lifted-epsilon-example} by a deformation).
\end{example}

\begin{lemma} \label{th:krein-conley-zehnder}
Take $A \in \mathit{Sp}^{**}(H)$. Then, for any lift $\tilde{A}$ to the universal cover,
\begin{equation} \label{eq:cz-equation}
\kappa(A) - n = \mu(\tilde{A}^2) - 2\mu(\tilde{A}).
\end{equation}
\end{lemma}

\begin{proof}
Both sides of \eqref{eq:cz-equation} are independent of the choice of lift $\tilde{A}$. Because they are also locally constant, it is enough to verify the equality for one $A$ in each connected component of $\mathit{Sp}^{**}(H)$. But each such component contains a representative which is either as in Example \ref{th:epsilon-example} or Example \ref{th:rotated-epsilon-example}.

Consider first the situation of Example \ref{th:epsilon-example}, and choose the lift $\tilde{A}$ as in Example \ref{th:lifted-epsilon-example}. Then $\tilde{A}^2$ is the corresponding lift of $A^2 = \exp(2t B)$, hence
\begin{equation}
\mu(\tilde{A}^2) - 2\mu(\tilde{A}) = i(Q) - 2i(Q) = \kappa(A)-n.
\end{equation}

Now switch to Example \ref{th:rotated-epsilon-example}. If $A$ is as in that example, then $A^2 = \mathrm{exp}(2t B)$ is as in Example \ref{th:epsilon-example}. However, if we choose a lift $\tilde{A}$ as in Example \ref{th:lifted-rotated-epsilon-example}, then $\tilde{A}^2$ differs from the lift of $A^2$ given in Example \ref{th:lifted-epsilon-example} by the action of the generator of the covering group. This means that $\mu(\tilde{A}^2) = i(Q) - 2$, which again leads to
\begin{equation}
\mu(\tilde{A}^2) - 2\mu(\tilde{A}) = i(Q)-2 - 2(i(Q)-1) = \kappa(A)-n.
\end{equation}
% \mu(\tilde{A}^2) - 2\mu(\tilde{A}) + n
\end{proof}

Let $S$ be the pair-of-pants surface, as in Section \ref{subsec:product}. Any $A \in \mathit{Sp}^{**}(H)$ determines a flat symplectic vector bundle on $\bR \times S^1$, which has fibre $H$ and holonomy $A$ around the circle. Pulling this back via \eqref{eq:s-projection} yields a flat symplectic vector bundle $F \rightarrow S$, with holonomy $A$ around each of the two ends \eqref{eq:epsilon-in}, and holonomy $A^2$ around the remaining end \eqref{eq:epsilon-out}. Let's equip $F$ with a family of compatible almost complex structures $J_F$ on its fibres, which has the property that over each end, it is covariantly constant in $s$-direction (here, $(s,t)$ are the coordinates on the ends). We can then associate to this a Cauchy-Riemann operator
\begin{equation} \label{eq:d-bar-a}
D_A: \scrE^1 \longrightarrow \scrE^0, \quad D_A = \nabla^{0,1},
\end{equation}
which is the $(0,1)$-part of the covariant derivative (for the given flat connection $\nabla$ on $F$), from $\scrE^1 = W^{k,p}(F)$ to $\scrE^0 = W^{k-1,p}(\Omega^{0,1}_S \otimes F)$. Because neither $A$ nor $A^2$ have $1$ as an eigenvalue, $D_A$ is elliptic. 

\begin{lemma} \label{th:index-theory}
The Fredholm index of $D_A$ is $\mathrm{index}(D_A) = \kappa(A)-n$.
\end{lemma}

Using Lemma \ref{th:krein-conley-zehnder}, this becomes a special case the index formula for Cauchy-Riemann operators on surfaces with tubular ends \cite[Proposition 3.3.10]{schwarz95}.

\begin{lemma} \label{th:no-kernel}
$D_A$ is always injective.
\end{lemma}

\begin{proof}
This is an analogue of our previous discussion of \eqref{eq:constant-linearization}. The total space of $F$ carries a canonical closed two-form $\omega_F$, which fibrewise reduces to $\omega_H$. The counterpart of \eqref{eq:zero-energy-0} for a section $\xi \in \scrE^1$, where we again set $(k,p) = (2,2)$, is
\begin{equation}\textstyle 
\int_{S} \half |D_A \xi|^2 + \int_{S} \xi^*\omega_F = \int_S \half |\nabla \xi|^2,
\end{equation}
where the norms are taken with respect to the metric induced by $J_F$. The integral of $\xi^*\omega_F$ is a topological invariant (unchanged under deforming $\xi$), hence must vanish (since it's trivial for $\xi = 0$). Hence, if $D_A \xi = 0$ for some $\xi \in W^{k,p}(F)$, then $\xi$ must be covariantly constant, which (since it goes to zero at the ends) shows that it vanishes.
\end{proof}

%\newpage
\section{Local contributions\label{sec:local}}

This section contains the proof of Theorem \ref{th:main}. We want to prove that the map \eqref{eq:chain-product} becomes a quasi-isomorphism after tensoring with $\bK((h))$. The strategy is to show that the corresponding statement holds for the associated graded spaces of a suitable filtration, which in our case will be the action filtration. In a standard pseudo-holomorphic map setup, this would mean that we only have to count the solutions with zero energy, which are constant. Our situation is technically slightly more complicated, because we have perturbed the pseudo-holomorphic map equation by adding inhomogeneous terms; but it still true that the relevant contributions are local in nature, and can be determined in an essentially elementary way.

\subsection{Definition and general properties}
For our computations to be meaningful, we need to restrict the inhomogeneous terms to be small. As usual, we work with a fixed symplectic automorphism $\phi$ as in Setup \ref{th:setup-2}.

\begin{setup} \label{th:h-setup}
Fix a constant $\epsilon > 0$ such that the following holds:
\begin{align}
\label{eq:action-difference} & A_\phi(x^+) - A_\phi(x^-) \notin (0,2\epsilon) \quad \text{for all fixed points $x^{\pm}$ of $\phi$}, \\
\label{eq:action-difference-2} & A_{\phi^2}(y^+) - A_{\phi^2}(y^-) \notin (0,2\epsilon) \quad \text{for all fixed points $y^{\pm}$ of $\phi^2$.}
\end{align}
When choosing Data \ref{th:inhomogeneous-data}, we assume that it satisfies \eqref{eq:bound-h} with this particular constant. 
\end{setup}

For any fixed point $x$ of $\phi$ and sign $\sigma$, define
\begin{equation} \label{eq:c-sigma}
c_x^{\sigma} = \sum_i h^i \, \# \scrM_{\mathit{prod}}^{i,\sigma}(x,x,x) \in \bK[[h]].
\end{equation}
The sum $c_x = c_x^+ + c_x^-$ is called the local contribution of $x$ to the equivariant pair-of-pants product \eqref{eq:chain-product}.

\begin{lemma} \label{th:independence-1}
$c_x$ is independent of all auxiliary data that enter into the construction of the moduli space $\scrM_{\mathit{prod}}^{i,\sigma}(x,x,x)$.
\end{lemma}

\begin{proof}
This is an argument involving moduli spaces with one additional parameter. The data under discussion are: the almost complex structures used to define the differentials on $\mathit{CF}^*(\phi)$ and $\mathit{CF}^*(\phi^2)$; the additional almost complex structures that enter into the differential on $\mathit{CF}^*_{\mathit{eq}}(\phi^2)$ (Data \ref{th:eq-data}); and the almost complex structures (Data \ref{th:j-prod-data}) as well as inhomogeneous terms (Data \ref{th:inhomogeneous-data}) required to construct $\wp$. Suppose that we have two choices of such data. We can interpolate between them by a one-parameter family of the same kind of choices, which satisfy the same bound \eqref{eq:bound-h} for all parameter values. 

To be more precise, denote the parameter by $r \in [0,1]$ (so that the two choices of data that we want to compare appear at the endpoints $r = 0,1$). For each value of $r$, we have spaces $\scrM_{\mathit{prod}}^{i,\sigma}(y,x^+,x^-)^{(r)}$ defined as before, and compactifications $\bar\scrM_{\mathit{prod}}^{i,\sigma}(y,x^+,x^-)^{(r)}$. The parametrized analogues are defined as
\begin{align}
& \textstyle \scrM_{\mathit{para}}^{i,\sigma}(y,x^+,x^-) = \bigsqcup_r \scrM_{\mathit{prod}}^{i,\sigma}(y,x^+,x^-)^{(r)}, \\
& \textstyle \bar\scrM_{\mathit{para}}^{i,\sigma}(y,x^+,x^-) = \bigsqcup_r \bar\scrM_{\mathit{prod}}^{i,\sigma}(y,x^+,x^-)^{(r)}.
\end{align}
The transversality theory for these spaces is a parametrized version of the previous one. In particular, while one cannot expect $\scrM_{\mathit{prod}}^{i,\sigma}(y,x^+,x^-)^{(r)}$ to be regular for all $r$, it is true that if the choices are made generically, $\scrM_{\mathit{para}}^{i,\sigma}(y,x^+,x^-)$ will be a smooth manifold with boundary (the boundary points are precisely the points where $r = 0,1$).

We now specialize to the case relevant to our statement,
\begin{equation} \label{eq:coincide}
y = x^+ = x^-.
\end{equation}
We want to consider the one-dimensional components of $\scrM_{\mathit{para}}^{i,\sigma}(y,x^+,x^-)$, and their closure inside the compactification. The aim is a standard cobordism argument: if the one-dimensional components were themselves compact, their number of boundary points would be even, and hence the expressions $c_x^\sigma$ derived from our two choices ($r = 0$ or $1$) would be the same, since they count those boundary points.

The general structure of a point in $\bar\scrM_{\mathit{para}}^{i,\sigma}(y,x^+,x^-)$ is as follows: there is a principal component, which is a solution of the perturbed version \eqref{eq:perturbed-equation} of \eqref{eq:pants-equation}. The remaining non-principal components are solutions of homogeneous Cauchy-Riemann equations, either \eqref{eq:parametrized-equation} or ordinary Floer trajectories. Because of Setup \ref{th:h-setup}, each of the non-principal components has energy at least $2\epsilon$, unless it is constant. The principal component has energy (in the topological sense, meaning the difference of the actions involved) greater than $-2\epsilon$. However, in our situation \eqref{eq:coincide}, the total sum of those energies is $A_{\phi^2}(x) - 2A_{\phi}(x) = 0$. This shows that any non-principal component is in fact constant.

Take a point of $\bar\scrM_{\mathit{para}}^{i,\sigma}(y,x^+,x^-)$, and consider the stratum \eqref{eq:stratification-2} in which the associated point of $\bar\scrP^{i,\sigma}$ lies. The previously mentioned principal component is a pair $(u_j,w_j)$. The fact that this component exists (given that the moduli spaces are regular in the parametrized sense) means that
\begin{equation} \label{eq:regular-bound}
i_j + \mathrm{index}(D_{u_j}) + 1 \geq 0.
\end{equation}
Here, $D_{u_j}$ is the linearized operator associated to $u_j$ as a perturbed pseudo-holomorphic map; $i_j$ is the dimension of the factor $\scrP^{i_j,\sigma_j}$ in \eqref{eq:stratification-2}; and the last term counts the additional degree of freedom introduced by the parameter. Now suppose that our point of $\bar\scrM_{\mathit{para}}^{i,\sigma}(y,x^+,x^-)$ lies in the closure of a one-dimensional component of $\scrM_{\mathit{para}}^{i,\sigma}(y,x^+,x^-)$. Using the previously mentioned fact that all the non-principal components are constant (hence their linearized operators have index $0$), one gets a dimension constraint
\begin{equation} \label{eq:dim-constraint}
i + \mathrm{ind}(D_{u_j}) + 1 =  1.
\end{equation}
Combining \eqref{eq:dim-constraint} with \eqref{eq:regular-bound} and the fact that $i = i_1 + \cdots + i_d$ in \eqref{eq:stratification-2}, one gets
\begin{equation}
\textstyle \sum_{k \neq j} i_k \leq 1.
\end{equation}
This leaves only two kinds of strata in $\bar\scrP^{i,\sigma}$ which can arise, namely
\begin{align}
& \scrQ^{1,\sigma_1} \times \scrP^{i-1,\sigma_2} \quad \text{and } \label{eq:0diff} \\
& \scrP^{i-1,\sigma_1} \times \scrQ^{1,\sigma_2}, \label{eq:0switch}
\end{align}
where $\sigma_1\sigma_2 = \sigma$. For \eqref{eq:0diff}, the principal component is an isolated point of $\scrM^{i-1,\sigma_2}_{\mathit{para}}(y,x^+,x^-)$. One combines it with a suitable constant non-principal component, and that (for different choices of $\sigma_1)$ yields a point of $\bar\scrM^{i,\sigma_2}_{\mathit{para}}(y,x^+,x^-)$ as well as a point of $\bar\scrM^{i,-\sigma_2}_{\mathit{para}}(\rho(y),x^+,x^-)$ (here, the notation is suggestive of the general picture, but of course in our context \eqref{eq:coincide}, $\rho(y) = y$). Both points in the compactified moduli space produced in this way are regular (which means that they are smooth boundary points of the compactification of one-dimensional components). Similarly, in \eqref{eq:0switch}, the principal component is an isolated point of $\scrM^{i-1,\sigma_1}_{\mathit{para}}(y,x^+,x^-)$; which gives rise to a point in $\bar\scrM^{i,\sigma_1}_{\mathit{para}}(y,x^+,x^-)$, as well as in $\bar\scrM^{i,-\sigma_1}_{\mathit{para}}(y,x^-,x^+)$.

The outcome of this consideration is that, while the one-dimensional part of $\scrM^{i,\sigma}_{\mathit{para}}(x,x,x)$ is not compact, its closure in $\bar\scrM^{i,\sigma}_{\mathit{para}}(x,x,x)$ adds boundary points which appear in pairs, and whose contributions therefore cancel.
\end{proof}

\begin{lemma} \label{th:independence-2}
$c_x$ depends only on the local behaviour of $\phi$ near $x$.
\end{lemma}

\begin{proof}
Define a sequence of moduli spaces $\scrM^{i,\sigma}_{\mathit{prod}}(y,x^+,x^-)^{(k)}$, $k = 1,2,\dots$, where the almost complex structures are independent of $k$, but the inhomogeneous terms are multiplied with $1/k$. This can be done in such a way that all these moduli spaces are regular (since regularity is a generic condition for any given $k$, and countably many such conditions can be imposed at the same time). We also want to define a limiting case $\scrM^{i,\sigma}_{\mathit{prod}}(y,x^+,x^-)^{(\infty)}$, where the inhomogeneous terms are set to zero.

Suppose that we have a sequence of points in the moduli spaces defined above, for $k_1, k_2, \cdots \rightarrow \infty$. Appealing to Gromov compactness, this has a subsequence with a limit in $\bar\scrM^{i,\sigma}_{\mathit{prod}}(x,x,x)^{(\infty)}$. For energy reasons, all components of that limit are constant maps. Hence, if we fix a neighbourhood of $x$, all but finitely many elements of our sequence must have image contained in that neighbourhood. This shows that for fixed $i$ and for sufficiently large $k$, all points of $\scrM^{i,\sigma}_{\mathit{prod}}(x,x,x)^{(k)}$ are given by maps whose image is contained in our fixed neighbourhood. By Lemma \ref{th:independence-1}, we can use that moduli space to compute the coefficient of $h^i$ in $c_x$. This proves the statement (order by order in $h$).
\end{proof}

Note that Lemma \ref{th:independence-2} would be easier to see if we used virtual perturbation techniques, since then, taking the inhomogeneous term to be zero would be a viable choice in itself.

\begin{lemma} \label{th:c-grading}
$c_x$ is a $\bK$-multiple of $h^{n-\kappa(D\phi_x)}$, where $\kappa$ is the Krein index.
\end{lemma}

\begin{proof}
Suppose first that $c_1(M) = 0$, and that $\phi$ is a graded symplectic automorphism \cite{seidel99}. In that case, all Floer complexes are canonically $\bZ$-graded (including the equivariant one, where the formal variable $h$ has degree $1$). More concretely, at any fixed point $x$, the grading determines a preferred lift $\widetilde{D\phi}_x$ of the differential to the universal cover $\widetilde{\mathit{Sp}}(TM_x)$. The degree of the generator corresponding to $x$ is the Conley-Zehnder index $\mu(\widetilde{D\phi}_x)$. In this situation, the map \eqref{eq:chain-product} preserves the grading. More concretely, the dimension formula \eqref{eq:dimension-formula-prod} then holds as an equality in $\bZ$. By combining this with Lemma \ref{th:krein-conley-zehnder}, one sees that
\begin{equation} \label{eq:xxx-dimension}
\mathrm{dim} \, \scrM_{\mathit{prod}}^{i,\sigma}(x,x,x) = \mu(\widetilde{D\phi}_x^2) - 2\mu(\widetilde{D\phi}_x) + i = \kappa(D\phi_x) - n + i.
\end{equation}
Since the only nontrivial contribution to $c_x$ comes from the zero-dimensional spaces $i = n-\kappa(D\phi_x)$, we get the desired result.

In general, even though gradings may not exist globally, they always exist locally near $x$. From the proof of Lemma \ref{th:independence-2}, one sees that $c_x$ can be computed entirely from moduli spaces of maps which remain close to $x$. Those moduli spaces will have the same dimension as in \eqref{eq:xxx-dimension}, so the statement is true in general.
\end{proof}

\begin{lemma} \label{th:independence-3}
$c_x$ depends only on $D\phi_x$.
\end{lemma}

\begin{proof}
Fix a neighbourhood of $x$, and identify it symplectically with a neighbourhood of the origin in the symplectic vector space $H = TM_x$. For any $k = 1,2,\dots$, one can find a Hamiltonian isotopy $(\phi_t^{(k)})$, $t \in [0,1]$, such that the following holds:
\begin{align}
& \phi_0^{(k)} = \phi; \\
& \phi_t^{(k)}(x) = x, \;\; \text{ and } \;\; (D\phi_t^{(k)})_x = D\phi_x; \\
& \text{the isotopy is constant (in $t$) outside a ball of size $1/k$ around $x$;} \\
& \phi_1^{(k)} \text{ is linear near $x$ in our local coordinates;} \label{eq:locally-linear} \\
& \text{as $k \rightarrow \infty$, $\phi_t^{(k)}$ $C^1$-converges to $\phi$, uniformly in $t$.} \label{eq:c1-convergence}
\end{align}
To clarify, in \eqref{eq:locally-linear} we are not saying anything about the size of the neighbourhood in which $\phi_1^{(k)}$ is linear. We omit the details of the construction of the isotopies, which is elementary.

We claim that, as long as $k$ is sufficiently large, the fixed points of $\phi_t^{(k)}$ remain the same for all $t$. By construction, all fixed points of $\phi$ remain fixed points of $\phi_t^{(k)}$, and we only need to worry about new fixed points which may arise. Suppose that (maybe after passing to a subsequence of $k$) we have such new fixed points $x^{(k)}$. Necessarily, these converge to $x$ in the limit $k \rightarrow \infty$. In our local coordinates where $x$ is the origin, the normalized vectors $x^{(k)}/\|x^{(k)}\|$ have a subsequence converging to a unit length vector $\xi \in TM_x$. Because the $x^{(k)}$ as well as the $x$ are fixed points, and \eqref{eq:c1-convergence} holds, it follows that $D\phi_x(\xi_x) = \xi_x$, in contradiction to nondegeneracy. This establishes our claim. Moreover, the action of the fixed points changes under the isotopy only by an amount which goes to zero as $k \rightarrow \infty$. Hence, for $k \gg 0$, one can arrange that \eqref{eq:action-difference} applies to all $\phi_t^{(k)}$, with a bound $\epsilon$ which is independent of $t$. Parallel results hold for $2$-periodic points.

With this in mind, the same argument as in the proof of Lemma \ref{th:independence-1} (but this time varying the symplectomorphism as well) can be used to show that $c_x$ is the same for $\phi$ and for $\phi_1^{(k)}$. An application of Lemma \ref{th:independence-2} concludes the argument, since the local structure of $\phi_1^{(k)}$ near $x$ is completely determined by $D\phi_x$.
\end{proof}

\begin{lemma} \label{th:independence-4}
$c_x$ depends only on the sign of $\mathrm{det}(I - D\phi_x)$ and the Krein index $\kappa(D\phi_x)$.
\end{lemma}

\begin{proof}
Consider a deformation $A_t$ ($0 \leq t \leq 1$) of $A = D\phi_x$ inside the linear symplectic group. One can find a Hamiltonian isotopy $(\phi_t)$ during which $x$ remains a fixed point, such that $\phi_0 = \phi$, and $(D\phi_t)_x = A_t$ for small $t$. It is easy to see that the local contribution $c_x$ for $\phi_t$ remains the same for small $t$: after all, for $t = 0$ we define $c_x$ by counting points in a zero-dimensional compact and regular moduli space $\scrM_{\mathit{prod}}^{i,\pm}(x,x,x)$ (where $i$ is determined by Lemma \ref{th:c-grading}), and a sufficiently small perturbation will not affect the structure of that space. 

Note that we already knew that $c_x$ depends only on $D\phi_x$. We have now shown that it remains constant if we deform $D\phi_x$ slightly. Hence, it is a locally constant function on the open subset \eqref{eq:starstar} of the linear symplectic group. Lemma \ref{th:connected} now yields the desired result.
\end{proof}

Combining Lemmas \ref{th:c-grading} and \ref{th:independence-4}, we can write
\begin{equation} \label{eq:universal-constants}
c_x = h^{n-\kappa(D\phi_x)} c_{s,k},
\end{equation}
where $(s,k) \in \{\pm 1\} \times \bZ$ is the image of $D\phi_x$ under \eqref{eq:invariant}. The $c_{s,k} \in \bK$ are universal constants, depending only on $(s,k)$ and the dimension of the ambient symplectic manifold. We will show the following:

\begin{proposition} \label{th:key}
$c_{s,k} = 1$ for all $(s,k)$.
\end{proposition}

The proof will take up the rest of Section \ref{sec:local}; but before embarking on that task, we want to explain how Proposition \ref{th:key} implies Theorem \ref{th:main}. We will work under the following technical assumption:

\begin{setup} \label{th:setup-4}
Let $\phi$ be as in Setup \ref{th:setup-2}, and with the following additional property. For any fixed points $x^\pm$ of $\phi$, and any fixed point $y$ of $\phi^2$,
\begin{equation} \label{eq:non-coincident}
A_{\phi^2}(y) - A_\phi(x^+) - A_\phi(x^-) \neq 0, \quad
\text{except if $x^- = x^+ = y$.}
\end{equation}
\end{setup}

For applications, one needs to know that this is generically satisfied.

\begin{lemma}
Given any $\phi$ as in Setup \ref{th:setup-2}, there is a small Hamiltonian perturbation, supported in the interior of $M$, so that the perturbed automorphism satisfies \eqref{eq:non-coincident}.
\end{lemma}

\begin{proof}
For any $H \in \scrH_\phi$ (in the notation from Data \ref{th:inhomogeneous-data}), one can consider the perturbed action functional
\begin{equation} \label{eq:tildea-1}
A_{\phi,H}(x) = A_\phi(x) + \textstyle \int_0^1 H_t(x(t)) \, \mathit{dt}.
\end{equation}
This is equivalent to the ordinary action functional $A_{\tilde{\phi}}$ for a suitable Hamiltonian perturbation of $\phi$, determined by $H$ (``equivalent'' means that the two functionals correspond to each other under an identification $\scrL_\phi \iso \scrL_{\tilde{\phi}}$). In the same way, $A_{\tilde{\phi}^2}$ corresponds to
\begin{equation} \label{eq:tildea-2}
A_{\phi^2,H}(y) = A_{\phi^2}(y) + \textstyle \int_0^2 H_t(y(t)) \, \mathit{dt}.
\end{equation}

We will allow only the subspace $\scrH_{\phi}^{\mathit{fixed}} \subset \scrH_\phi$ of those $H$ such that $dH_t$ vanishes at all fixed points of $\phi^2$. This (and nondegeneracy) implies that as long as $H$ is $C^2$-small, the critical points of $A_{\phi^2,H}$ remain the same, which means constant loops at the fixed points of $\phi^2$. The same then holds for $\phi$ as well. To prove the desired result, one has to find a small $H \in \scrH_{\phi}^{\mathit{fixed}}$ such that: 
\begin{itemize} \itemsep.5em
\item $A_{\phi^2,H}(y) - A_{\phi,H}(x^+) - A_{\phi,H}(x^-) \neq 0$ whenever $y$ is a periodic orbit of period exactly two, and $x^+ \neq x^-$ are fixed points;
\item $A_{\phi^2,H}(y) - 2 A_{\phi,H}(x) \neq 0$ whenever $y$ is a periodic orbit of period exactly two, and $x$ is a fixed point (this implies \eqref{eq:non-coincident} for $x^+ = x^- = x$, and $y$ as given);
\item $2A_{\phi,H}(x) - A_{\phi,H}(x^+) - A_{\phi,H}(x^-) \neq 0$ whenever $x,x^+,x^-$ are three different fixed points (this implies \eqref{eq:non-coincident} for $y = x$, and $x^\pm$ as given);
\item Any two different fixed points have different values of $A_{\phi,H}$ (this implies \eqref{eq:non-coincident} for the case where $x^+ \neq x^-$, but $y$ is one of the $x^\pm$; it also takes care of the case where $x^+ = x^-$, and $y$ is a different fixed point of $\phi$).
\end{itemize}

To help formulate the technical argument, let's introduce a linear map
\begin{equation} \label{eq:action-critical}
\scrH_{\phi}^{\mathit{fixed}} \longrightarrow \bR^{p_1+p_2},
\end{equation}
where $p_1$ is the number of fixed points of $\phi$, and $p_2$ the number of periodic orbits of period exactly two (which means, excluding the fixed points). The components of \eqref{eq:action-critical} are: $A_{\phi,H}(x) - A_{\phi}(x)$ at each fixed point $x$; and $A_{\phi^2,H}(y) - A_{\phi^2}(y)$ for a representative $y$ of each two-periodic orbit. Inspection of the formulae \eqref{eq:tildea-1}, \eqref{eq:tildea-2} shows that \eqref{eq:action-critical} is onto. All the desired properties stated above can be formulated as having to avoid the preimage of certain affine submanifolds under \eqref{eq:action-critical}, hence are generic conditions. Note that issues of the functional-analytic nature of $\scrH_{\phi}^{\mathit{fixed}}$ are irrelevant here, since one can replace it by a finite-dimensional subspace such that the restriction of \eqref{eq:action-critical} to that subspace is onto.
\end{proof}

Fix a constant $\epsilon > 0$ which satisfies \eqref{eq:action-difference}, \eqref{eq:action-difference-2}, as well as the following strengthened version of \eqref{eq:2h}: 
\begin{equation}
A_{\phi^2}(y) - A_\phi(x^+) - A_\phi(x^-) \notin (-2\epsilon,2\epsilon), \quad
\text{except if $x^- = x^+ = y$.}
\end{equation}
When constructing $\wp$, choose the inhomogeneous terms to be correspondingly small.

Define a filtration of $\mathit{CF}^*(\phi) \otimes \mathit{CF}^*(\phi)$, so that $F^d$ is generated by expressions $x^+ \otimes x^-$ where $A_\phi(x^+) + A_\phi(x^-) \geq 2\epsilon d$. The condition \eqref{eq:action-difference} implies that the Floer differential maps $F^d$ to $F^{d+1}$. This induces a filtration of the Tate complex $\hat{C}^*(\bZ/2; \mathit{CF}^*(\phi) \otimes \mathit{CF}^*(\phi))$, which is preserved by its differential. In fact, the only part of the Tate differential which does not strictly increase the filtration is that which comes from group cohomology. 

The next part of the argument repeats Addendum \ref{th:low-energy-differential} in a slightly more precise form. Define a filtration of $\mathit{CF}^*(\phi^2)$, so that $F^d$ is generated by those $y$ for which $A_{\phi^2}(y) \geq 2\epsilon d$. Again, the Floer differential strictly increases the filtration, because of \eqref{eq:action-difference-2}. The induced filtration of $\mathit{CF}^*_{\mathit{eq}}(\phi^2)$ is also compatible with the differential. More precisely, the only term in the equivariant differential which does not strictly increase the filtration is $h(\mathit{id} + \rho)$, where $\rho$ is the naive $\bZ/2$-action on $\mathit{CF}^*(\phi^2)$.

Consider the map obtained from $\wp$ after tensoring with $\bK((h))$. We know from Addendum \ref{th:careful-action} that it is compatible with the filtrations on both sides. In fact, because of \eqref{eq:non-coincident}, it follows that all contributions to $\wp$ except the local ones strictly increase the filtration. 

Let's see what the resulting spectral sequence comparison argument yields (as noted before, we are dealing with finite filtrations, hence with the comparison theorem in its most classical form \cite[Theorem 5.2.12]{weibel}). On the $E^0$ page we have the associated graded spaces, and the map between them. Concretely, these are:
\begin{equation} \label{eq:e0-map1}
\mathit{CF}^*(\phi) \otimes \mathit{CF}^*(\phi) \otimes \bK((h)) \longrightarrow
\mathit{CF}^*(\phi^2) \otimes \bK((h)),
\end{equation}
where: the differential on the left hand side is the group cohomology differential for the $\bZ/2$-action exchanging the two factors; the differential on the right hand is the same kind of differential for the naive $\bZ/2$-action on $\mathit{CF}^*(\phi^2)$; and finally, the map \eqref{eq:e0-map1} (assuming Proposition \ref{th:key}) takes
\begin{equation} \label{eq:e0-map}
x \otimes x \longmapsto h^{n-\kappa(D\phi_x)} x,
\end{equation}
and kills the other generators. On the $E^1$ page, we get a map
\begin{equation}
\hat{H}^*(\bZ/2; \mathit{CF}^*(\phi) \otimes \mathit{CF}^*(\phi)) \longrightarrow
\hat{H}^*(\bZ/2; \mathit{CF}^*(\phi^2)).
\end{equation}
As discussed in \eqref{eq:tate-isomorphism}, the left hand side has a basis over $\bK((h))$ represented by $x \otimes x$. As discussed in Addendum \ref{th:low-energy-differential}, the right hand side has a basis represented by $x$, where $x$ is again a fixed point of $\phi$. In particular, it is clear that the two sides are abstractly isomorphic; but what's essential for us is a slightly stronger form of that statement, namely that the map induced by \eqref{eq:e0-map} is an isomorphism. Applying the spectral sequence comparison theorem therefore shows that tensoring $\wp$ with $\bK((h))$ turns it into a quasi-isomorphism. Since tensoring with $\bK((h))$ commutes with passing to cohomology, this is equivalent to the statement of Theorem \ref{th:main}.

\begin{remark} \label{th:alternative-proof}
There is a possible alternative strategy of proof, which would go by constructing a map in inverse direction to \eqref{eq:equi-pants}, such that the two become inverses after tensoring with $\bK((h))$. The putative inverse is not mysterious in itself: it is just a coproduct, constructed dually to \eqref{eq:equi-pants}. The key expectation is that the composition of product and coproduct (in either order) is an ``equivariant quantum cap product'' with the class $\delta \in H^n_{\bZ/2}(M \times M)$ which is Poincar{\'e} dual to the diagonal $\Delta \subset M \times M$. Figures \ref{fig:glue-1} and \ref{fig:glue-2} attempt to give a picture of the degenerations which underlie that expectation (note that both times, they are compatible with a suitable $\bZ/2$-action).

It is well-known that $\delta$ becomes invertible after tensoring with $\bK((h))$. In fact, in view of the localization theorem (Theorem \ref{th:localization}), it is enough to show that the restriction of $\delta$ to $\Delta$ has that property. But that restriction is the equivariant (mod $2$) Euler class of the normal bundle, which is $\sum_i h^{n-i} w_i(TM)$, hence invertible since $w_0(TM) = 1$. This would conclude the argument.

We have not pursued this alternative strategy, because it is less geometric and requires additional moduli spaces and gluing machinery. Nevertheless, there are two potentially attractive aspects to it. One is that it would quantify the failure of \eqref{eq:equi-pants} itself to be an isomorphism (because it depends only on the negative powers of $h$ which appear in $\delta^{-1}$). The second advantage is that a more abstract TQFT-like viewpoint may be better for generalizations beyond the exact case.
\end{remark}
\begin{figure}
\begin{picture}(0,0)%
\includegraphics{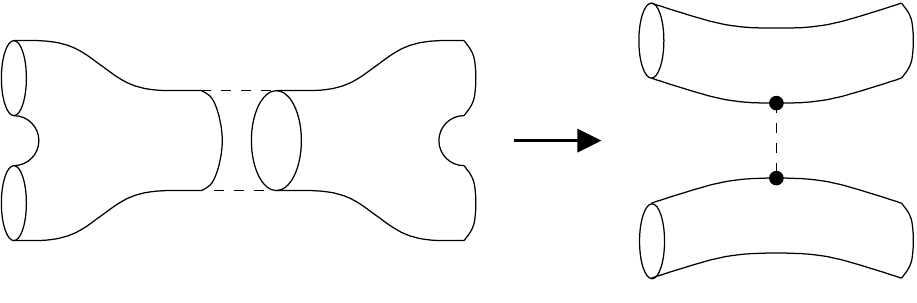}%
\end{picture}%
\setlength{\unitlength}{3158sp}%
\begingroup\makeatletter\ifx\SetFigFont\undefined%
\gdef\SetFigFont#1#2#3#4#5{%
  \reset@font\fontsize{#1}{#2pt}%
  \fontfamily{#3}\fontseries{#4}\fontshape{#5}%
  \selectfont}%
\fi\endgroup%
\begin{picture}(5492,1806)(1043,-1030)
\put(1351,-961){\makebox(0,0)[lb]{\smash{{\SetFigFont{10}{12.0}{\rmdefault}{\mddefault}{\updefault}{coproduct}%
}}}}
\put(2851,-961){\makebox(0,0)[lb]{\smash{{\SetFigFont{10}{12.0}{\rmdefault}{\mddefault}{\updefault}{product}%
}}}}
\end{picture}%
\caption{\label{fig:glue-1}}
\end{figure}%
\begin{figure}
\begin{picture}(0,0)%
\includegraphics{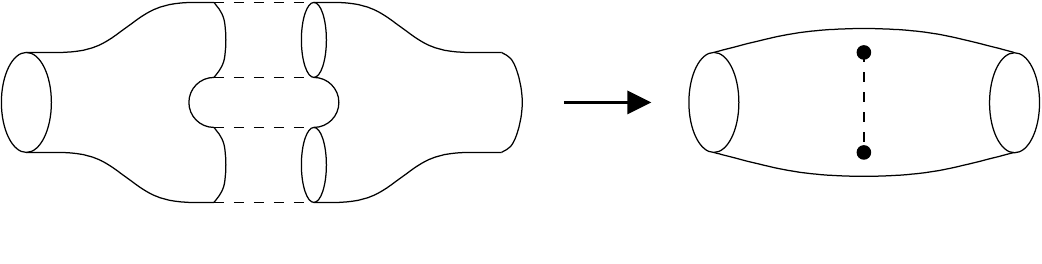}%
\end{picture}%
\setlength{\unitlength}{3158sp}%
\begingroup\makeatletter\ifx\SetFigFont\undefined%
\gdef\SetFigFont#1#2#3#4#5{%
  \reset@font\fontsize{#1}{#2pt}%
  \fontfamily{#3}\fontseries{#4}\fontshape{#5}%
  \selectfont}%
\fi\endgroup%
\begin{picture}(6245,1581)(743,-2680)
\put(2851,-2611){\makebox(0,0)[lb]{\smash{{\SetFigFont{10}{12.0}{\rmdefault}{\mddefault}{\updefault}{coproduct}%
}}}}
\put(1051,-2611){\makebox(0,0)[lb]{\smash{{\SetFigFont{10}{12.0}{\rmdefault}{\mddefault}{\updefault}{product}%
}}}}
\end{picture}%
\caption{\label{fig:glue-2}}
\end{figure}%

Having said that, we now begin with the actual proof of Proposition \ref{th:key}. There is one special case which is particularly simple, since it does not involve equivariance at all:

\begin{lemma} \label{th:cheese}
$c_{+1,n} = 1$.
\end{lemma}

\begin{proof}
In this case, we are looking at the $h^0$ coefficient of $\wp(x,x,x)$, which means the ordinary pair-of-pants product. Suppose temporarily that the inhomogeneous term is taken to be zero. In that case, $\scrM^{0,+}_{\mathit{prod}}(x,x,x) = \bar{\scrM}^{0,+}_{\mathit{prod}}(x,x,x)$ consists of a single point, the constant map $S \rightarrow M$ at $x$. The linearization of the pseudo-holomorphic curve equation at that point is one of the operators \eqref{eq:d-bar-a}. This has index zero by Lemma \ref{th:index-theory}, and is injective by Lemma \ref{th:no-kernel}, hence a regular point. Hence, for any small perturbation of this setup (introduced by choosing an inhomogeneous term), it will still be true that $\scrM^{0,+}_{\mathit{prod}}(x,x,x)$ consists of a single regular point.
\end{proof} 

In principle, it should be possible to determine each $c_{s,k}$ by itself, let's say by starting with the degenerate case in which the inhomogeneous term is zero, and applying a suitable obstruction theory. However, it is clear that these numbers for different $(s,k)$ are not really independent: the fact that $\wp$ is a chain map implies relations between them. We will use those relations to derive the rest of Proposition \ref{th:key} from Lemma \ref{th:cheese}. %(yet another possible approach would be to establish a relation between our equivariant product and the standard one, in the situation of Example \ref{th:?}, and then use that relation).

\subsection{Two Morse-theoretic examples}
The following considerations are local, which means that they should be thought of as taking place in a Darboux chart inside some Liouville domain. We consider only the part of Floer theory that takes place inside that chart. This is a ``local Floer cohomology'' argument, which makes sense because the energies involved can be made arbitrary small. In particular, because of the local nature of the argument, we can assume that Floer cohomology and its product structure are $\bZ$-graded (as in Lemma \ref{th:c-grading}). Of course, local Floer cohomology is convenient, but not really essential here: one could specify exactly what the ambient Liouville domain should be, and how our symplectic automorphism behaves away from the local chart (and then show that this is irrelevant for the actual computation).

\begin{remark}
In fact, in the two examples below, we consider situations which can be obtained by perturbing a single degenerate fixed point, which is local Floer cohomology in the most commonly used sense (see e.g.\ \cite[Section 3]{ginzburg-gurel10}). This relies on Gromov compactness arguments similar to those in Lemma \ref{th:independence-2}. 
%Strictly speaking, these can be applied only to a finite number of terms (in $h$) of equivariant Floer differential and equivariant pair-of-pants product, but that is entirely sufficient for our purpose. 
%
Subsequently (Section \ref{subsec:final-model}), we will consider an example of a slightly more complicated nature. To prove that local Floer cohomology can be defined in that context, one combines the Gromov compactness arguments with a priori bounds (such as \cite[Lemma 4.3.1]{mcduff-salamon}, but with varying almost complex structure).
\end{remark}

Let $H$ be a Morse function with exactly two critical points (in our local chart) $x,y$, of index \begin{equation} \label{eq:morse-index-values}
i(x) = i-1, \;\; i(y) = i,
\end{equation}
for some $1 \leq i \leq 2n$. We suppose that these two annihilate each other under a (local) deformation of the Morse function, which implies that the Morse differential (or rather, its local part) sends $x$ to $y$. Obviously, in this situation 
\begin{equation} \label{eq:morse-function-values}
H(x) < H(y).
\end{equation}

\begin{remark}
Since this language recurs later on, it may be worth spelling out what we mean by it. We start with a function $H_0$ which has a degenerate critical point of class $(A_2)$, and form $H = H_c$ by a perturbation depending on a small parameter $c>0$, which yields a pair of nondegenerate critical points (by the generic birth-death process in one-parameter families of Morse functions \cite{cerf}). A local picture of such a perturbation is
\begin{align}
& H_0(\xi_1,\cdots,\xi_n) = \xi_1^3/3 - \xi_2^2 - \cdots - \xi_i^2 + \xi_{i+1}^2 + \cdots + \xi_n^2, \\
& H_c(\xi_1,\dots,\xi_n) = H_0(\xi_1,\dots,\xi_n) - c \xi_1.
\end{align}
In such local coordinates, $x = (c^{1/2},0,\dots,0)$ and $y = (-c^{1/2},0,\dots,0)$, and then \eqref{eq:morse-index-values} and \eqref{eq:morse-function-values} are obvious. If the metric is standard in our local coordinates, one can explicitly write down the Morse trajectory connecting $x$ to $y$. For a general metric, the simplest argument may be an indirect one: the Morse homology of $H_c$ and $H_{-c}$ are the same, and the same is true for the local contributions to it (near the degenerate critical point). However, for $H_{-c}$ this local contribution is zero since the critical points have disappeared. The advantage of this indirect argument is that it also applies to Floer theory (without requiring a reduction to Morse theory). Of course, other approaches are also possible: for instance, a direct study of the behaviour of Floer complexes under birth-death of generators, as in \cite{lee}.
\end{remark}

Let $(\phi_t)$ be the Hamiltonian flow of $H$. We consider $\phi = \phi_t$ for small $t>0$, and its square $\phi^2 = \phi_{2t}$. Both $\phi$ and $\phi^2$ have only $x$ and $y$ as fixed points (in our local chart), and
\begin{equation} \label{eq:action-values}
A_\phi(x) = t H(x), \;\; A_\phi(y) = t H(y), \;\; A_{\phi^2}(x) = 2 t H(x), \;\; A_{\phi^2}(y) = 2 t H(y).
\end{equation}
The associated Floer cochain complexes (or rather, their local parts; we will now stop putting in that proviso) are
\begin{align}
& \mathit{CF}^*(\phi) = \mathit{CF}^*(\phi^2) = \bK x \oplus \bK y, \quad |x| = i-1, \; |y| = i, \\
& d_{J_\phi}(x) = d_{J_{\phi^2}}(x) = y. \label{eq:morse-floer}
\end{align}
To determine \eqref{eq:morse-floer}, one can use the general relation between Morse complex and Floer complex, which holds for a specific class of almost complex structures \cite{hofer-salamon95}; or alternatively, appeal to the isotopy invariance of Floer cohomology, and the fact that the two fixed points are known to kill each other under such an isotopy. From Addendum \ref{th:low-energy-differential}, one sees that the equivariant Floer differential strictly increases the action. By combining this with the $\bZ$-grading (which exists for the same reason as in Lemma \ref{th:c-grading}), one sees that there are no higher order contributions in $h$:
\begin{equation} \label{eq:d-on-the-right}
d_{\mathit{eq}} = d_{J_{\phi^2}}. 
\end{equation}
The differential on $C^*(\bZ/2; \mathit{CF}^*(\phi) \otimes \mathit{CF}^*(\phi))$ is
\begin{equation} \label{eq:d-on-the-left}
\left\{
\begin{aligned}
& x \otimes x \longmapsto y \otimes x + x \otimes y, \\
& x \otimes y, \, y \otimes x \longmapsto y \otimes y + h(x \otimes y + y \otimes x), \\
& y \otimes y \longmapsto 0.
\end{aligned}
\right.
\end{equation}

One can arrange that the equivariant pair-of-pants product \eqref{eq:chain-product} does not decrease the action (Addendum \ref{th:careful-action}). With this and the $\bZ$-grading in mind, it is necessarily of the form
\begin{equation} \label{eq:wp-coefficients}
\left\{
\begin{aligned}
& \wp(x \otimes x) = c_{(-1)^{i-1}, n-i+1} h^{i-1} x + b_{xx} h^{i-2} y, \\
& \wp(x \otimes y) = b_{xy} h^{i-1} y, \\
& \wp(y \otimes x) = b_{yx} h^{i-1} y, \\
& \wp(y \otimes y) = c_{(-1)^i, n-i} h^i y,
\end{aligned}
\right.
\end{equation}
where the $c$'s are local contributions (the relevant Krein indices are computed in Example \ref{th:epsilon-example} or Lemma \ref{th:krein-conley-zehnder}), and the $b$'s a priori unknown coefficients in $\bK$. The fact that $\wp$ is a chain map yields 
\begin{equation}
c_{(-1)^{i-1}, n-i+1} = b_{xy} + b_{yx} = c_{(-1)^i,n-i}.
\end{equation}
For later reference, we summarize the outcome with slightly different notation:

\begin{lemma} \label{th:constant-1}
For any $-n \leq k \leq n-1$ and $s = (-1)^{k+n}$, we have $c_{s,k} = c_{-s,k+1}$. 
\end{lemma}

Let's consider a twisted version of the previous situation. Namely, suppose that our Hamiltonian has the form $H(p,q) = p_1q_1 + (\text{\it function in the other $2n-2$ variables})$, We want it to have critical points $x,y$ as before, which now obviously must lie in $\{p_1 = q_1 = 0\}$, and have Morse index \eqref{eq:morse-index-values} with $2 \leq i \leq 2n-1$. Take $\phi = \rho \phi_t$, where $\phi_t$ is the Hamiltonian flow, and $\rho$ is the involution which reverses $(p_1,q_1)$. The fixed points of $\phi$ are still just $x$ and $y$, and the same is true for $\phi^2 = \phi_{2t}$. The action values are as in \eqref{eq:action-values}, since they can be computed entirely inside the locus $\{p_1 = q_1 = 0\}$. However, the degrees of the generators now come out slightly differently: if we connect the identity to $\rho$ by a $\pi$ rotation in the $(p_1,q_1)$-plane, and use that to equip $\phi$ with the structure of a graded symplectic isomorphism, then
\begin{align}
& \mathit{CF}^*(\phi) = \bK x \oplus \bK y, \quad |x| = i-2, |y| = i-1, \\
& \mathit{CF}^*(\phi^2) = \bK x \oplus \bK y, \quad |x| = i-3, \; |y| = i-2.
\end{align}
The differentials on these groups are as before. The same applies to the equivariant differential \eqref{eq:d-on-the-right} and to \eqref{eq:d-on-the-left}. The same computation as before, together with Example \ref{th:rotated-epsilon-example}, shows the following:

\begin{lemma} \label{th:constant-2}
For any $1-n \leq k \leq n-2$ and $s = (-1)^{k+n+1}$, we have $c_{s,k} = c_{-s,k+1}$. 
\end{lemma}

Together, Lemmas \ref{th:constant-1} and \ref{th:constant-2} show that within the allowed set of values \eqref{eq:conditions-kappa}, $c_{s,k}$ remains the same if we change $k$ by $\pm 1$ and simultaneously reverse $s$.
%
% Krein index is n-Morse index. For y, this ranges from n-(2n-1) = (1-n) to n-2.
% s is the sign of y for \phi, which is i-1. k is the Krein index of y, which is n-i

\subsection{An example with nontrivial periodic points\label{subsec:final-model}}
We consider another local model, this time starting in two dimensions, for the sake of concreteness. Take a disc $U$, divided into an inner disc $U_{\mathit{in}}$, a middle annulus $U_{\mathit{mid}}$ surrounding it, and another outer annulus $U_{\mathit{out}}$ around that; see Figure \ref{fig:function1}. Consider Morse functions $H_{\mathit{in}}$, $H_{\mathit{out}}$ defined in the respective regions; Figure \ref{fig:function1} shows their level sets as well as the direction in which the associated Hamiltonian vector fields go. Importantly for our purpose, $U_{\mathit{in}}$ should admit an involution (rotation by $\pi$ around $x$ in Figure \ref{fig:function1}) which leaves $H_{\mathit{in}}$ unchanged. Define a symplectic automorphism $\phi$ as follows: on $U_{\mathit{out}}$, it is the flow of $H_{\mathit{out}}$ for small positive time; on $U_{\mathit{in}}$, it is the flow of $H_{\mathit{in}}$ for small positive time, composed with rotation by $\pi$; and in $U_{\mathit{mid}}$, we interpolate between the two, by a right-handed half Dehn twist (this means that, as one enters $U_{\mathit{mid}}$ from the outside, $\phi$ starts moving around the annulus by increasing amounts in anticlockwise direction).

\begin{figure}
\begin{picture}(0,0)%
\includegraphics{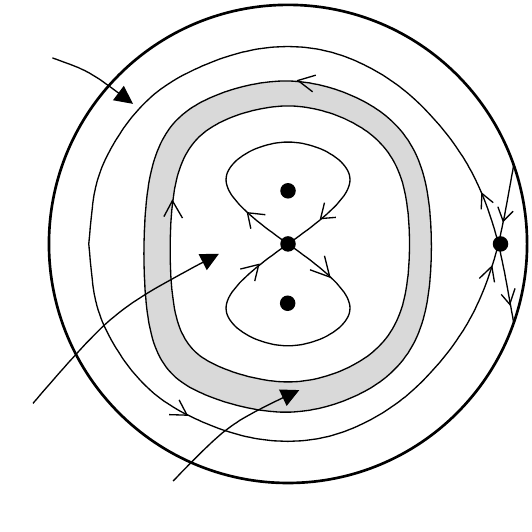}%
\end{picture}%
\setlength{\unitlength}{3355sp}%
\begingroup\makeatletter\ifx\SetFigFont\undefined%
\gdef\SetFigFont#1#2#3#4#5{%
  \reset@font\fontsize{#1}{#2pt}%
  \fontfamily{#3}\fontseries{#4}\fontshape{#5}%
  \selectfont}%
\fi\endgroup%
\begin{picture}(2992,2986)(924,-2433)
\put(2701,-863){\makebox(0,0)[lb]{\smash{{\SetFigFont{10}{12.0}{\rmdefault}{\mddefault}{\updefault}{$x$}%
}}}}
\put(1009,-1906){\makebox(0,0)[lb]{\smash{{\SetFigFont{10}{12.0}{\rmdefault}{\mddefault}{\updefault}{$U_{\mathit{in}}$}%
}}}}
\put(1786,-2360){\makebox(0,0)[lb]{\smash{{\SetFigFont{10}{12.0}{\rmdefault}{\mddefault}{\updefault}{$U_{\mathit{mid}}$}%
}}}}
\put(939,347){\makebox(0,0)[lb]{\smash{{\SetFigFont{10}{12.0}{\rmdefault}{\mddefault}{\updefault}{$U_{\mathit{out}}$}%
}}}}
\put(2626,-1186){\makebox(0,0)[lb]{\smash{{\SetFigFont{10}{12.0}{\rmdefault}{\mddefault}{\updefault}{$z_1$}%
}}}}
\put(3526,-886){\makebox(0,0)[lb]{\smash{{\SetFigFont{10}{12.0}{\rmdefault}{\mddefault}{\updefault}{$y$}%
}}}}
\put(2551,-417){\makebox(0,0)[lb]{\smash{{\SetFigFont{10}{12.0}{\rmdefault}{\mddefault}{\updefault}{$z_0$}%
}}}}
\end{picture}%
\caption{\label{fig:function1}}
\end{figure}%
We can lift $\phi$ to a graded symplectic automorphism, and such a lift is uniquely specified by the following requirement: in a neighbourhood of $\partial U$, the grading agrees with what one would get from deforming the trivial grading of the identity map (bearing in mind that $\phi$ is a small deformation of the identity near $\partial U$). Then, the generators of $\mathit{CF}^*(\phi)$ corresponding to the two fixed points $x$ and $y$ satisfy
\begin{equation}
|x| = 0, \quad |y| = 1.
\end{equation}
One can deform $\phi$ to remove all fixed points, without changing the behaviour near $\partial U$. Hence, the Floer complex must be acyclic (alternatively, one one can deform $\phi$ to be close to the identity without changing the fixed points, and then argue by comparison with Morse theory). Hence,
$d_{J_\phi}(x) = y$ and $A_{\phi}(x) < A_{\phi}(y)$.

The square $\phi^2$ admits the following simpler description: on $U_{\mathit{in}}$ and $U_{\mathit{out}}$, it is the flow of the respective functions for small positive times; and in $U_{\mathit{mid}}$, we interpolate between them by a right-handed Dehn twist. In particular, the grading inside $U_{\mathit{in}}$ is close to that of the upwards shift by $2$, hence the degrees of the relevant generators of $\mathit{CF}^*(\phi^2)$ are lower by $2$ than the Morse indices. Concretely, there are four fixed points $x, \, z_0, \, z_1,\, y$ with
\begin{equation}
|x| = -1, \quad |z_0| = |z_1| = 0, \quad |y| = 1,
\end{equation}
and they satisfy
\begin{equation}
d_{J_{\phi^2}}(x) = z_0 + z_1, \quad d_{J_{\phi^2}}(z_0) = d_{J_{\phi^2}}(z_1) = y.
\end{equation}
The computation of the differential uses two arguments: as before, the Floer cohomology must be zero; and for the generators coming from $U_{\mathit{in}}$, one can appeal to a comparison with Morse theory. Note that in particular, 
\begin{equation} \label{eq:action-ordering}
A_{\phi^2}(x) < A_{\phi^2}(z_0) = A_{\phi^2}(z_1) < A_{\phi^2}(y). 
\end{equation}
A degree and action argument then shows that the only nontrivial additional contribution to the equivariant differential is
\begin{equation}
d_{\mathit{eq}}^1(z_0) = d_{\mathit{eq}}^1(z_1) = u(z_0 + z_1).
\end{equation}

In parallel with \eqref{eq:wp-coefficients}, one can write
\begin{equation} \label{eq:wo-coefficients-2}
\left\{
\begin{aligned}
& \wp(x \otimes x) = c_{+1,0}\, hx + b_{xx0} \, z_0 + b_{xx1} \, z_1, \\
& \wp(x \otimes y) = b_{xy} \, y , \\
& \wp(y \otimes x) = b_{yx} \, y,  \\
& \wp(y \otimes y) = c_{-1,0}\, hy.
\end{aligned}
\right.
\end{equation}
The Krein indices can be computed from Examples \ref{th:rotated-epsilon-example} and \ref{th:epsilon-example} (or alternatively from Lemma \ref{th:krein-conley-zehnder}). The absence of $hz_k$ terms in $\wp(x \otimes y)$ and $\wp(y \otimes x)$ is established by an action argument, which refines \eqref{eq:action-ordering}: by a suitable choice of details, one can make sure that $A_{\phi^2}(z_k)$ is much closer to $A_{\phi^2}(x)$ than to $A_{\phi^2}(y)$, in which case $A_{\phi}(x) + A_{\phi}(y) = A_{\phi^2}(x) + \half(A_{\phi^2}(y) - A_{\phi^2}(x)) > A_{\phi^2}(z_k)$. Then, the fact that $\wp$ is a chain map yields the relations
\begin{equation}
c_{+1,0} = b_{xx0} + b_{xx1} = b_{xy} + b_{yx} = c_{-1,0}.
\end{equation}
Even though we have considered a two-dimensional situation only, the same applies in $2n$ dimensions as well, by taking the product with a Hamiltonian flow in the remaining $2n-2$ variables, whose underlying function has a unique critical point. By choosing that critical point to have all possible Morse indices, one gets:

\begin{lemma} \label{th:constant-3}
For $1-n \leq k \leq n-1$, we have $c_{+1,k} = c_{-1,k}$.
\end{lemma}

Clearly, Lemmas \ref{th:cheese}, \ref{th:constant-1}, \ref{th:constant-2} and \ref{th:constant-3} together imply Proposition \ref{th:key}.
\section{Beyond the exact case\label{sec:monotone}}

The exactness assumption has been used in the body of the paper in several different ways. This has technical advantages, since it rules out holomorphic sphere bubbles; but there are other situations where bubbling can be dealt with easily (the monotone case, for instance). There are much more important conceptual issues, which arise already at the point of defining equivariant Floer cohomology. These are roughly similar to, but not quite the same as, those encountered in \cite{jones} for classical equivariant homology, or in \cite{zhao14, albers-cieliebak-frauenfelder14} (see Section \ref{subsec:symplectic}) for $S^1$-equivariant symplectic cohomology. The aim of this section is to give a short and rather sketchy introduction to these questions, in the monotone case (note that the negatively monotone case seems much less interesting).

\subsection{Definition}
Take a closed symplectic manifold $M$ with $[\omega_M] = 2c_1(M)$ and $H^1(M) = 0$, and a symplectic automorphism $\phi$ with nondegenerate fixed points. Given a solution $u$ of \eqref{eq:floer} with limits $(y,x)$, both the energy $E(u)$ and the index of the linearized operator $D_u$ can depend on $u$, but their difference only depends on the limits. In fact, one can associate to each fixed point $x$ a normalized action $\bar{A}_\phi(x) \in \bR$, in such a way that for $u$ as before,
\begin{equation} \label{eq:monotone-energy}
E(u) - \mathrm{ind}(D_u) = \bar{A}_\phi(y) - \bar{A}_\phi(x).
\end{equation}
For those $u$ that contribute to the Floer differential $d_{\phi}$, $\mathrm{ind}(D_u) = 1$, which provides an a priori energy bound. Bubbling off of holomorphic spheres reduces the energy of the remaining part by at least $2$, hence is a codimension $2$ phenomenon (this is just a sketch of the classical construction of $\mathit{HF}^*(\phi)$, see \cite{floer88, dostoglou-salamon93}).
% differential goes from x to y and has index 1, then LHS >= -1

Let's pass to $\phi^2$, again assuming that its fixed points are nondegenerate. One can define $\mathit{HF}^*_{\mathit{eq}}(\phi^2)$ as in the exact case, as the cohomology of $\mathit{CF}^*(\phi^2)[[h]]$ with the equivariant differential. From the long exact sequence \eqref{eq:floer-u-sequence}, together with the fact that $\mathit{HF}^*_{\mathit{eq}}(\phi^2)$ is a finitely generated $\bK[[h]]$-module, one derives \eqref{eq:smith-2}. In particular, if $\mathit{HF}^*(\phi^2)$ vanishes, the same holds for $\mathit{HF}^*_{\mathit{eq}}(\phi^2)$.

For our next observation, we have to dig a bit deeper into the details. For those $[w,u] \in \scrM_{\mathit{eq}}^{i,\sigma}(y,x)$ which contribute to $d^{i,\sigma}_{\mathit{eq}}$, we have $\mathrm{ind}(D_u) = 1-i$. By \eqref{eq:monotone-energy}, $E(u)$ becomes negative if $i$ is large, hence
\begin{equation} \label{eq:polynomiality}
d_{\mathit{eq}}^{i,\sigma} = 0 \quad \text{for $i \gg 0$.}
\end{equation}
Therefore, the equivariant differential preserves the subspace $\mathit{CF}^*_{\mathit{poly}}(\phi^2) = \mathit{CF}^*(\phi^2)[h]$. We denote the resulting cohomology by $\mathit{HF}^*_{\mathit{poly}}(\phi^2)$. This polynomial version of equivariant cohomology is a finitely generated $\bZ/2$-graded $\bK[h]$-module. It is related to the previous one by 
\begin{equation} \label{eq:completion}
\mathit{HF}^*_{\mathit{eq}}(\phi^2) \iso \mathit{HF}^*_{\mathit{poly}}(\phi^2) \otimes_{\bK[h]} \bK[[h]].
\end{equation}

\subsection{Basic properties}
The polynomial version is much more delicate to handle, because the $h$-adic filtration on the underlying complex is no longer complete. Some basic properties can nevertheless be established easily. It fits into the usual kind of long exact sequence \eqref{eq:floer-u-sequence}, but the implications are weaker in this case. In particular, if $\mathit{HF}^*(\phi^2)$ vanishes, it only follows that $h$ must act invertibly on $\mathit{HF}^*_{\mathit{poly}}(\phi^2)$, which means that $0$ can't be an eigenvalue. 

\begin{lemma} \label{th:perturb-bound}
If $\phi$ is fixed point free, both $\mathit{HF}^*_{\mathit{poly}}(\phi^2)$ and $\mathit{HF}^*_{\mathit{eq}}(\phi^2)$ are finite-dimensional over $\bK$; in fact, their dimension is bounded above by the number of two-periodic orbits of $\phi$.
\end{lemma}

In the exact case, such bounds follow from the action filtration spectral sequence (Addendum \ref{th:low-energy-differential}). The argument below uses instead normalized actions, and an algebraic framework which is slightly more explicit than spectral sequences.

\begin{proof}
In $\mathit{CF}^*(\phi^2)[h]$, assign to a generator $x h^j$ the normalized action
\begin{equation} \label{eq:normalized-action-2}
\bar{A}_{\phi}(x h^j) = \bar{A}_\phi(x) - j.
\end{equation}
The maps that contribute to the $h^i$ term of $d_{\mathit{eq}}$ have $\mathrm{ind}(D_u) = 1-i$.  From \eqref{eq:monotone-energy}, one therefore sees that $d_{\mathit{eq}}$ decreases \eqref{eq:normalized-action-2} by at most $1$. Moreover, if one subtracts the zero energy part $\delta = h(\mathit{id} + \rho)$, then $d_{\mathit{eq}} - \delta$ decreases normalized actions by strictly less than $1$.

Divide the fixed points of $\phi^2$ into two subsets exchanged by $\phi$ (this is possible since $\phi$ itself is fixed point free). Elements of those two subsets will be denoted by $x_+$ and $\rho(x_+)$, respectively. Denote by $D^*$ the $\bZ/2$-graded $\bK$-vector space generated by the $x_+$. Consider the maps
\begin{align}
&
i: D^* \longrightarrow \mathit{CF}^*(\phi^2)[h], && i(x_+) = x_+ + \rho(x_+), 
\\
&
p: \mathit{CF}^*(\phi^2)[h] \longrightarrow D^*, &&
\left\{
\begin{aligned}
& p(x_+) = x_+, \\
& p(\rho(x_+)) = 0, \\
& p(x h^j) = 0 \text{ if $j>0$,}
\end{aligned}
\right.
\\
& k: \mathit{CF}^*(\phi^2)[h] \longrightarrow \mathit{CF}^{*-1}(\phi^2)[h],
&&
\left\{\begin{aligned}
& k(x_+) = 0, \\ 
& k(x_+ h^j) = h^{j-1} \rho(x_+) \text{ if $j>0$,} \\
& k(\rho(x_+) h^j) = 0 \text{ for all $j$.} \\
\end{aligned}
\right.
\\
\intertext{which satisfy}
& p \circ i = \mathit{id}, \\
& p \circ k = 0, \\
& k \circ i  = 0, \\
& k \circ k = 0, \\
& \delta \circ i = 0, \\
& p \circ \delta= 0, \\
& i \circ p = \mathit{id} + \delta \circ k + k \circ \delta. 
\end{align}

Since $k \circ (d_{\mathit{eq}} - \delta)$ strictly increases normalized actions, it must be a locally nilpotent endomorphism (which means that any element of $\mathit{CF}^*(\phi^2)[h]$ is annihilated by some power of it). With that in mind, one can define a differential on $D^*$ by the formula
\begin{multline}
d_D = p \circ \big((d_{\mathit{eq}} - \delta) + (d_{\mathit{eq}} - \delta) \circ k \circ (d_{\mathit{eq}} - \delta) \\ + (d_{\mathit{eq}} - \delta) \circ k \circ (d_{\mathit{eq}} - \delta) \circ k \circ (d_{\mathit{eq}} - \delta) + \cdots \big) \circ i.
\end{multline}
This is part of a standard ``transfer'' or ``perturbation'' formalism \cite{huebschmann-kadeishvili91, markl01}: similar formulae define chain maps between $D^*$ and $\mathit{CF}^*(\phi^2)$, which are chain homotopy equivalences \cite[Lemma 1.1]{huebschmann-kadeishvili91}. Hence, 
\begin{equation}
H^*(D^*,d_D) \iso \mathit{HF}^*_{\mathit{poly}}(\phi^2),
\end{equation}
which in view of the definition of $D^*$ implies the desired bound. The corresponding result for $\mathit{HF}^*_{\mathit{eq}}(\phi^2)$ then follows from \eqref{eq:completion}. 
\end{proof}

\subsection{A Lagrangian intersection analogue}
Given the previous remarks, it is an obvious question whether there are concrete examples in which $\mathit{HF}^*_{\mathit{poly}}(\phi^2)$ gives a better bound on two-periodic points than $\mathit{HF}^*_{\mathit{eq}}(\phi^2)$ (or ordinary Floer cohomology). We can't answer this, but we can show an instance of parallel behaviour for Lagrangian intersection Floer cohomology.

Namely, inside $M = \bC^2$, take $L_0 = \bR^2$ and $L_1 = S^1 \times S^1$ (the Clifford torus). We use the standard symplectic form, rescaled so that the unit disc has area $2$. This is chosen for compatibility with our previous monotonicity considerations. One can then associate to points $x \in L_0 \cap L_1$ normalized actions $\bar{A}_{L_0,L_1}(x)$, so that the analogue of \eqref{eq:monotone-energy} for pseudo-holomorphic strips holds. Specifically in our example, we have $L_0 \cap L_1 = \{(\pm 1, \pm 1)\}$; denote its four points by $x_{\pm,\pm}$. They all have the same normalized action. The differential on $\mathit{CF}^*(L_0,L_1)$ squares to zero because the disc-counting obstructions for both $L_0$ and $L_1$ \cite{oh-appendix} vanish mod $2$.  The Floer complex must be acyclic, because $L_1$ can be displaced from $L_0$ by a translation. Using the standard complex structure (which turns out to be regular), one determines it explicitly:
\begin{equation}
\begin{aligned}
& d_{L_0,L_1}(x_{--}) = d_{L_0,L_1}(x_{++}) = x_{-+} + x_{+-}, \\ 
& d_{L_0,L_1}(x_{-+}) = d_{L_0,L_1}(x_{+-}) = x_{--} + x_{++}.
\end{aligned}
\end{equation}

Now let $\bZ/2$ act on $M$ by $\iota(z_1,z_2) = (-z_1,-z_2)$. One can define an equivariant Floer differential $d_{\mathit{eq}}$ for the pair $(L_0,L_1)$ by a formalism parallel to that in Section \ref{subsec:equi-define}, see \cite{seidel-smith10}. In fact, the analogue of \eqref{eq:monotone-energy} shows that the only $u$ that can contribute to the equivariant differential are the constant (energy zero) ones. Hence, it is straightforward to determine
\begin{equation}
\begin{aligned}
& d_{\mathit{eq}}(x_{--}) = d_{\mathit{eq}}(x_{++}) = x_{-+} + x_{+-} + h(x_{--} + x_{++}), \\
& d_{\mathit{eq}}(x_{-+}) = d_{\mathit{eq}}(x_{+-}) = x_{--} + x_{++} + h(x_{-+} + x_{+-}).
\end{aligned}
\end{equation}
If we define equivariant Floer cohomology in the standard way, using $\mathit{CF}^*(L_0,L_1)[[h]]$, the resulting group $\mathit{HF}^*_{\mathit{eq}}(L_0,L_1)$ is zero (as must be the case for general reasons). However, for the polynomial version based on $\mathit{CF}^*(L_0,L_1)[h]$, one has
\begin{equation} \label{eq:poly-lagrangian}
\mathit{HF}^*_{\mathit{poly}}(L_0,L_1) \iso \bK[h]/(h^2+1).
\end{equation}
This saturates the bound given by the analogue of Lemma \ref{th:perturb-bound} (the dimension of \eqref{eq:poly-lagrangian} over $\bK$ equals the number of orbits of the free $\bZ/2$-action on $L_0 \cap L_1$). Of course, to obtain a geometric conclusion about equivariant non-displaceability, one would have to show the invariance of $\mathit{HF}^*_{\mathit{poly}}(L_0,L_1)$ under equivariant isotopies, which we have not done.
% kernel = (x_-- + x_++) and (x_{-+} + x_{+-}. h times one generator is identified with the other
% one.

%
% u connects the two: has energy 1, index 1, hence A(.) = 1.

\end{document}